\documentclass{amsart}

\usepackage{amsthm}

\usepackage[top=1in, left=1in, right=1in, bottom=1in]{geometry}
\usepackage{
amssymb,amsmath,enumerate,stackrel,tabu,mathrsfs,enumitem,verbatim,stmaryrd}
\usepackage{xcolor} 
\usepackage{xr-hyper}
\usepackage{nicefrac,stackrel}
\usepackage[pagebackref]{hyperref}
\usepackage[all,cmtip]{xy}
\usepackage[utf8]{inputenc} 
\chardef\cprime"7E 
\usepackage{tikz}

\usepackage{xcolor}
\hypersetup{
    colorlinks,
    linkcolor=red,
    citecolor={blue!50!black},
    urlcolor={blue!80!black}
}

\def\label#1{\label{#1}}

\definecolor{labelkey}{rgb}{1,0,0}

\setlist[enumerate,1]{label=\roman*\textup{)},
style=nextline, leftmargin=7mm,
topsep=0mm,
itemsep=0.4mm, 
align=parleft
}
\makeatletter
\@addtoreset{equation}{section}
\makeatother

\numberwithin{equation}{subsection}

\theoremstyle{definition}
\newtheorem{Defi}[equation]{Definition} \newcommand{\defi}{\begin{Defi}} \newcommand{\xdefi}{\end{Defi}} \newcommand{\refde}[1]{Definition~\ref{defi--#1}}
\newtheorem{DefiLemm}[equation]{Definition and Lemma} \newcommand{\defilemm}{\begin{DefiLemm}} \newcommand{\xdefilemm}{\end{DefiLemm}} \newcommand{\refdele}[1]{Definition and Lemma~\ref{defilemm--#1}}
\newtheorem{Bsp}[equation]{Example} \newcommand{\exam}{\begin{Bsp}} \newcommand{\xexam}{\end{Bsp}} \newcommand{\refex}[1]{Example~\ref{exam--#1}}
\newtheorem{Syno}[equation]{Synopsis} \newcommand{\syno}{\begin{Syno}} \newcommand{\xsyno}{\end{Syno}} \newcommand{\refsy}[1]{Synopsis~\ref{syno--#1}}
\newtheorem{Wish}[equation]{Wishful thinking} \newcommand{\wish}{\begin{Wish}} \newcommand{\xwish}{\end{Wish}} 
\newtheorem{Ques}[equation]{Question} \newcommand{\ques}{\begin{Ques}} \newcommand{\xques}{\end{Ques}} 

\newtheorem{Bem}[equation]{Remark} \newcommand{\rema}{\begin{Bem}} \newcommand{\xrema}{\end{Bem}} \newcommand{\refre}[1]{Remark~\ref{rema--#1}}

\theoremstyle{plain}
\newtheorem*{Theox}{Theorem} \newcommand{\theox}{\begin{Theox}} \newcommand{\xtheox}{\end{Theox}}
\newtheorem{Theo}[equation]{Theorem} \newcommand{\theo}{\begin{Theo}} \newcommand{\xtheo}{\end{Theo}} \newcommand{\refth}[1]{Theorem~\ref{theo--#1}}
\newtheorem{Satz}[equation]{Proposition} \newcommand{\prop}{\begin{Satz}} \newcommand{\xprop}{\end{Satz}} \newcommand{\refpr}[1]{Proposition~\ref{prop--#1}}
\newtheorem{Lemm}[equation]{Lemma} \newcommand{\lemm}{\begin{Lemm}} \newcommand{\xlemm}{\end{Lemm}} \newcommand{\refle}[1]{Lemma~\ref{lemm--#1}}
\newtheorem{Coro}[equation]{Corollary} \newcommand{\coro}{\begin{Coro}} \newcommand{\xcoro}{\end{Coro}} \newcommand{\refco}[1]{Corollary~\ref{coro--#1}}
\newtheorem{Nota}[equation]{Notation} \newcommand{\nota}{\begin{Nota}} \newcommand{\xnota}{\end{Nota}} \newcommand{\refno}[1]{Notation~\ref{nota--#1}}

\newcommand{\refsect}[1]{§\ref{sect--#1}}
\newcommand{\refit}[1]{\ref{item--#1}}
\newcommand{\refeq}[1]{(\ref{eqn--#1})}

\newcommand{\eqn}{\begin{equation}} \newcommand{\xeqn}{\end{equation}}
\newcommand{\eqnarr}{\begin{eqnarray*}} \newcommand{\xeqnarr}{\end{eqnarray*}}
\newcommand{\eqnarra}{\begin{eqnarray}} \newcommand{\xeqnarra}{\end{eqnarray}}

\newcommand{\pf}{\begin{proof}} \newcommand{\xpf}{\end{proof}}

\newcommand{\nc}{\newcommand}
\nc{\StP}[1]{\cite[\href{http://stacks.math.columbia.edu/tag/#1}{Tag #1}]{StacksProject}} 

\nc{\on}{\operatorname}
\nc{\aff}{{\on{aff}}}
\nc{\modi}{{\on{mod}}} 
\nc{\even}{{\on{even}}}
\nc{\odd}{{\on{odd}}}
\nc{\naive}{{\on{naive}}}
\nc{\hofib}{\on{hofib}}
\nc{\Bun}{\on{Bun}}
\nc{\ad}{{\on{ad}}}
\nc{\lft}{{\on{lft}}}

\nc{\str}{\on{-}}
\nc{\perf}{{\on{perf}}}
\nc{\Rel}{{\on{Pos}}}
\nc{\lan}{\langle}
\nc{\ran}{\rangle}

\nc{\bbA}{{\A}} 
\nc{\bbB}{{\mathbb B}}
\nc{\bbC}{{\mathbb C}}
\nc{\bbD}{{\mathbb D}}
\nc{\bbE}{{\mathbb E}}
\nc{\bbG}{{\mathbf G}}
\nc{\bbH}{{\mathbb H}}
\nc{\bbI}{{\mathbb I}}
\nc{\bbJ}{{\mathbb J}}
\nc{\bbK}{{\mathbb K}}
\nc{\bbL}{{\mathbb L}}
\nc{\bbM}{{\mathbb M}}
\nc{\bbN}{{\N}} 
\nc{\bbO}{{\mathbb O}}
\nc{\bbP}{{\P}} 
\nc{\bbQ}{{\Q}} 
\nc{\bbR}{{\mathbb R}}
\nc{\bbS}{{\mathbb S}}
\nc{\bbT}{{\mathbb T}}
\nc{\bbU}{{\mathbb U}}
\nc{\bbV}{{\mathbb V}}
\nc{\bbW}{{\mathbb W}}
\nc{\bbX}{{\mathbb X}}
\nc{\bbY}{{\mathbb Y}}
\nc{\bbZ}{{\bf Z}}

\nc{\calA}{{\mathcal A}}
\nc{\calB}{{\mathcal B}}
\nc{\calC}{{\mathcal C}}
\nc{\calD}{{\mathcal D}}
\nc{\calE}{{\mathcal E}}
\nc{\calF}{{\mathcal F}}
\nc{\calG}{{\mathcal G}}
\nc{\calH}{{\mathcal H}}
\nc{\calI}{{\mathcal I}}
\nc{\calJ}{{\mathcal J}}
\nc{\calK}{{\mathcal K}}
\nc{\calL}{{\mathcal L}}
\nc{\calM}{{\mathcal M}}
\nc{\calN}{{\mathcal N}}
\nc{\calO}{{\mathcal O}}
\nc{\calP}{{\mathcal P}}
\nc{\calQ}{{\mathcal Q}}
\nc{\calR}{{\mathcal R}}
\nc{\calS}{{\mathcal S}}
\nc{\calT}{{\mathcal T}}
\nc{\calU}{{\mathcal U}}
\nc{\calV}{{\mathcal V}}
\nc{\calW}{{\mathcal W}}
\nc{\calX}{{\mathcal X}}
\nc{\calY}{{\mathcal Y}}
\nc{\calZ}{{\mathcal Z}}

\nc{\Sht}{{\on{Sht}}}
\nc{\Frob}{{\on{Frob}}}
\nc{\Hecke}{{\on{Hecke}}}
\nc{\inv}{{\on{inv}}}
\nc{\Conv}{{\on{Conv}}}
\nc{\triv}{{\on{triv}}}
\nc{\Isom}{{\on{Isom}}}

\nc{\scrB}{{\mathscr{B}}}
\nc{\scrA}{{\mathscr{A}}}
\nc{\bbf}{{\mathbf{f}}}
\nc{\bba}{{\mathbf{a}}}

\nc{\al}{\alpha}
\nc{\be}{\beta}
\nc{\ga}{\gamma}
\nc{\la}{\lambda}
\nc{\qcqs}{{\on{qcqs}}}

\nc{\pot}[1]{ [\hspace{-0,5mm}[ {#1} ]\hspace{-0,5mm}] }
\nc{\rpot}[1]{ (\hspace{-0,7mm}( {#1} )\hspace{-0,7mm}) }

\nc{\defined}{\hspace{0.1cm}\stackrel{\text{\tiny \rm def}}{=}\hspace{0.1cm}}
\nc{\co}{\colon}

\newcommand{\category}[1]{\mathrm{#1}}
\newcommand{\DGCat}{\category{DGCat}} 
\newcommand{\cont}{\category{cont}} 
\newcommand{\Cat}{\category{Cat}} 
\newcommand{\simpl}{\category s} 
\newcommand{\res}{\mathrm{res}} 
\newcommand{\Rings}{\category{Rings}} 
\newcommand{\Shv}{\category{Shv}} 
\newcommand{\Fun}{\category{Fun}} 
\newcommand{\AffSch}{\category{AffSch}} 
\newcommand{\PreStk}{\category{PreStk}} 
\newcommand{\Stk}{\category{Stk}} 
\newcommand{\Gpd}{\category{Gpd}} 
\newcommand{\Grps}{\category{Grps}} 
\newcommand{\Zar}{\mathrm{Zar}} 
\newcommand{\Du}{\mathrm{D}} 
\newcommand{\sPSh}{\simpl\category{PSh}} 
\newcommand{\Sm}{\category{Sm}} 
\newcommand{\Corr}{\category{Corr}} 
\newcommand{\Sch}{\category{Sch}} 
\newcommand{\IndSch}{\category{IndSch}} 
\newcommand{\IndArt}{\category{IndArt}} 
\newcommand{\IndAlgSp}{\category{IndAlgSp}} 
\newcommand{\Sets}{\category{Sets}} 
\newcommand{\Mod}{\category{Mod}} 
\newcommand{\BarC}{{\rm Bar}} 

\def\wt{\mathrm {w}} 
\def\cl{\mathrm {cl}} 
\def\mot{\mathrm {m}} 
\def\clH{{}^\cl \H} 
\def\motH{{}^\mot \H} 
\def\Gm{\mathbf {G}_\mathrm m} 
\def\GL{\mathrm {GL}} 
\newcommand{\GmX}[  1]{\mathbf {G}_{\mathrm {m}, #1}} 
\def\IC{\mathrm{IC}} 
\def\red{\mathrm{red}} 
\def\pfp{\mathrm{pfp}} 
\def\ft{\mathrm{ft}} 

\font\tencyr=wncyr10
\font\sevencyr=wncyr7
\font\fivecyr=wncyr5
\def\HBei{{\rm H_{\fam15 B}}} 

\newcommand{\colim}{\operatornamewithlimits{colim}} 
\newcommand{\codim}{\operatornamewithlimits{codim}} 
\def\id{{\rm id}} 
\def\ev{{\rm ev}} 
\def\opp{{\rm op}} 
\newcommand{\op}[1]{#1^\circ} 
\def\To#1#2{\mathop{\count0=#1 \loop\ifnum\count0>0 \smash-\mkern-7mu \advance\count0 -1 \repeat \mathord\rightarrow}\limits^{#2}} 
\def\CH{\mathop{\rm CH}\nolimits} 
\def\Hom{\mathop{\rm Hom}\nolimits} 
\def\rk{\mathop{\rm rk}} 
\def\Ind{\category{Ind}} 
\def\Pro{\category{Pro}} 
\def\Gr{\mathop{\rm Gr}\nolimits} 
\def\Fl{\mathop{\rm Fl}\nolimits} 
\def\Sht{\mathop{\rm Sht}\nolimits} 
\def\Ext{\mathop{\rm Ext}\nolimits} 
\def\HdR{\mathop{\rm H_{dR}}\nolimits} 
\def\IHom{\underline{\Hom}} 
\def\Aut{\mathop{\rm Aut}\nolimits} 
\def\Rep{\category{Rep}} 
\def\bound{{\rm b}} 
\def\et{\mathrm{\acute et}} 
\def\fpqc{\mathrm{fpqc}} 
\def\ins{\mathrm{ins}} 
\def\incl{\mathrm{incl}} 

\definecolor{hellgrau}{RGB}{200,200,200} 
\definecolor{dunkelgrau}{RGB}{160,160,160} 
\definecolor{hellblau}{RGB}{194, 215, 249} %
\definecolor{dunkelblau}{RGB}{68, 128, 226} %

\def\Z{{\bf Z}} 
\def\bbF{{\bf F}}
\def\Fp{{\bf F}_p} %
\def\Fq{{\bf F}_q} %
\def\N{{\bf N}} 
\def\Q{{\bf Q}} 
\def\Ql{{\Q_\ell}} 
\def\Zl{{\Z_\ell}} 
\def\A{{\bf A}} 
\renewcommand{\P}[1][1]{\mathbf P^{#1}} 
\def\Gm{\mathbf {G}_\mathrm m} 

\def\H{{\rm H}} 
\def\h{{\rm h}} 
\def\im{{\rm im}} 
\def\SH{\category{SH}} %
\def\DM{\category{DM}} 
\def\DA{\category{DA}} 
\def\DTM{\category{DTM}} 
\def\Art{\category{Art}} 
\def\MTM{\category{MTM}} 
\def\ii{$\infty$}
\def\bound{{\rm b}} 

\def\Spec{\mathop{\rm Spec}} 
\newcommand{\M}{\mathrm{M}} 
\newcommand{\comp}{\mathrm{c}} 
\newcommand{\cstr}{\mathrm{cons}} 
\newcommand{\C}{\mathcal{C}} 
\newcommand{\D}{\category{D}} 
\newcommand{\cO}{\mathcal{O}}

\def\sbuildrel#1\over#2{\mathrel{\smash{\mathop{\kern0pt #2}\limits^{#1}}}}

\let\x\times
\let\ol\overline
\renewcommand{\t}{\otimes}

\renewcommand{\r}{\rightarrow}
\newcommand{\lr}{\longrightarrow}

\def\matrix#1{\null\,\vcenter{\normalbaselines
    \ialign{\hfil$##$\hfil&&\quad\hfil$##$\hfil\crcr
      \mathstrut\crcr\noalign{\kern-\baselineskip}
      #1\crcr\mathstrut\crcr\noalign{\kern-\baselineskip}}}\,}

\newdimen\harrowsize
\def\mapright#1{\smash{\mathop{\hbox to\harrowsize{\rightarrowfill}}\limits^{#1}}}

{\catcode`@=11
\gdef\cal{\fam\tw@}
\global\let\over\@@over
\global\let\atop\@@atop
\global\let\above\@@above
\global\let\overwithdelims\@@overwithdelims
\global\let\atopwithdelims\@@atopwithdelims
\global\let\abovewithdelims\@@abovewithdelims
\gdef\eqalign#1{\null\,\vcenter{\openup\jot\m@th
  \ialign{\strut\hfil$\displaystyle{##}$&$\displaystyle{{}##}$\hfil
      \crcr#1\crcr}}\,}
\newskip\xcentering \global\xcentering=0pt plus 1000pt minus 1000pt
\gdef\eqalignno#1{\displ@y \tabskip\xcentering
  \halign to\displaywidth{\hfil$\@lign\displaystyle{##}$\tabskip\z@skip
    &$\@lign\displaystyle{{}##}$\hfil\tabskip\xcentering
    &\llap{$\@lign##$}\tabskip\z@skip\crcr
    #1\crcr}}
\global\def\cases#1{\left\{\,\vcenter{\normalbaselines\m@th
    \ialign{$##\hfil$&\quad##\hfil\crcr#1\crcr}}\right.}
\gdef\eqlabel#1{\refstepcounter{equation}\label{eqn--#1}\eqno\hbox{\@eqnnum}}
}

\def \nts#1{ 
}
\def \arxiv#1{ 
\noindent\colorbox{hellblau}{\parbox{\dimexpr\textwidth-2\fboxsep\relax}{#1}}
}
\def \journal#1
{ 
\noindent\colorbox{dunkelblau}{\parbox{\dimexpr\textwidth-2\fboxsep\relax}{#1}}
}

\begin{document}

\title{The intersection motive of the moduli stack of shtukas}
\author{by Timo Richarz* and Jakob Scholbach}

\thanks{*Research of T.R.~funded by the Deutsche Forschungsgemeinschaft (DFG, German Research Foundation) - 394587809. Resarch of J.S.~funded by DFG, Sonderforschungsbereich 878.}

\maketitle

\begin{abstract}
For a split reductive group $G$ over a finite field, we show that the intersection (cohomology) motive of the moduli stack of iterated $G$-shtukas with bounded modification and level structure is defined independently of the standard conjectures on motivic $t$-structures on triangulated categories of motives. This is in accordance with general expectations on the independence of $\ell$ in the Langlands correspondence for function fields.
\end{abstract}

\setcounter{tocdepth}{1}
\tableofcontents

\section{Introduction}

\subsection{Motivation and goals}
Let $X$ be a variety over a field $k$, and fix a prime number $\ell\in \bbZ$ invertible on $k$. Since Grothendieck's construction of the $\ell$-adic \'etale cohomology groups $\H^i_\et(X,\bbQ_\ell)$, a central question is whether these are independent of the auxiliary prime number $\ell$, or even whether it is possible to find ``natural'' rational structures on them. This leads to the {\em theory of motives}.
Envisioned by Beilinson, and realised by Voevodsky \cite{Voevodsky:TCM}, Levine \cite{Levine:MixedMotives}, and Hanamura \cite{Hanamura:Motives1} for motives over $S=\Spec(k)$, and extended by Ayoub \cite{Ayoub:Six1, Ayoub:Six2, Ayoub:Realisation} and Cisinski--Déglise \cite{CisinskiDeglise:Triangulated, CisinskiDeglise:Etale} to motives over general base schemes $S$, there now exists a theory of motivic sheaves, i.e., a full six functor formalism for suitable triangulated categories $\DM(S)=\DM(S,\bbQ)$ of motives with rational coefficients over $S$.
By construction this theory of motives is independent of $\ell$, but explicit computations are difficult.
One of the main obstacles is the lack of a motivic $t$-structure on these categories, i.e., the existence of an object $\h^i(X) \in \DM(\Spec k)$ whose $\ell$-adic realization would be $\H^i_\et(X, \Ql)$.
This is part of the package of standard conjectures on motives which seem to be out of reach at the moment.

Levine \cite{Levine:Tate} showed, however, that a $t$-structure does exist on the subcategory $\DTM(S)\subset \DM(S)$ of Tate motives for certain ``nice'' $S$, for example, $S=\Spec(\Fq)$ or a smooth curve over $\Fq$, cf.~\refex{BS.exam}.
By definition, $\DTM(S)$ is the subcategory generated by motives of $\P[n]_S$ ($n \ge 0$) and their duals.
Soergel and Wendt \cite{SoergelWendt:Perverse} have extended Levine's observation to the case when $X$ is an $S$-scheme equipped with a so-called cellular Whitney-Tate stratification:
loosely speaking, this condition means that the strata of $X$ are built out of products of $\bbG_{m,S}$ or $\A^1_S$, and that one needs to be able to control the singularities of the closures of the strata.
While this condition is rather restrictive in general, it turns out that several varieties $X$ of interest in geometric representation theory do carry such stratifications.
For example, the flag variety $X=G/B$ associated with a split reductive $k$-group $G$ and a Borel subgroup $B$, equipped with its stratification by $B$-orbits, has this property \cite[Prop.~4.10]{SoergelWendt:Perverse}. In this situation, the category $\DTM(X)$ of {\em stratified Tate motives}, i.e., those which are Tate motives on each stratum, carries a $t$-structure whose heart is the abelian category of {\em mixed stratified Tate motives} $\MTM(X)$. The simple objects in $\MTM(X)$ are Tate twists of the intersection motives on the closures of the orbits of the left $B$-action on $X$.

The present paper is the first in a series aiming towards systematically applying the theory of motives as above to the constructions in the work of V.~Lafforgue \cite{Lafforgue:Chtoucas} on the Langlands correspondence over function fields. Here, we have two goals:

\begin{enumerate}
\item We provide a framework to handle motives on a large class of ``geometric objects'', namely prestacks.
\item We show that the {\em intersection cohomology motive} of the moduli stack of iterated $G$-shtukas with bounded modification and level structure is unconditionally defined, i.e., without reference to the standard conjectures.
\end{enumerate}

Part i) follows ideas of Gaitsgory-Rozenblyum \cite[\S3, 0.1.1]{GaitsgoryRozenblyum:StudyI}, and of Raskin \cite{Raskin:D-modules}.
A benefit of \refde{DM.prestacks} is that categories of motives on objects such as $X=L^+G\backslash LG/ L^+G$ (the double quotient of the loop group by the positive loop group) are well-defined independently of choices of presentations of the ind-scheme $LG$ or the pro-algebraic group $L^+G$. Here we provide only as much of the general theory as needed in order to construct ii).

Given i), the construction in ii) ultimately boils down (cf.~\S\ref{sect--DTM.Fl}, \S\ref{sect--intersection}) to an extension of the methods of \cite{SoergelWendt:Perverse} and of Soergel-Virk-Wendt \cite{SoergelVirkWendt:Equivariant} from ordinary schemes with an action of an algebraic group to ind-schemes with an action of a pro-algebraic group.
Our constructions are also applied in forthcoming work \cite{RicharzScholbach:Satake} to construct a motivic Satake equivalence in this set-up.

We note that ii) is in accordance with V.~Lafforgue's conjecture \cite[Conj.~12.12]{Lafforgue:Chtoucas} which says that the decomposition of the space of cusp forms obtained in {\em loc.~cit.} over $\bar{\bbQ}_\ell$ is in fact defined over $\bar{\bbQ}$ and does neither depend on the embedding $\bar{\bbQ}\hookrightarrow\bar{\bbQ}_\ell$ nor on the chosen prime number $\ell$.

\subsection{Statement of results}

Let $S$ be the spectrum of a field, or the spectrum of the integers (or more generally as in \refno{S}). Let $\AffSch_S$ (resp.~$\AffSch_S^\ft$) be the category of affine schemes equipped with a map (resp.~finite type map) to $S$. For each $X\in \AffSch_S^\ft$, we let $\DM(X)=\DM(X,\bbQ)$ be the triangulated category of motives with rational coefficients à la Ayoub and Cisinski-D\'eglise, cf.~\refsect{DM.schemes}. For a map $f\co X\to Y$ in $\AffSch_S^\ft$, there are pairs of adjoint functors $(f^*,f_*)$, $(f_!,f^!)$ satisfying the usual compatibilities, cf.~\refsy{motives}.
For our purposes, it is convenient to view $\DM(X)$ as a presentable stable \ii-category, and following Hoyois and Khan the pair of adjoint functors $(f^*,f_*)$, $(f_!,f^!)$ can be viewed as colimit-preserving functors between these categories.
We need to consider motives on quite general geometric objects, e.g.~ind-Artin stacks, or stacky quotients by pro-algebraic groups.
Following ideas of Gaitsgory-Rozenblyum and Raskin, it is convenient to use the notion of \ii-prestacks: The category of prestacks is the functor category in the sense of Lurie
\[
\PreStk_S\defined \Fun((\AffSch_S)^\opp,\infty\str\Gpd).
\]
The existence of all (homotopy) limits and colimits in $\PreStk_S$ has the advantage that ind-objects (ind-schemes, ind-Artin stacks), pro-objects (pro-algebraic groups), or quotients under group actions are prestacks.

\medskip
\noindent\textbf{Theorem A.}\textit{
i\textup{)} The presheaf of \ii-categories $\DM\co (\AffSch_S^\ft)^\opp\to \Cat_\infty$ given by
\[
X\mapsto \DM(X), \;\; f\mapsto f^!
\]
can be upgraded to a presheaf $\DM\co \PreStk_S\to \DGCat_\cont$ where the target is the \ii-category of presentable stable $\bbQ$-linear dg-\ii-categories with continuous functors. In particular, for any map of prestacks $f\co X\to Y$ there is a functor $f^!\co \DM(Y)\to \DM(X)$ in $\DGCat_\cont$ \textup{(}see \textup{\refsect{DM.prestacks})}.\medskip\\
ii\textup{)} The presheaf $\DM\co \PreStk_S\to \DGCat_\cont$ is a sheaf of \ii-categories for Voevodsky's $h$-topology \textup{(}see \textup{\ref{theo--DM.descent}, \ref{prop--unseparated}, \ref{theo--descent.prestacks})}.\medskip\\
iii\textup{)}  For any map $f\co X\to Y$ of strict ind-Artin stacks ind-\textup{(}locally of finite type\textup{)} over $S$, there is an adjunction
\[
f_! : \DM(X) \leftrightarrows \DM(Y) : f^!.
\]
For any prime $\ell\in \bbZ$ invertible on $S$, this adjunction agrees under the $\ell$-adic \'etale realization functor with the adjunction constructed in the work of Liu-Zheng in the case of Artin stacks $X$ and $Y$ \textup{\cite{LiuZheng:Enhanced, LiuZheng:EnhancedAdic}} \textup{(}see \textup{\ref{theo--f!.Artin})}.\medskip\\
iv\textup{)} For strict ind-schemes of ind-\textup{(}finite type\textup{)} over $S$, there is a six functor formalism under certain restrictions on the $*$-pullback functor \textup{(}see \textup{\ref{theo--motives.Ind-schemes})}.
\medskip\\
v\textup{)} Hom-sets in the category $\DM(X/G)$ of equivariant motives reproduce equivariant higher Chow groups as introduced by Totaro and Edidin--Graham \cite{Totaro:Chow, Totaro:Motive, EdidinGraham:Equivariant} \textup{(}see \textup{\ref{theo--equivariant.Chow})}.
}
\medskip

We comment on the result: In i), it is also possible to upgrade the presheaf $X \mapsto \DM(X), f \mapsto f^*$ to prestacks. We work with the $!$-pullback because of the following result of Lurie (cf.~\refle{Lurie.co.limit} below): Let $X=\colim_i X_i$ where $X_i$ are Artin stacks locally of finite type over $S$, and the transition maps $t_{i,j}\co X_i\to X_j$ are closed immersions. Then there are natural equivalences (in $\DGCat_\cont$)
\[
\DM(X) \;=\; \lim_{t^!} \DM(X_i) \;=\; \colim_{t_!} \DM(X_i).
\]
This is in accordance with ad hoc definitions of say bounded derived categories of \'etale constructible sheaves on strict ind-Artin stacks or strict ind-schemes as commonly used. Part ii) shows that the category of motives on prestacks is insensitive to $\tau$-stackification (in the \ii-sense) where $\tau$ is a Grothendieck topology contained in the $h$-topology. More precisely, for any prestack $X$, the $!$-pullback along the canonical map $X\to X^\tau$ induces an equivalence (in $\DGCat_\cont$)
\[
\DM(X^\tau) \;=\; \DM(X).
\]
In particular (take $\tau=\Zar$), it shows that for all Noetherian schemes of finite Krull dimension $X$ (equipped with a map to $S$) the category $\DM(X)$ is the category of motives as defined in Cisinski-D\'eglise \cite{CisinskiDeglise:Triangulated}, cf.~\refre{DM.prestacks} v).
Part ii) is similar to \cite[Prop.~6.23]{Hoyois:Six}. Part iii) allows to conveniently define the motive of an ind-Artin stack $f\co X\to S$ ind-(locally of finite type) as
\[
M(X)\defined f_!f^!(X)\;\in\; \DM(S).
\]
This recovers the motive $\M(X)$ of finite type schemes $X$ over a perfect field defined by Voevodsky, and puts computations of the motive of the affine Grassmannian in \cite{Bachmann:Affine}, or of the motive of the moduli stack of vector bundles on a curve in \cite{HoskinsLehalleur:Formula} into a more functorial context. Part iv) concludes the basic framework as needed in the constructions of the present manuscript. To keep the manuscript at a reasonable length, we chose not to discuss extensions of the full six functor formalism, say to (higher) Artin stacks as provided in the \'etale set-up by \cite{LiuZheng:Enhanced, LiuZheng:EnhancedAdic}.

With applications to geometric representation theory in mind, and more specifically to \cite{Lafforgue:Chtoucas}, we aim to construct {\em intersection motives}. Their existence is predicted by the standard conjectures. In order to make our results unconditional, we need to drastically restrict the class of objects, cf.~\refsect{DTM}. Here, we consider a strict ind-scheme of ind-finite type over $S$ equipped with a stratification into locally closed strata
$$\iota: X^+ := \bigsqcup_{w \in W} X_w \r X,$$
where each stratum $X_w$ is a cellular $S$-scheme, i.e., each $X_w\to S$ is smooth, and can be stratified further by products of $\bbG_{m,S}$ or $\bbA^1_{S}$.
Let $\DTM(X^+)\subset \DM(X^+)$ be the full subcategory generated by Tate motives, i.e., by all unit motives $1_{X_w}(n)[m]$, for all $w\in W$, $n,m \in \Z$, cf.~\refde{Tate.geometry}.
Following Soergel-Wendt \cite[Def.~4.5]{SoergelWendt:Perverse} (who consider the case of finite type schemes over fields), the stratification $\iota\co X^+\to X$ is called Whitney-Tate if
\[
\iota^*\iota_*1_{X^+} \;\in\; \DTM(X^+),
\]
i.e., $(\iota|_{X_v})^*(\iota|_{X_w})_*1_{X_w}$ is a Tate motive on $X_v$ for each $v,w\in W$. As is well-known \cite[1.4.9]{BBD}, this condition ensures that the categories of Tate motives on the single strata can be glued together.
Applications of Whitney-Tate stratified motives to representation theory have also been studied by Eberhardt and Kelly \cite{EberhardtKelly:Mixed}.

\medskip
\noindent\textbf{Theorem B.}
\textit{
Let $X$ be an ind-scheme of ind-finite type over $S$ equipped with a cellular Whitney-Tate stratification $\iota\co X^+\to X$.\medskip\\
i\textup{)}
There is a well-defined stable full \ii-subcategory $\DTM(X) \subset \DM(X)$ of stratified Tate motives. It consists of those motives $M$ such that the pullback $\iota^* M$ or, equivalently $\iota^! M$ is Tate, i.e., lies in $\DTM(X^+)$ \textup{(}see \textup{\ref{defilemm--Whitney.Tate.condition})}.\medskip\\
ii\textup{)}
If $S$ satisfies the Beilinson-Soulé vanishing conjecture, there is a self-dual motivic $t$-structure on $\DTM(X)$. Its heart $\MTM(X)$ is generated by the intersection motives
\[
\IC_w(n) \defined j_{w,!*}\left(1_{X_w}(n)[d_w]\right) \;\in\; \DTM(X)
\]
supported on the closure $\bar X_w$, where $j_w\co X_w\to \bar{X}_w$, $w\in W$, $n\in\bbZ$. These are precisely the simple objects in $\MTM(X)$, and their $\ell$-adic realizations are the intersection complexes defined by the middle perverse extension of the constant $\ell$-adic sheaves $\bbQ_\ell(n)$ on the strata $X_w$ \textup{(}see \textup{\ref{coro--t-structure}, \ref{theo--simple.objects})}.\medskip\\
iii\textup{)} Let $G=\lim_{i\geq 0}G_i$ be a pro-smooth affine $S$-group scheme acting on $X$ compatibly with the stratification. Assume that the scheme underlying each $G_i$ is cellular, and that each $\ker(G_{i+1}\to G_i)$ is a vector group. There is a well-defined stable full \ii-subcategory $\DTM_G(X) \subset \DM(G\backslash X)$ of equivariant stratified Tate motives where $G\backslash X$ is the prestack quotient. If $S$ is as in ii\textup{)}, then there is a self-dual motivic $t$-structure on $\DTM_G(X)$. Its heart $\MTM_G(X)$ contains the intersection motives $\IC_w(n)$ which map under the forgetful functor
\[
\MTM_G(X) \to \MTM(X)
\]
to the intersection motives as in ii\textup{)}. If $G$ has connected fibers, this forgetful functor is fully faithful. If in addition the stabilizers of the $G$-action are connected, then $\MTM_G(X)$ is also generated by the $\IC_w(n)$, $w\in W$, $n\in \bbZ$ \textup{(}see \textup{\ref{defi--DTM.G}, \ref{prop--DTM.G}, \ref{prop--MTM.G}, \ref{prop--equivariant.MTM}, \ref{prop--generators.DTM.G})}.
}
\medskip

Having the six-functor formalism for motives on ind-schemes, the proofs of i) and ii) are immediate from \cite{SoergelWendt:Perverse}. We note that the Beilinson-Soulé conjecture is known by Quillen's, Borel's, and Harder's work, for $S$ being the spectrum of a finite field, a number field or localizations of its ring of integers, for a smooth curve over a finite finite or its function field, and finally for filtered colimits of such objects, cf.~\refex{BS.exam}. We now comment on iii). The notation $\DTM_G(X)$ (as opposed to $\DTM(G\backslash X)$) highlights the fact that this category depends on the chosen presentation of the prestack $G\backslash X$. Put differently, Tate motives do (by definition) not satisfy descent, cf.~\refex{descent.Tate.not}. The assumption on $\ker(G_{i+1}\to G_i)$ being a vector group ensures that $\DM(G\backslash \bar{X}_w)=\DM(G_i\backslash \bar{X}_w)$ for $i>\!\!>0$. It is satisfied for an interesting class of pro-smooth affine group schemes which are constructed as ``positive loop groups'' (a.k.a.~``jet groups''), cf.~\refpr{vector.extension}.

We now give two closely related applications of the theory developed so far, cf.~\S\S\ref{sect--DTM.Fl}-\ref{sect--intersection}. Let $G$ be a split reductive group over the integers $\bbZ$ (a.k.a.~a Chevalley group), and fix a Borel pair $T\subset B$. The loop group is the functor
\[
LG\co \Rings\to \Grps,\;\; R\mapsto G(R\rpot{\varpi}),
\]
where $R\rpot{\varpi}$ denotes the Laurent series ring in the formal variable $\varpi$. The group functor $LG$ is represented by an ind-affine ind-scheme over $\bbZ$. Associated with each facet $\bbf$ in the standard apartment $\scrA=\scrA(G,T)$ of the Bruhat-Tits building is the closed $\bbZ$-subgroup
\[
\calP_\bbf \;\subset \; LG,
\]
called {\em parahoric subgroup}. The group $\calP_\bbf=\lim_{i\geq 0}\calP_{\bbf,i}$ is a pro-smooth affine $\bbZ$-group scheme with connected fibers such that each $\ker(\calP_{\bbf,i+1}\to \calP_{\bbf,i})$ is a vector group. If $\bbf=\{0\}\subset \scrA$ is the base point, then $L^+G=\calP_{\{0\}}$ is the positive loop group given by $L^+G(R)=G(R\pot{\varpi})$. If $\bbf=\bba_0$ is the standard alcove determined by $B$, then $\calP_{\bba_0}$ is the (standard) Iwahori subgroup given as the preimage of $B$ under the reduction map $L^+G\to G$, $\varpi\mapsto 0$. For general facets, these groups can be described in terms of Bruhat-Tits theory.

Now fix two facets $\bbf,\bbf'\subset \scrA$ which are contained in the closure of the standard alcove. We are interested in motives on the double quotient
\[
\calP_{\bbf'}\backslash LG/\calP_\bbf \;\in\; \PreStk_{\Spec(\bbZ)}.
\]
We recover the set-up of Theorem B by considering the left-$\calP_{\bbf'}$-action on the ind-projective ind-scheme $\Fl_\bbf=(LG/\calP_{\bbf})^\et$, resp.~the right-$\calP_{\bbf}$-action on $\Fl_{\bbf'}^\opp=(\calP_{\bbf'}\backslash LG)^\et$. Using \'etale descent for motives as in Theorem A ii), we see that the resulting equivariant categories are the same, and equal to $\DM(\calP_{\bbf'}\backslash LG/\calP_\bbf)$. For convenience of the reader, we note that the \'etale stackification $(\calP_{\bbf'}\backslash LG/\calP_\bbf)^\et$ automatically is an fpqc stack, cf.~\refle{DM.double.tau}.

\medskip
\noindent\textbf{Theorem C.}\textit{
i\textup{)} The stratification of $\Fl_\bbf$ \textup{(}resp.~$\Fl_{\bbf'}^\opp$\textup{)} by orbits of the left-$\calP_{\bbf'}$-action \textup{(}resp.~right-$\calP_{\bbf}$-action\textup{)} is cellular and Whitney-Tate in the sense outlined above. Further, Theorem B i\textup{)}-iii\textup{)} applies \textup{(}see \textup{\ref{theo--Fl.WT}, \ref{theo--generators.DTM.flag})}.\medskip\\
ii\textup{)}
There is an equivalence of the resulting categories of equivariant stratified Tate motives
\[
\DTM_{\calP_{\bbf'}}(\Fl_\bbf)\;=\; \DTM_{\calP_{\bbf}}(\Fl_{\bbf'}^\opp)
\]
as full subcategories of $\DM(\calP_{\bbf'}\backslash LG/\calP_\bbf)$. If $\bbf=\bbf'$, then the $t$-structures delivered by part i\textup{)} \textup{(}and Theorem B\textup{)} also agree, and their heart is the $\bbQ$-linear abelian full subcategory
\[
\MTM(\calP_\bbf\backslash LG/\calP_\bbf) \;\subset\; \DM(\calP_\bbf\backslash LG/\calP_\bbf),
\]
which is the symmetric version of Theorem B iii\textup{)}. It is generated by intersection motives of the Schubert varieties inside $\Fl_\bbf$ \textup{(}see \textup{\ref{theo--equivariant.DTM.flag})}.
}
\medskip

The basic geometric properties of affine flag varieties for Chevalley groups over $\bbZ$, resp.~more general base schemes $S$ are given in \refsect{loop.grps}.
In Theorem C, we can replace $\Spec(\bbZ)$ by any regular base scheme $S$ which satisfies the Beilinson-Soulé vanishing conjecture as above.
Also it seems possible to extend our results to twisted affine flag varieties in the sense of Pappas-Rapoport \cite{PappasRapoport:LoopGroups}.
In \cite{RicharzScholbach:Satake} we will take $\bbf=\bbf'=\{0\}$ and use the category $\MTM(L^+G\backslash LG/L^+G)$ to establish a motivic Satake equivalence in this set-up.
A nice feature of the symmetry of the double quotient is that the operation $LG\to LG$, $g\mapsto g^{-1}$ induces an anti-involution on the category of stratified mixed Tate motives.

A version of Theorem C for Witt vector (or $p$-adic) affine flag varieties is contained in the first arxiv version of this paper, and will be published elsewhere together with a motivic Satake equivalence in this context.

With this general set-up in hand, we construct in \refsect{intersection} the intersection cohomology motive of the moduli stack of bounded shtukas.
We view these constructions as a first step towards a motivic rendition of V.~Lafforgue's work in the function field case of the Langlands program since it offers a geometric understanding of the $\ell$-adic intersection cohomology of the moduli stack of $G$-shtukas in a way which is independent of $\ell$.
Further steps towards this goal, including the afore-mentioned Satake equivalence, a motivic Drinfeld lemma, a motivic construction of excursion operators, and their identification with Hecke operators, remain to be done.

Let $T\subset B\subset G$ be defined over the finite field $k=\Fq$. Let $X$ be a smooth projective geometrically connected curve over $k$. For any effective divisor $N\subset X$ and any partitioned finite index set $I=I_1\sqcup\ldots\sqcup I_r$, there is the moduli stack of iterated $G$-shtukas with level-$N$-structure
\[
\Sht_{N,I}^{(I_1,\ldots,I_r)}\defined \lan (\calE_r,\be_r)\overset{\al_r}{\underset{I_r}{\dashrightarrow}}(\calE_{r-1},\be_{r-1})\overset{\al_r}{\underset{I_{r-1}}{\dashrightarrow}}\ldots\overset{\al_2}{\underset{I_2}{\dashrightarrow}} (\calE_1,\be_1)\overset{\al_1}{\underset{I_1}{\dashrightarrow}} (\calE_0,\al_0)=({^\tau\calE_r},{^\tau\al}_r) \ran,
\]
as considered in \cite[D\'ef.~2.1]{Lafforgue:Chtoucas}, cf.~\refsect{shtukas.def} for notation.
This moduli stack is equipped with a forgetful map $\Sht_{N,I}^{(I_1,\ldots,I_r)}\to (X\backslash N)^I$.
For an admissible tuple $\underline{\mu}=(\mu_i)_{i\in I}\in X_*(T)_+^I$ of dominant cocharacters, we can bound the relative positions of each modification $\al_j$ by $\mu_j$ to get a closed substack $\Sht_{N,I,\underline{\mu}}^{(I_1,\ldots,I_r)}\subset \Sht_{N,I}^{(I_1,\ldots,I_r)}$ which is representable by a (reduced) Deligne-Mumford stack locally of finite type over $k$. Varying $\underline{\mu}$, we see that $\Sht_{N,I}^{(I_1,\ldots,I_r)}$ has the structure of a strict ind-Deligne-Mumford stack ind-(locally of finite type) over $k$. In particular, Theorem A iii) applies in this context.
Fixing a total order on $I=\{1,\ldots,n\}$ which refines the partition $I=I_1\sqcup\ldots\sqcup I_r$, there are maps of \'etale sheaves of groupoids
\[
\Sht_{N,I}^{(I_1,\ldots,I_r)} \;\overset{\pi}{\longleftarrow}\; \Sht_{N,I}^{(\{1\},\ldots,\{n\})} \;\overset{\inv}{\longrightarrow}\; \bigsqcap_{i=1,\ldots,n}(L_{\{i\}}^+G\backslash L_{\{i\}}G/L_{\{i\}}^+G)^\et,
\]
where $\inv\co ((\calE_n,\be_n)\overset{\al_n}{\dashrightarrow}\dots \overset{\al_1}{\dashrightarrow}(\calE_0,\be_0)=(^\tau{\calE_n},{^\tau\be}_n))\mapsto (\inv(\al_i))_{i=1,\ldots, n}$ is the relative position, cf.~\eqref{invariant.map}.
The target of the map $\inv$ is the stack of relative positions: each factor in the product has a forgetful map to $X$ whose fiber over $x \in X$ is $(L^+G\backslash LG/L^+G)^\et\otimes_k \kappa(x)$.
Colloquially speaking, two such $G$-bundles $\calE_i$ and $\calE_{i-1}$ differ by an elementary modification $\al_i$ at some point of $X$, and the invariant $\inv(\al_i)$ is simply the double coset of the matrix changing the transition function from the one $G$-bundle to the other.
Further, there is a canonical map
\begin{equation}\label{map.intro}
(L_{\{i\}}^+G\backslash L_{\{i\}}G/L_{\{i\}}^+G)^\et\to \left(L^+\bbG_m\backslash(L^+G\backslash LG/L^+G)\right)^\et,
\end{equation}
where $L^+\bbG_m$ acts on $L^+G\backslash LG/L^+G$ by changing the variable $\varpi$ used to form the loop groups. For each $\mu\in X_*(T)_+$, the intersection motive $\IC_\mu\in \MTM(L^+G\backslash LG/L^+G)$ is $L^+\bbG_m$-equivariant, and its pullback via \eqref{map.intro} is denoted $\IC_{\mu,\{i\}}\in \DM((L_{\{i\}}^+G\backslash L_{\{i\}}G/L_{\{i\}}^+G)^\et)$. For each effective divisor $N\subset X$, and each admissible tuple $\underline{\mu}=(\mu_i)_{i\in I}\in X_*(T)_+$, we define
\[
\calF_{N,I,\underline{\mu}}^{(I_1,\ldots,I_r)}\defined\pi_{!}\left(\inv^! \left(\boxtimes_{i=1}^n\IC_{\{i\},\mu_i}\right)\right)\;\in\; \DM\left(\Sht_{N,I}^{(I_1,\ldots,I_r)}\right).
\]
Note that the construction combines the intersection motives obtained by Theorem B (whose assumptions are satisfied by Theorem C) and the general functoriality given by Theorem A.

\medskip
\noindent\textbf{Theorem D.} (see \ref{coro--prelim.intersection.motives})
\textit{
The motive $\calF_{N,I,\underline{\mu}}^{(I_1,\ldots,I_r)}$ is supported on $\Sht_{N,I,\underline{\mu}}^{(I_1,\ldots,I_r)}$, and its $\ell$-adic realization is \textup{(}up to twist and the choice of a lattice in the adelic center\textup{)} the intersection complex defined in \textup{\cite[D\'ef.~2.14]{Lafforgue:Chtoucas}}. This defines the intersection cohomology motive
\[
\calH_{N,I,\underline{\mu}}\defined p_!\left(\calF_{N,I,\underline{\mu}}^{(I_1,\ldots,I_r)}\right)\;\in\; \DM\left((X\backslash N)^I\right),
\]
whose $\ell$-adic realization is \textup{(}up to the normalizations above\textup{)} the intersection cohomology complex defined in \textup{\cite[D\'ef.~4.1]{Lafforgue:Chtoucas}}.
}
\medskip

The proof is immediate from results of Varshavsky. Namely, \cite[Cor.~2.21]{Varshavsky:Moduli} implies that the motives $\calF_{N,I,\underline{\mu}}^{(I_1,\ldots,I_r)}$ realize $\ell$-adically to the intersection complexes on $\Sht_{N,I,\underline{\mu}}^{(I_1,\ldots,I_r)}$:
Roughly, after cutting down the situation to finite-dimensional stacks, the map $\inv$ is the map onto the local model, cf.~\cite[Prop.~2.8, 2.9]{Lafforgue:Chtoucas}. This map is smooth, and thus pullback preserves the intersection complex (up to shift).
Next $\pi$ is a small map with connected fibers (cf.~\cite[Cor.~2.18]{Lafforgue:Chtoucas}), and thus pushforward preserves the intersection complex as well.
By construction, the motives $\calF_{N,I,\underline{\mu}}^{(I_1,\ldots,I_r)}$ depend on a total ordering of $I$, whereas their $\ell$-adic realizations only depend on the ordered partition $I=I_1\sqcup\ldots\sqcup I_r$, cf.~\cite[Thm.~1.17, Cor.~2.15, 2.16]{Lafforgue:Chtoucas}.
We expect that the same result holds true for the intersection motives, cf.~\refre{WT.properties.fusion}.

\bigskip

\noindent\textbf{Acknowledgements.} The authors greatly benefited from discussions with Johannes Ansch\"utz, Denis-Charles Cisinski, Laurent Fargues, Dennis Gaitsgory, Jochen Heinloth, Simon Henry, Thomas Nikolaus, Peter Scholze, Beno\^it Stroh, Torsten Wedhorn and Xinwen Zhu around the subject of the manuscript, and it is a pleasure to thank all of them. We also thank Peter Scholze for pointing out a mistake in a previous version, and the anonymous referee for his/her careful reading which greatly improved the quality of the manuscript.
The authors thank the University of M\"unster, Harvard University, Deutsche Forschungsgemeinschaft, the Institut de Math\'ematiques de Jussieu and the Technische Universit\"at Darmstadt for financial and logistical support which made this research possible.

\section{Motives}\label{sect--DM}

In this section we develop motives in the generality we need.
In \refsect{DM.schemes}, we list properties of motives on finite-dimensional schemes.
This part is mostly expository, except for two aspects: the functors $f_!$ and $f^!$ are established for non-separated maps. This is possible because of a conceptual formulation of $h$-descent for motives.
In \refsect{DM.prestacks}, we define motives on prestacks.
Prestacks are very general ``geometric objects'', encompassing (ind-)schemes, and quotients of them by (pro-)algebraic groups, and also (ind-)Artin stacks.
In \refsect{DM.Artin}, we prove a comparison result with the $!$-pushforward for a map of Artin stacks locally of finite type under the $\ell$-adic \'etale realization.
In \refsect{DM.ind-schemes}, we construct a six functor formalism for motives on ind-schemes.

\nota
\label{nota--S}
Throughout \refsect{DM}, $S$ is a regular scheme which is of finite type over a scheme $B$ which is Noetherian, excellent, and of dimension at most 2 \textup{(}for example of $B = \Spec(k)$, a field or the integers\textup{)}.
The category of \textup{(}not necessarily separated\textup{)} finite type $S$-schemes is denoted $\Sch_S^\ft$.
\xnota

\subsection{Motives on finite-dimensional schemes}
\label{sect--DM.schemes}
In this subsection, we recall several properties of the category of \emph{motives on a scheme $X$ with rational coefficients}
$$\DM(X) \defined \DM(X, \Q), \ \ \text{ for } X \in \Sch_S^\ft.$$
Very briefly, the category $\DM(X)$ is constructed from (unbounded) complexes of étale sheaves of $\Q$-vector spaces on the site $\Sm / X$ of smooth $X$-schemes. In this category, the relation $Y \x_X \A^1_X \cong Y$ is imposed for any $Y \in \Sm /X$.
Moreover, tensoring with $\P$ is made invertible.
This category is denoted by $\DA^\et(X, \Q)$ in \cite{Ayoub:Realisation} and by $\D_{\A, \et}(X, \Q)$ in \cite{CisinskiDeglise:Triangulated}.

The power of this approach to motives lies in the existence of a six-functor-formalism. It was first established by Ayoub \cite{Ayoub:Six1,Ayoub:Six2}.
Particularly relevant for this paper is the construction of $\DM$ due to Cisinski and Déglise \cite{CisinskiDeglise:Triangulated}.
These papers phrase many key results in terms of triangulated categories which is not sufficiently structured for the purposes of this paper.
Instead, we need to use an \ii-categorical formalism, which was worked out by Hoyois \cite{Hoyois:Six} and Khan \cite{Khan:Motivic} for the stable homotopy category $\SH$.
We now list properties of $\DM$ as needed in the paper.

\syno\label{syno--motives}
Let $f\co X\r Y$ be a map in $\Sch_S^\ft$. The category $\DM$ of motives with rational coefficients satisfies the following properties.
\smallskip\\
i) The category $\DM(X)$ is a stable, presentable, closed symmetric monoidal \ii-category with tensor structure denoted $\otimes$ and internal homomorphisms denoted $\IHom$.
Its monoidal unit is denoted $1$ or $1_X$.
It has all (homotopy) limits and colimits. As usual, the suspension functor is denoted by $[1]$.
(These properties hold since $\DM$ is the \ii-category associated to the model category denoted $\HBei$-mod in \cite[14.2.7ff]{CisinskiDeglise:Triangulated}.)
\nts{Existence of (homotopy) (co)limits holds since it the model category is bicomplete.}
\smallskip\\
ii)
The assignment $X \mapsto \DM(X)$ can be upgraded to a presheaf of symmetric monoidal \ii-categories \cite[Def 1.1.29]{CisinskiDeglise:Triangulated}
$$\DM^*\co (\Sch_S^{\ft})^\opp \r \Cat_\infty^\t,\; X \mapsto \DM(X),\; f \mapsto f^*.$$
For each $f$, there is an adjunction
$$f^* : \DM(Y) \rightleftarrows \DM(X) : f_*.\eqlabel{adjunction.star}$$
\nts{The existence of $f^*$, $f_*$ with monoidal $f^*$ is called complete in \cite[Definition 1.1.29]{CisinskiDeglise:Triangulated}. By definition, it is required for a motivic category.}
\smallskip\\
iii)
If $f$ is smooth, then $f^*$ has a left adjoint, denoted $f_\sharp$
(\cite[Def 1.1.2]{CisinskiDeglise:Triangulated} with $\calP$ consisting of smooth morphisms.)
\smallskip\\
iv)
The assignment $X \mapsto \DM(X)$ can be upgraded to a presheaf of \ii-categories\footnote{\label{foot.separated}In loc.~cit., whenver $f^!$ or $f_!$ are concerned, $f$ is required to be separated. See \refpr{unseparated} how to drop this assumption.}
$$\DM: (\Sch_S^\ft)^\opp \r \Cat_\infty,\; X \mapsto \DM(X),\; f \mapsto f^!.$$
For each $f$, there is an adjunction
$$f_! : \DM(X) \rightleftarrows \DM(Y): f^!.\eqlabel{adjunction.shriek}$$
For any factorization $f = p \circ j$ with $j$ an open immersion and $p$ a proper map, there is a natural equivalence $f_! \stackrel \cong \r p_* j_\sharp$.
In particular, for $p$ proper, $p^!$ is left adjoint to $p_! = p_*$.
(This pair of functors is the hardest to construct.
See \cite[§1.6.5]{Ayoub:Six1} for quasi-projective maps, \cite[Thm.~2.4.50, Prop 2.2.7]{CisinskiDeglise:Triangulated} in general. These statements are using the language of triangulated categories. See \cite[§6.2]{Hoyois:Six} or \cite[§5.2]{Khan:Motivic} in the context of \ii-categories.)
\nts{The existence of $f^!$, $f_!$ is in \cite[Theorem 2.4.50]{CisinskiDeglise:Triangulated}.
The given characterization of $f_!$ is \cite[Proposition 2.2.7]{CisinskiDeglise:Triangulated}.}
\smallskip\\
v) For the projection $p: \GmX X \r X$, and any $M \in \DM(X)$, the map $p_\sharp p^* M[-1] \r M[-1]$ in $\DM(X)$ is a split monomorphism.
The complementary summand is denoted by $M(1)$.
The functor $M \mapsto M(1)$ is an equivalence with inverse denoted by $M \mapsto M(-1)$ \cite[Def 2.4.17]{CisinskiDeglise:Triangulated}.
\nts{This is called the stability property \cite[Definition 2.4.17]{CisinskiDeglise:Triangulated}, by definition part of being motivic.}
\smallskip\\
vi) The category $\DM(X)$ is compactly generated by the objects $t_\sharp 1(n)$, $t: T \r X$ smooth and $n \in \Z$ \cite[Thm.~4.5.67]{Ayoub:Six2}. In particular, the monoidal unit $1_X \in \DM(X)$ is compact.
The functors $f_\sharp$, $f_*$, $f^*$, $f_!$, and $f^!$ preserve compact objects \cite[Prop.~15.1.4, Thm.~15.2.1]{CisinskiDeglise:Triangulated}.
These functors also preserve arbitrary (homotopy) colimits:
for $f_*$ and $f^!$ this follows from compact generation of $\DM$ and preservation of compact objects under $f^*$ and $f_!$, respectively.\nts{for the left adjoints $f_\sharp$, $f^*$ and $f_!$ this is clear}
\smallskip\\
vii)
There is a \emph{projection formula} $(f_! M) \t N = f_! (M \t f^* N)$ \cite[Thm.~2.4.50]{CisinskiDeglise:Triangulated}.\textsuperscript{\ref{foot.separated}}
\smallskip\\
viii)
If $p\co X\r S$ denotes the structural map, then the \emph{dualising functor}
$$\Du_X \defined \IHom(-, p^! 1)$$
is a (contravariant) involution on the subcategory $\DM(X)^\comp$ of compact objects, i.e., $\Du_X \circ \Du_X = \id$. Furthermore, on compact objects, there are equivalences (\cite[Thm.~15.2.4]{CisinskiDeglise:Triangulated}, this uses all the assumptions in \refno{S})\textsuperscript{\ref{foot.separated}}
$$\Du_Y f_! = f_* \Du_X, f^* \Du_Y = \Du_X f^!.$$
ix) For a closed immersion $i: Z \r X$ with complement $j: U \r X$, the (co)units of the adjunctions above assemble into so-called \emph{localization} homotopy fiber sequences
(see around \cite[Prop.~2.3.3]{CisinskiDeglise:Triangulated}\textsuperscript{\ref{foot.separated}}):
$$i_! i^! \r \id \r j_* j^* \stackrel {[1]\;} \r,\eqlabel{localizationOne}$$
$$j_! j^! \r \id \r i_* i^*\stackrel {[1]\;} \r.\eqlabel{localization}$$
x) For a cartesian diagram in $\Sch_S^\ft$
$$\xymatrix{
X' \ar[r]^{g'} \ar[d]^{f'} & X \ar[d]^f \\
Y' \ar[r]^g & Y,
}$$
there are natural equivalences (called \emph{base change}) \cite[Thm.~2.4.50]{CisinskiDeglise:Triangulated}\textsuperscript{\ref{foot.separated}}:
\begin{align}
g^! f_* & \stackrel \cong \lr f'_* g'^!, \label{eqn--base.change.1} \\
f^* g_! & \stackrel \cong \lr g'_! f'^*. \label{eqn--base.change.2}
\end{align}
xi) The category $\DM$ is \emph{homotopy-invariant} in the sense that for the projection map $p\co \A^n_X \r X$ for any $n\in\Z_{\geq 0}$, the counit and unit maps $p_\sharp p^* \r \id$ and $\id \r p_* p^*$ are functorial equivalences in $\DM(X)$ \cite[2.1.3]{CisinskiDeglise:Triangulated}.
\smallskip\\
xii) If $f$ is smooth of relative dimension $d$, there is a functorial equivalence (called \emph{relative purity})
$$f^! = f^* (d)[2 d].\eqlabel{relative.purity}$$
(See \cite[§1.6.3]{Ayoub:Six1} or \cite[Thm 2.4.50]{CisinskiDeglise:Triangulated} or \cite[pp. 272--273]{Hoyois:Six} for a formulation in the language of \ii-categories.
The identification of the Thom equivalence as stated follows from the orientability of $\DM$ \cite[§14.1.5]{CisinskiDeglise:Triangulated}.)
\smallskip\\
xiii) If $X$ is regular,
$$ \Hom_{\DM(X)}(1_X, 1_X(n)[m])\,=\,(K_{2n-m}(X) \t \Q)^{(n)},\eqlabel{DM.K-theory}$$
where the right term denotes the $n$-th Adams eigenspace in the rationalized algebraic $K$-theory of $X$ \cite[§14]{CisinskiDeglise:Triangulated}.
Here $\Hom_{\DM(X)}$ denotes the set of morphisms in the homotopy category.
If $X$ is of finite type over a field, this group also identifies with Bloch's higher Chow group  $\CH^n(X, 2n-m)_\Q$, see \cite{Bloch:Algebraic}.
\smallskip\\
xiv)
The presheaves $\DM^*$ and $\DM^!$ are sheaves for the h-topology (for which see, e.g., \cite[\S 8]{Rydh:Submersions}).
That is, if $f\co X \r Y$ is an $h$-covering in $\Sch_S^\ft$ with \v Cech nerve $C_f^\bullet \r Y$, $C_f^n=X^{\x_Y{n+1}}$, the natural map
$$\DM(Y) \;\r\; \DM^*(C_f^\bullet) := \lim_{n \in (\Delta_s)^\opp} \DM(C_f^n)$$
is an equivalence and likewise for $\DM^!$.
Here $\lim$ denotes the limit in the \ii-bicategory $\Cat_\infty$ of \ii-categories.
The category $\Delta_s$ is the subcategory of the usual simplex category $\Delta$ consisting of injective order-preserving maps.
We refer to this property as \emph{h-descent}.
(This is discussed in \refth{DM.descent} below.)
\smallskip\\
xv) Suppose $S$ is an excellent scheme and of Krull dimension at most 2 and $X  / S$ is separated.
The category $\DM(X)$ is equipped with a \emph{weight structure} $(\DM(X)^{\wt \le 0}, \DM(X)^{\wt \ge 0})$.
The subcategory $\DM(X)^{\wt \le 0}$ is generated -- under extensions, shifts $[n]$ with $n \le 0$, and arbitrary coproducts -- by objects of the form $f_! 1(n)[2n]$, where $f: T \r X$ is proper, $T$ is regular, and $n \in \Z$. Similarly $\DM(X)^{\wt \ge 0}$ is generated by these objects using shifts $[n]$ for $n \ge 0$ instead.
If $X$ is regular, $1_X$ is in the \emph{heart} $\DM(X)^{\wt = 0} = \DM(X)^{\wt \le 0} \cap \DM(X)^{\wt \ge 0}$ of this weight structure.
Moreover, $f^*$ and $f_!$ are weight-left exact (preserve ``$\wt \le 0$'') while $f_*$ and $f^!$ are weight-right exact (preserve ``$\wt \ge 0$''). The weight structure on compact objects is constructed in \cite[Thm 3.8]{Hebert:Structure} or \cite[Thm 2.2.1]{Bondarko:WeightsForRelative}. See \cite[Thm 2.2.1]{BondarkoSosnilo:Purely} for how to extend a weight structure on a stable \ii-category to its Ind-completion.\smallskip\\
xvi)
Let $\ell$ be a prime number invertible on $S$.
Then $S$ admits an \emph{$\ell$-adic realization functor}
$$\rho_\ell\co \DM(X) \r \D_\et(X, \Ql) := \Ind(\D^\bound_\cstr(X, \Ql))$$
taking values in the ind-completion of the bounded derived category of constructible $\Ql$-adic étale sheaves.
This functor is outlined in \refsect{realization}.
It commutes with the six functors $f_*, f^*, f_!, f^!, \t, \IHom$.
\xsyno

\rema
We emphasize that all work done here is with rational coefficients, cf. Synopsis xiii).
If we were to work with motives with integral coefficients, one can use the category of étale motives which also satisfies the properties listed in \refsy{motives}, except for xiii) and xv), cf.~\cite{Ayoub:Realisation, CisinskiDeglise:Etale}.
Using upcoming work of Spitzweck on $t$-structures on integral Tate motives, the constructions in this paper should carry over and produce a category of mixed motives on the affine Grassmannian $\MTM(\Gr, \Z)$ whose rationalization is $\MTM(\Gr,\Q)$.
\xrema

\subsubsection{h-descent and !-functoriality for non-separated maps}

\lemm
\label{lemm--descent.Beck.Chevalley}
Let $F: (\Sch_S^\ft)^\opp \r \Cat_\infty$ be a presheaf such that for any map $f \in \Sch_S^\ft$, $f^+ := F(f)$ has a left adjoint, denoted $f_+$.

Let now $f\co X \r Y$ be a map with \v Cech nerve $C^\bullet_f$, i.e., $C^n_f = X^{\x_Y n+1}$.
If $f$ is such that $f^+$ is conservative and satisfying the \emph{Beck-Chevalley condition}, i.e., such that for any cartesian diagram
$$\xymatrix{
X'' \ar[r]^{g'} \ar[d]^{f''} & X' \ar[r] \ar[d]^{f'} & X \ar[d]^f \\
Y'' \ar[r]^g & Y' \ar[r] & Y
}\eqlabel{Beck}$$
the natural transformation $f''_+ g'^+ \r g^+ f'_+$ is an equivalence, then $F$ satisfies descent with respect to $f$, i.e.,
the following natural functor is an equivalence of \ii-categories:
$$F(Y) \overset{\simeq}{\lr} \lim F(C_f^\bullet).$$
\xlemm

\pf
This is a slight reformulation of \cite[§ II.4.7]{GaitsgoryRozenblyum:StudyI}: the Beck-Chevalley condition for $F(C^\bullet_f)$ is satisfied by assumption, so that Proposition 7.2.2 (a) there carries over and shows that $F(Y)\to \lim F(C_f^\bullet)$ is the right adjoint of a monadic adjunction for the functor $(p_2)_+ p_1^+ : F(X) \r F(X)$, where $p_{i}: X \x_Y X \r X$, $i=1,2$ are the projections.
Again invoking the Beck-Chevalley condition, the proof of 7.2.2 (b) there carries over and shows that $f_+ : F(X) \rightleftarrows F(Y) : f^+$ is a monadic adjunction (using that $f^+$ is continuous and conservative).
Both monads are equivalent by invoking the assumption again.
\xpf

\theo
\label{theo--DM.descent}
The presheaves $\DM^*$ and $\DM^!$ are sheaves for the h-topology on $\Sch_S^\ft$, i.e., for an h-covering $f\co X \r Y$, the natural map
$$\DM(Y) \r \lim \DM^!(C^\bullet_f), \;\;\;\;\textup{(resp.}\;\;\DM(Y) \r \lim \DM^*(C^\bullet_f))$$
are equivalences, where $C^\bullet_f$ is the \v Cech nerve of $f$.
\xtheo

\pf
We apply \refle{descent.Beck.Chevalley}.
By \cite[Thm.~14.3.3]{CisinskiDeglise:Triangulated}, $f^*$ is conservative for any surjective map.
The same proof, dualized, shows the conservativity of $f^!$.
Since the h-topology is generated by Zariski coverings and proper coverings \cite[Thm.~8.4]{Rydh:Submersions}, it is enough to check the Beck-Chevalley condition in these cases separately.

We begin with $\DM^!$: in this case $f_!$ is left adjoint to $f^!$ for any separated map $f$ by \cite[Thm.~2.4.50]{CisinskiDeglise:Triangulated}.
If $f$ is proper, the Beck-Chevalley condition in \refle{descent.Beck.Chevalley} for $g'^!$ vs. $f'_! = f'_*$ holds true by base change.
If $f$ is a disjoint union of open embeddings, $f^! = f^*$ has the left adjoint $f_\sharp$, and again base change holds.
For the proper descent for $\DM^*$, we apply \refle{descent.Beck.Chevalley} to $\DM^\opp$, the opposite category. Then the Beck-Chevalley condition is satisfied by proper base change, i.e., the equivalence $g^* f'_* \r f''_* g'^*$ for $f$ proper.
\xpf

In \cite{Ayoub:Six1, Ayoub:Six2, CisinskiDeglise:Triangulated,Hoyois:Six,Khan:Motivic}, the existence and above-mentioned properties of $f_!$ and $f^!$ are stated for separated maps $f$ of finite type between schemes.
In \refth{f!.Artin}, which is applied in \refsect{intersection} to the stack of $G$-shtukas, we need the $!$-pushforward for Deligne-Mumford stacks locally of finite type.
The following application of Zariski descent allows to drop this hypothesis, similarly to \cite[Exam.~4.1.10]{LiuZheng:Enhanced}.

\prop
\label{prop--unseparated}
The adjunction \textup{\refeq{adjunction.shriek}} exists for any \textup{(}not necessarily separated\textup{)} map $f: X \r Y$ in $\Sch_S^\ft$ in a way such that properties iv\textup{)}, vi\textup{)}-x\textup{)}, xii\textup{)} and xvi\textup{)} in the above synopsis continue to hold.
\footnote{As for the weight structure in xv) we don't claim one can drop the separation hypothesis on $X$ since we do not know if de Jong's resolution of singularities can be extended to non-separated schemes.}
\xprop

\pf
Given $f$, we can pick open covers $X = \bigcup U_i$ and $Y = \bigcup V_j$ with $U_i, V_j \in \AffSch^\ft_S$ so that $f$ restricts to $U_i \r V_j$. The latter map is separated, since $U_i$ and $V_j$ are affine.
Let $X_\bullet$ and $Y_\bullet$ be the \v Cech nerves of these covers.
Each simplicial component $f_n$ of the map $f_\bullet : X_\bullet \r Y_\bullet$ is separated and of finite type.
\nts{If $U \r V$ is separated, then so is $U \x_X U \r V \x_Y V$: the monomorphism $U \x_X U \r U \x U$ is separated (true for any monomorphism), and $U \x U \r V \x U \r V \x V$ is also separated. Hence so is the map in question.}

By Zariski descent (\refth{DM.descent}), $\DM(X) = \lim_{\Delta^\opp} \DM^*(X_\bullet)$. (The * indicates that transition functors in the limit are *-pullbacks along the open immersions.)
Using that $(f_n)_!$ and $(f_n)^!$ commute with $j^*$ for an open immersion $j$, $f_!$ (and likewise for $f^!$) can be defined to be the unique functors so that the following diagram commutes:
$$\xymatrix{
\DM(X) \ar[d]^{f_!} \ar[r]^{j_n^*} & \DM(X_n) \ar[d]^{(f_n)_!} \\
\DM(Y) \ar[r]^{j_n^*} & \DM(Y_n).
}$$
Since (for a Noetherian scheme $X$) compactness of an object in $\DM(X)$ is a local condition on $X$, property vi) extends.
Property viii) extends since $j^*$ commutes with $\IHom$ \cite[Thm 2.4.50(5)]{CisinskiDeglise:Triangulated} and with $f_*$ and $f^*$.
The latter commutation also saves xii).
Part ix) carries over as one sees by using in addition that restriction along $\bigsqcup U_i \r X$ (for an open cover) creates colimits in $\DM$.
Likewise x) continues to hold by taking a covering, as above, of the map $g$ and considering its pullback along $f$.
Part xvi) again carries over by construction of $\D_\et(X, \Ql)$ in the non-separated case in \cite{LiuZheng:Enhanced} and the fact that $\rho_\ell$ commutes with $j^*$.
\xpf

\subsubsection{The $\ell$-adic realization}
\label{sect--realization}

With $S$ as in \refno{S}, fix a prime number $\ell$ which is invertible in $\calO_S$.
In order to discuss the $\ell$-adic realization we need to use a category of motives with $R$-coefficients, where $R$ is either $\Z$ or $\Z/\ell^n$ (as opposed to $R=\Q$ in \refsy{motives}).
By the work of Cisinski--Déglise (whose work extends similar results by Ayoub \cite{Ayoub:Realisation} to  more general base schemes), the bulk of the properties in \refsy{motives} also holds for the category $\DM_\h(X, R)$: i)--iii), v), xi) hold by construction in \cite[§5.1.1]{CisinskiDeglise:Etale},
iv), vii), ix), x), and xii) hold by \cite[Thm.~5.6.2]{CisinskiDeglise:Etale}.
As for vi), the identification of compact objects and constructible objects in $\DM(X, R)$ does hold by \cite[Thm.~5.6.4]{CisinskiDeglise:Etale} if (and by \cite[Rmk.~5.5.11]{CisinskiDeglise:Etale} also only if) the étale cohomological dimensions (for $R$-coefficients) of all residue fields of $X$ are finite. This excludes, say, $X = \Spec \Q$, which is a case we are interested in in this paper.
However, the following results in op.~cit. suggest the slogan that constructible motives should take over the role of compact objects: as for viii), the dualising functor $\Du$ is an involution \emph{for constructible h-motives} by \cite[Cor.~6.3.15]{CisinskiDeglise:Etale}, the remaining formulas in viii) hold (for arbitrary motives) by \cite[Thm.~5.6.2]{CisinskiDeglise:Etale}.
The identification of Hom-sets with K-theory in xiii) and the weight structure in xv) do not work for $R=\Z$ or $\Z/\ell^n$.
As for xiv), the proof of the h-descent property in \refth{DM.descent} carries over, the only remaining point being the conservativity of $f^*$ for surjective maps $f$. By \cite[Prop.~2.3.9]{CisinskiDeglise:Triangulated}, we only need to consider finite étale covers and finite surjective radicial maps $f$.
For the conservativity in the latter case use \cite[Prop.~6.3.16]{CisinskiDeglise:Etale}.
As is well-known, in the former case it holds by construction of $\DM_\h$: first embedd $\DM_\h$ in $\underline{\DM}_\h$ \cite[Def. 5.1.2]{CisinskiDeglise:Etale}.\nts{This commutes with $f^*$.}
which is defined as a Bousfield localization (implementing the $\P$-stabilisation) of a category of modules, denoted $R\mathrm{-mod}$, in symmetric sequences \cite[§7.2, §5.3.d]{CisinskiDeglise:Triangulated}
The left adjoint $f_\sharp$ of $f^*$ is a left Quillen functor with respect to the stable model structure on the module category $R\mathrm{-mod}$, hence the derived right adjoint of the Bousfield localization commutes with $f^*$.
The same argument applies for the $\A^1$-localization.
We hence reduce to modules in symmetric sequences of complexes of h-sheaves. The forgetful functor forgetting the module structure is again conservative and commutes with $f^*$, so we end up stating the conservativity for $f^*$ (for finite étale maps) on $\Shv_\h(\Sch / X)$, which is clear, noting that any h-sheaf is in particular an étale sheaf and vanishing of étale sheaves can be tested stalkwise.

By \cite[(5.4.1.c), Cor. 5.5.4, Thm.~6.3.11]{CisinskiDeglise:Etale}, the étale realization functor for $\ell$-torsion coefficients $\DM_{\h, \cstr}(X, \Z) \r \D^\bound_\cstr(X, \Z/\ell^n)$ takes values in the subcategory (always meant to be an \ii-category) of complexes of finite Tor-dimension with cohomology sheaves which are constructible in the sense of \cite[Exp.~IX, Déf.~2.3]{SGA4:3} and vanish in almost all cohomological degrees.
(At least if $\ell$ is odd and the $\ell$-cohomological dimension of all residue fields of $S$ is finite, it also results from \cite[Thm.~9.7]{Ayoub:Realisation} using the equivalence in \cite[Cor.~5.5.7, see also Rem.~5.5.8]{CisinskiDeglise:Etale}.)
The category $\D^\bound_\cstr(X, \Zl)$ is defined as the full subcategory of $\D(X, \Zl) := \lim \D(X, \Z/\ell^n)$ \cite[§0.1, §1.3]{LiuZheng:EnhancedAdic} of objects which are of finite Tor-dimension, bounded and constructible in each $\Z/\ell^n$-degree \cite[Def.~6.3.1]{LiuZheng:Enhanced}, \cite[7.2.18]{CisinskiDeglise:Etale}.
Taking the limit over $n$, the above realization functor takes values in $\D^\bound_\cstr(X, \Zl)$ \cite[Thm.~6.9]{Ayoub:Realisation}.
With rational coefficients, the categories $\DM_{\mathrm h}^\comp$ and $\DM^\comp$ as sketched in the beginning of \refsect{DM.schemes} are equivalent by \cite[Thm.~5.2.2, Cor.~5.5.7]{CisinskiDeglise:Etale}, so the rationalization of the $\Zl$-linear realization gives a functor $\DM(X)^\comp \to \D^\bound_\cstr(X, \Ql)$ \cite[Rem.~7.2.25]{CisinskiDeglise:Etale}.
Taking the ind-completion, we arrive at a functor
$$\rho_\ell\co \DM(X, \Q) = \Ind(\DM(X, \Q)^\cstr) \to \Ind(\D^\bound_\cstr(X, \Ql)) =: \D(X_\et, \Ql).$$

\theo
\label{theo--D.et!.sheaf}
Suppose $S$ satisfies the assumptions in \refno{S}. Then the presheaf
$$\D_\et^!(-, \Ql) : (\Sch_S^\ft)^\opp \to \Cat_\infty, X \mapsto \D_\et(X, \Ql), f \mapsto f^!$$
is a sheaf for the h-topology and likewise for $\D^\bound_\cstr(-, \Ql)$, $\D^\bound_\cstr(-, \Zl)$ and $\D^\bound_\cstr(-, \Z/\ell^n)$.
\xtheo

\pf
Note that $f^!$ preserves bounded constructible complexes, so the statement makes sense to begin with.
By Gabber's work \cite[Exp.~XIII, Thm.~4.2.3]{IllusieLaszloOrgogozo}, $f_*$ and therefore also $f_!$ preserves bounded constructible complexes for any map $f$ in $\Sch_S^\ft$.
This allows us to adapt the proof of \refth{DM.descent} to the presheaf given by $\D^\bound_\cstr(X, \Ql)$:
for a surjective map $f$, the conservativity of $f^!$ is reduced to $\D^\bound_\cstr(X, \Zl)$ and then to $\Z/\ell^n$-coefficients.
There, it is reduced to the conservativity of $f^*$ using that $f^! \Du = \Du f^*$, where $\Du$ denotes the dualizing functor, which is an involution on constructible complexes (\cite[Exp.~XVII, Thm.~0.9]{IllusieLaszloOrgogozo}, note that the assumptions on the base scheme there are weaker than those in \refno{S}).
The conservativity of $f^*$ holds since isomorphisms of étale sheaves are detected stalkwise at each geometric point.
The Beck-Chevalley condition holds, by proper and smooth base change (again, first for $\Z/\ell^n$-coefficients, which formally implies the one for $\Zl$-, and then $\Ql$-coefficients).

This shows the h-descent property for $\D^\bound_\cstr (X, \Ql)$.
In other words (see the proof of \refle{descent.Beck.Chevalley}), for an h-covering map $f: X \r Y$, $\D^\bound_\cstr(Y, \Ql) = \Mod_T (\D^\bound_\cstr(X, \Ql))$, where $T$ is the monad on the sheaf category on $X$ given by $f^! f_!$. The functor $f^! (:= \Ind(f^!)) : \Ind \D^\bound_\cstr(Y, \Ql) \r \Ind \D^\bound_\cstr(X, \Ql)$ is conservative by \refle{monadic.Ind}.
The Beck-Chevalley condition for the functors on the ind-completed categories follows formally from the one on $\D^\bound_\cstr(-, \Ql)$.
\xpf

\lemm
\label{lemm--monadic.Ind}
Let $T$ be a monad on an \ii-category $\calD$. Then
the right adjoint $\tilde U$ of the adjunction
$$\tilde F: \Ind \calD \rightleftarrows \Ind \Mod_T(\calD) : \tilde U$$
(obtained by applying $\Ind$ to the free-forgetful adjunction) is conservative.
\xlemm

\pf
It is enough to show that the image of $\tilde F$, which is the Ind-extension of the free $T$-module functor $F$, generates $\Ind \Mod_T(\calD)$ under colimits.
Indeed, any object in $\Ind \Mod_T(\calD)$ is the (filtered) colimit of objects in $\Mod_T(\calD)$, and any object $d \in \Mod_T(\calD)$ is the colimit of the diagram $\dots \rightrightarrows (FU)^2(d) \rightrightarrows FU(d)$.
\xpf

\subsection{Motives on prestacks}\label{sect--DM.prestacks}

In order to deal with motives on $L^+G\backslash LG/L^+G$, i.e., the double quotient of the loop group by the positive loop group, it is convenient to define motives on a large class of geometric objects, namely prestacks.
We follow the method proposed by Gaitsgory and Rozenblyum in the context of quasi-coherent and ind-coherent sheaves \cite[Ch.~3, 0.1.1]{GaitsgoryRozenblyum:StudyI} and by Raskin for $D$-modules \cite{Raskin:D-modules}.
Derived categories of sheaves on more general base `spaces' have also been used by Scholze \cite{Scholze:Etale} in his work on a six functor formalism for the derived category of étale sheaves on small v-stacks.

\subsubsection{Definitions and first properties}

Let $\AffSch_S^{\ft}$ be the category of schemes $X$ which are affine, and equipped with a map $X \r S$ of finite type.
Such schemes $X$ are Noetherian,
of finite Krull dimension,\nts{use $\dim R[x]=\dim R+1$ for Noetherian local ring}
and any map $X \r Y$ in $\AffSch^\ft_S$ is separated and of finite type.
Any (not necessarily affine) scheme $X \r S$ which is locally of finite type has a Zariski covering by objects in $\AffSch_S^\ft$.
The category $\AffSch_S^\ft$ is essentially small. Throughout, we will replace this category by a small skeleton containing the objects of interest to us. Fixing some regular cardinal $\kappa$, we embed $\AffSch_S^\ft$ into two larger \ii-categories:
$$\AffSch_S^\ft \subset \AffSch_S^\kappa \subset \PreStk_S^\kappa := \Fun((\AffSch_S^\kappa)^\opp,\text{\ii-}\Gpd).$$
The (ordinary) category in the middle consists of affine $S$-schemes which are obtained as $\kappa$-filtered limits of objects in $\AffSch_S^\ft$.
Equivalently, $\AffSch_S^\kappa = \Pro_\kappa (\AffSch_S^\ft)$ is the $\kappa$-pro-completion.
Again, this category is small.
Finally, the category of \emph{prestacks} on the right is the presheaf category in the \ii-sense \cite[§5.1]{Lurie:Higher} on this category.
It contains the category of affine schemes via the Yoneda embedding, and is generated under colimits by them, cf.~\cite[Cor.~5.1.5.8]{Lurie:Higher}.

We now define motives on prestacks adapting \cite[Def.~2.3]{Raskin:D-modules}.
Following \cite[Ch. 1, 0.6.11]{GaitsgoryRozenblyum:StudyI}, we write $\DGCat_\cont$ for the \ii-category of presentable, stable $\Q$-linear dg-\ii-categories with continuous \textup{(}i.e., colimit-preserving\textup{)} functors.

\defi
\label{defi--DM.prestacks}
The functor
$$\DM^!_\kappa\co (\AffSch^\kappa_S)^\opp \r \DGCat_\cont$$
is the left Kan extension of the functor $\DM^!_\kappa\co (\AffSch_S^{\ft})^\opp \r \DGCat_\cont$ mentioned in \refsy{motives}, ii).
The functor
$$\DM^!_\kappa\co (\PreStk_S^\kappa)^\opp \r \DGCat_\cont$$
is the right Kan extension of this functor along the Yoneda embedding.
\xdefi

The Kan extensions exist by \cite[Thm.~5.1.5.6]{Lurie:Higher}
since $\DGCat_\cont$ is bicomplete \cite[§I.1, Cor.~5.3.4]{GaitsgoryRozenblyum:StudyI},
and the category $\AffSch^\kappa_S$ is small (this is the purpose of the size restriction using $\kappa$). Choosing a larger cardinal $\kappa' > \kappa$, the restriction of $\DM^!_{\kappa'}$ on $(\PreStk^{\kappa'}_S)^\opp$ to $(\PreStk^{\kappa}_S)^\opp$ is equivalent to $\DM^!_{\kappa}$.
In this sense, the choice of $\kappa$ does not matter as long $\kappa$ is large enough so that $\AffSch_S^\kappa$ contains all affine schemes of interest to us.
For the purposes of this paper, it is enough to choose $\kappa$ to be the countable cardinal, since the only non-finite type objects we encounter are countably indexed. We now fix $\kappa$ throughout the document, and drop it from the notation.

\rema\label{rema--DM.prestacks}
i) For a prestack $X$, $\DM(X) = \lim_{T \r X} \DM^!(T)$, where the limit is over the category of $S$-maps $T\to X$ for any $T \in \AffSch_S$.
Thus, a motive $M$ on $X$ can be thought of system of motives $M_f$ for any $T \stackrel f \r X$, compatible under !-pullback (and these compatibilites are subject to higher coherence conditions).
Next, $\DM(T) = \colim_{T \r T'} \DM^!(T')$, where the colimit is over the category of $S$-maps $T\to T'$ for any $T' \in \AffSch_S^\ft$.
The (co)limits are taken in $\DGCat_\cont$ and are formed using  $!$-pullbacks as transition functors.
The inclusion $\DGCat_\cont\subset \Cat_\infty$ preserves limits, so that the limit above, and also the one in \refth{DM.descent} can be taken in $\DGCat_\cont$ or $\Cat_\infty$.
The colimit is taken in $\DGCat_\cont$. It can be computed as
$$\colim_{\DGCat_\cont} \DM(T') \;=\; \Ind (\colim_{\Cat_\infty} \DM(T')^\comp).\eqlabel{colim.DGCat}$$
This is a special case of
\cite[Ch.~1, Cor.~7.2.7]{GaitsgoryRozenblyum:StudyI}, according to which
$$\colim_{\DGCat_\cont} \Ind (\C_i) = \Ind (\colim_{\Cat_\infty} \C_i).\eqlabel{colim.DGCat.general}$$
Thus, a compact object $M \in \DM(T)$ can be thought of as being pulled back from some object in $\DM(T')$ for some $T \r T'$ with $T' \in \AffSch_S^\ft$.
\smallskip\\
ii) By construction, there is a functor $f^!\co \DM(Y) \r \DM(X)$ for \emph{any} map of prestacks.\smallskip\\
iii\textup{)} Let $X = \colim_{i \in I} X_i$ be a colimit of prestacks.
(Here and elsewhere, unless otherwise mentioned, (co)limits are meant in the \ii-categorical sense, also called homotopy (co)limits.) The universal property of the category of \ii-presheaves \cite[Thm.~5.1.5.6]{Lurie:Higher} yields an equivalence
$$\DM(X) \;=\; \lim_{i \in I} \DM^!(X_i).$$

\noindent iv\textup{)} It is possible to left/right Kan extend the presheaf $\DM^*$ in \refsy{motives}, ii) to $\PreStk_S$.
However, the computation of motives on an ind-scheme or more generally an ind-Artin stack $X$ in \refco{Ind.Artin.co.limit} only works for $\DM^!$. We will therefore rarely consider the presheaf $\DM^*$, and will soon just write $\DM$ instead of $\DM^!$.
There are, however, cases of prestacks $X$ where $\DM^*(X)$ and $\DM^!(X)$ are equivalent.
To give just one example, consider $X = \A^\infty_S = \lim \big (\dots \stackrel p \r \A^2_S \stackrel p \r \A^1_S \stackrel p \r S \big)$, where the maps are the standard projection maps.
Then $\bbA^\infty_S\to S$ is an $S$-affine scheme which is pro-($S$-smooth).
There is an equivalence
$$\DM^! (\A^\infty_S)^\comp = \colim_{\Cat_\infty} \left (\DM(S) \stackrel{p^!} \r \DM(\A^1_S) \stackrel{p^!} \r \dots \right ),$$
and likewise with ! exchanged by * throughout. Using the natural equivalence $p^* (-) \t p^! 1 \stackrel \cong \Rightarrow p^! (-)$ and the $\t$-invertibility of $p^! 1 = 1(1)[2]$, we get an equivalence of the above with $\DM^*(\A^\infty)^\comp$, which implies the claim by passing to the ind-completion.
\smallskip\\
v) For $X \in \Sch_S^\ft$, $\DM(X)$ as recalled in \refsy{motives} agrees with $\DM(X)$ as given by \refde{DM.prestacks}: if $X$ is affine, this is tautological, in general it holds by Zariski descent.
Likewise the a priori ambiguity for $\DM(X)$ for $X \in \AffSch_S$ Noetherian and of finite Krull dimension (but not necessarily of finite type over $S$) is harmless, since $\DM(X)$ as discussed in \refsy{motives} is equivalent to the $\DM(X)$ obtained by \refde{DM.prestacks}. This is a consequence of the continuity of $\DM$ (on finite type $S$-schemes) \cite[Prop.~14.3.1]{CisinskiDeglise:Triangulated} and \refre{DM.prestacks} i).
\smallskip\\
vi)
Let $X : I \to \Sch$ be a diagram of (Noetherian, finite-dimensional) schemes.
Ayoub \cite{Ayoub:Six1} constructs the stable homotopy category $\SH(X)$.
Ayoub's approach is based on the category $\Sm / X$ which he defines as the \emph{lax} colimit of the categories $\Sm / X_i$ ($i \in I$).
We did not investigate the precise relation, but one might speculate that once we replace the lax colimit by the colimit and run the usual steps in the construction of $\SH$ (taking presheaves, $\A^1$- and Nisnevich localization, $\P$-stabilization), we would get a category equivalent to $\SH^*(X_i) := \colim \SH(X_i)$, where the transition functors are given by the *-pullbacks along the maps in the diagram $X$.
\xrema

\rema
\label{rema--prestacks.examples}
Any functor $(\AffSch_S)^\opp \r \Sets$ is a prestack by regarding it as a simplicially constant presheaf.
This includes schemes, ind-schemes (for which see \refsect{DM.ind-schemes}), and also (ind-)algebraic spaces.
Likewise, any functor $(\AffSch_S)^\opp \r \Gpd$ to classical groupoids is a prestack by identifying classical groupoids with 1-truncated spaces.
This includes (ind-)Deligne-Mumford or (ind-)Artin stacks.
More generally, $n$-geometric stacks \cite[§1.3]{ToenVezzosi:HomotopicalII} (and again, their ind-variants) are prestacks as well.
\xrema

\subsubsection{Equivariant motives}
The category $\PreStk=\PreStk_S$ carries a cartesian monoidal structure.
Group objects in $\PreStk$ in the sense of \cite[Def.~7.2.2.1]{Lurie:HA} are referred to as \emph{group prestacks}.
The monoidal \ii-category $\PreStk$ is the underlying \ii-category of the model category $\calP \defined \sPSh(\AffSch_S)$ of simplicial presheaves on $\AffSch_S$, equipped with the injective model structure and the pointwise product.
This forms a combinatorial, symmetric monoidal monoidal model category in which all objects are cofibrant.
Thus for a non-symmetric colored operad $O$ in $\calP$, the \ii-category of $O$-algebras in the \ii-sense is presented by the model category of $O$-algebras in the strict sense \cite[Thm.~2.15]{Haugseng:Rectification}.
In particular, any group object $G \in \PreStk$ in the sense of \cite[Def.~7.2.2.1]{Lurie:HA} can be strictified and, for given $G$, any action of $G$ on some $X \in \PreStk$ can also be strictified.

The only group prestacks we need in this paper (e.g.~the positive loop group $L^+G$), are ordinary presheaves of (discrete) groups, which are regarded as simplicially constant prestacks, and therefore group objects in $\PreStk$. However, other interesting examples such as the Picard groupoid operating on the moduli stack of vector bundles on a curve, can also be treated using the notion of group prestacks.

\defi
\label{defi--DM.G}
Suppose a group prestack $G$ acts on the left on a prestack $X$.
We write $G \backslash X$ for the homotopy colimit of the $G$-action or, equivalently, the image of $X$ under the left adjoint to the functor $\PreStk \r \PreStk_G$ which equips a prestack with a trivial $G$-action.
The category of \emph{$G$-equivariant motives on $X$} is defined as
$$\DM(G \backslash X).$$
\xdefi

It is well-known (see, e.g., \cite[Exam.~IV.1.10]{GoerssJardine:Simplicial}) that $G \backslash X$ can be computed as the (homotopy) colimit of the  {\it bar construction} $\BarC(G, X)$ which is the semi-simplicial prestack built out of action maps and projection maps:
$$\BarC(G, X) \defined \left(\xymatrix{
	\dots \ar@<-1.5ex>[r] \ar@<-0.5ex>[r] \ar@<0.5ex>[r] \ar@<1.5ex>[r]
        &
	G \x_S G \x_S X
	\ar@<-1ex>[r]
	\ar[r]
	\ar@<1ex>[r]
	&
	G \x_S X
	\ar@<0.5ex>[r]^{\phantom{hh}a}
	\ar@<-0.5ex>[r]_{\phantom{hh}p_2}
	&
	X
	}\right).$$
This characterization of $G \backslash X$ and \refre{DM.prestacks} iii) yield the following result:

\lemm
\label{lemm--DM.G.BarC}
If a group prestack $G$ acts on a prestack $X$, there is an equivalence
$$\DM(G \backslash X) \stackrel \cong \lr \lim \DM^!(\BarC(G, X)). \eqlabel{Bar.descent}$$
\xlemm
\qed

Colloquially speaking, a $G$-equivariant motive is a collection of objects $M_n \in \DM(G^n \x_S X)$, $n\in \Z_{\geq 0}$ together with equivalences $a^! M_0 \stackrel \simeq \lr M_1$, $p_2^! M_0 \stackrel \simeq \lr M_1$ etc.; one for any map in the bar construction, subject to the natural compatibility conditions induced by the relations between the maps.

As was indicated in \refre{DM.prestacks}.iv), it is possible to apply the same considerations to the presheaf $\DM^*$ instead of $\DM^!$, in which case a $G$-equivariant motive would amount to specifying equivalences $a^* M_0 \cong M_1$ etc.
If $G$ is a smooth algebraic group, the two notions agree, up to equivalence, by the relative purity isomorphism \refeq{relative.purity}, cf.~also the discussion in \refre{classical.equivariant}.

The next result shows that the existence of adjoint functors on prestack quotients can be checked on the prestacks themselves.

\lemm
\label{lemm--functoriality.equivariant}
Let $G$ be a group prestack, and let $f\co X \r Y$ be a $G$-equivariant map of prestacks.
Let $\ol f\co G \backslash X \r G \backslash Y$ be the induced map.
If $f^!\co \DM(Y) \r \DM(X)$ has a left \textup{(}resp.~right adjoint\textup{)}, then the same is true for $\ol f^! \co \DM(G \backslash Y) \r \DM(G \backslash X)$.
This process can be iterated, for example if the left adjoint $f_!$ of $f^!$ has another left adjoint, the same holds for the left adjoint $\ol f_!$ of $\ol f^!$.
\xlemm

\pf
The categories $\DM(\str)$ being presentable, we can use the adjoint functor theorem \cite[Cor 5.5.2.9]{Lurie:Higher} to construct adjoints.
The forgetful functor $\DM(G \backslash X) \r \DM(X)$ is conservative and preserves (co)limits and therefore also creates them.
Thus, the preservation of (co)limits of a functor between categories of $G$-equivariant motives can be checked after forgetting the $G$-equivariance, so that the existence of the adjoints in the non-equivariant setting shows the claim.
\xpf

Totaro \cite{Totaro:Chow, Totaro:Motive} and Edidin--Graham \cite{EdidinGraham:Equivariant} have introduced \emph{equivariant higher Chow groups}.
We now show that the above equivariant category of motives reproduces these groups (with rational coefficients).
To focus on the essential point, suppose $X$ is a smooth finite type scheme of dimension $n$ over a field $k$, equipped with an action of a smooth affine $k$-group scheme $G$ so that the one of the assumptions (1)--(3) in \cite[Prop.~23]{EdidinGraham:Equivariant} is satisfied.
Let $s, t \in \Z$.
These assumptions ensure that we may choose a $G$-representation $V$ (viewed as an affine $k$-scheme) and an open subscheme $j\co U \subset V$ on which $G$ acts freely such that the reduced complement $i\co Z \r V$ satisfies $c := \codim_X Z > s$ and such that the quotient $(X \x U) / G$
exists in the category of schemes.
Let $l := \dim V/k$, $g := \dim G/k$ denote the Krull dimensions.
In this case, the definition in op.~cit. is in terms of Bloch's higher Chow groups
$$\CH_{n-s}^G(X, t) \defined \CH_{n+l-s-g}((X \x U) / G, t).$$

\theo
\label{theo--equivariant.Chow}
In the above situation, there is an isomorphism
$$\Hom_{\DM(X/G)}\big(1, 1(s)[2s-t]\big) = \CH_{\dim X-s}^G(X, t) \t_\Z \Q, \ s, t \in \Z.$$
\xtheo

\pf
The freeness assumption ensures that the (homotopy) quotient $(X \x U) / G$ computed in the \ii-category of prestacks agrees with the ordinary quotient.
Let $p\co X \x V \r V$ be the projection.
Using \refle{functoriality.equivariant} (including the notation $\ol i := (i / G)$ for the quotients etc.), there is a localization cofiber sequence $\ol i_* \ol i^! \r \id \r \ol j_* \ol j^!$.
By homotopy invariance, $p^*$ and therefore $\ol p^*$ is fully faithful.
In order to show that
$$\Hom_{X/G}\big(1, 1(s)[2s-t]\big) \stackrel[\cong]{\ol p^*} \r \Hom_{(X \x V)/G}\big(1, 1(s)[2s-t]\big) \stackrel{\ol j^*} \r \Hom_{(X \x U) / G}\big(1,1(s)[2s-t]\big)$$
is an isomorphism, we need to show that the groups $\Hom_{(X \x Z)/G}\big(1, \ol i^! 1(s)[r]\big)$ vanish for all $r \in \Z$.
It is enough to show the same vanishing in $\DM(X \x Z \x G^a)$ for all $a \ge 0$.
Since the regular locus in $Z$ is nonempty and open \StP{07R5}, there exists by Noetherian induction a stratification of $Z$ by regular subschemes (whose codimension in $V$ can only grow), and hence we may assume $Z$ to be regular (use induction and the localization sequence).
In this case we have $i^! 1 = 1(-c)[-2c]$ by absolute purity
and the group $\Hom_{\DM(X \x Z \x G^a)}\big(1, 1(s-c)[r-2c]\big)$ vanishes since $s-c < 0$.
As $X \x U$ is regular, so is $(X \x U) / G$.
We then conclude using $\Hom_{X \x U / G}\big(1, 1(s)[2s-t]\big) = \CH^s(X \x U / G, t) \t \Q = \CH_{n + l - g - s}(X \x U / G, t) \t \Q$.
\xpf

For a map of group prestacks $\pi\co H \r G$ and $G$ acting on a prestack $X$ (hence $H$ acts via $\pi$), there is a \emph{restriction functor} $\pi^!\co \DM(G \backslash X) \r \DM(H \backslash X)$.
In the above description, this functor sends a $G$-equivariant motive $(M_n)$ to the $!$-pullback along $H^{\x_S n} \x_S X \r G^{\x_S n} \x_S X$ in the $n$-th component.

\nts{It is true, but not needed that in the finite type case $\BarC(G, X)$ is the \v Cech nerve of the smooth cover $X \r G \setminus X$.
We recall this well-known fact for the slightly simpler case of simplicial sets (not simplicial presheaves).
To verify $X \x^h_{G \setminus X} X = G \x X$ we replace $X$ by the weakly equivalent $EG \x X$.
Here $EG$ is a contractible space which is projectively cofibrant as a $G$-simplicial set.
According to \cite[Lemma V.2.4]{GoerssJardine:Simplicial}, each simplicial term $(EG)_n$ is a free $G$-set.
We can then compute the homotopy quotient $G \setminus X$ as the ordinary quotient $EG \x^G X$.
Moreover, for a simplicial set $S$ (such as $S = EG \x X$) on which $G$ acts freely in each simplicial component, the map $S \r S/G$ (ordinary quotient) is a fibration \cite[Cor.~V.2.7]{GoerssJardine:Simplicial}, so the homotopy fiber $X \x^h_{G \setminus X} X = (EG \x X) \x^h_{EG \x^G X} (EG \x X) = S \x^h_{S/G} S$ can be computed as an ordinary fiber product.
Again using the termwise freeness of the $G$-action, the natural map $G \x S \r S \x_{S/G} S$, $(g,s) \mapsto (s, gs)$ is an isomorphism of simplicial sets.
This finishes the computation of $X \x^h_{G \setminus X} X$, the ones for the higher terms in the bar complex are identical.
}

The next result serves to cut down the size of $G$ from certain pro-algebraic to algebraic groups.
For our conventions on strictly pro-algebraic groups, we refer the reader to \refsect{algebraic.grps}.

\prop
\label{prop--DM.G.homotopy.invariant}
Let $G = \lim G_i$ be a strictly pro-algebraic $S$-group such that $U:=\ker (G \stackrel \pi \r G_0)$ is split pro-unipotent \textup{(}by \refde{unipotent} a possibly countably infinite successive extension of vector groups\textup{)}.
Suppose $G_0$ acts on an $S$-scheme $X$.
Then the restriction functor $\DM(G_0\backslash X) \r \DM(G\backslash X)$ is an equivalence.
\xprop

\pf
We apply \refle{lim.equivalence} below to $\pi^!\co \DM^!(\BarC(G_0, X)) \r \DM^! (\BarC(G, X))$ which in the $m$-th simplicial level is the $!$-pullback along $\pi^{\x m}\co G^{\x m} \x X \r G_0^{\x m} \x X$.
Replacing $G_0$ by $G_0^m \x X$ (and $S$ by $X$) etc., it remains to show that $\pi^!\co \DM(G_0) \r \DM(G)$ is fully faithful.
By assumption, the map $\pi\co G\to G_0$ is a torsor under the split pro-unipotent group $U$.
Since the claim is Zariski local on $S$, we may assume that this torsor is trivial by \refpr{unipotent}, so that the map $\pi$ is on the underlying schemes isomorphic to the projection $G_0\x U\to G_0$.
Replacing $G_0$ by $S$, it remains to show that $\pi^!\co \DM(S) \r \DM(U)$ is fully faithful.
By \refde{unipotent}, we can write $U=\lim U_i$ where $\ker(U_{i+1}\to U_i)=\bbV(\calE_i)$ is a vector group for some $S$-vector bundle $\calE_i$, see \refeq{VE} for notation.
By \refde{DM.prestacks}, we have $\DM(U)=\colim \DM(U_i)$, and $\pi^!$ is the canonical functor into this colimit.
Arguing as before, each transition functor is the $!$-pullback from a vector bundle, and thus fully faithful by homotopy invariance of $\DM$, cf.~\refsy{motives}, xi).
Using the equivalence $\colim_{\DGCat_\cont} \DM(U_i) = \Ind (\colim_{\Cat_\infty} \DM(U_i)^\comp)$ from \refre{DM.prestacks} i), the full faithfulness of $\pi^!$ follows from the one of $\DM(S)^\comp \r \colim_{\Cat_\infty}(\DM(U_i)^\comp)$, which in turn follows from the description of the filtered colimit of \ii-categories in \cite{Rozenblyum:Filtered} and the full faithfulness of the $\pi_i^!$.
%
\xpf

\lemm\textup{(\cite[Lem.~B.6]{BunkeNikolaus:Twisted})}
\label{lemm--lim.equivalence}
Let $C, C': \Delta \r \Cat_\infty$ be two cosimplicial \ii-categories. Let $F: C \r C'$ be a natural transformation between them, such that each $F_n$ is fully faithful and $F_0$ is an equivalence.
Then $\lim_\Delta C \r \lim_\Delta C'$ is an equivalence.
\nts{Proof idea if we additionally assume that $F_n$ have a right adjoint $G_n$: $F_n$ being fully faithful means $\id \gets G_n F_n$ is an equivalence.
We must show that $F_n G_n A \gets A$ is an equivalence for any object $A \in C(n)$ which is of the form $A' = a'(A'_0)$, where $a': C'(0) \r C'(n)$ is induced by a(ny) map $[0] \r [n]$. By assumption we may choose $A' = F_0(A)$.
So we indeed get $F_n G_n a'(A'_0) = F_n G_n F_n (A) = F_n(A) = a'(A'_0)$.}
\xlemm

\exam \label{exam--affine.proj}
i) For a split unipotent $S$-group scheme $U$ (i.e., a smooth affine $S$-group scheme which is a successive extension of vector groups) acting trivially on $X$, \refpr{DM.G.homotopy.invariant} shows that $\DM_{U}(X)= \DM(X)$. \smallskip\\
ii) Proposition \ref{prop--DM.G.homotopy.invariant} applies to any parahoric subgroup $\calP=\on{lim}_i\calP_i$ of $LG$ by \refle{affine.proj} below. More generally, it applies to any pro-algebraic group which is constructed as a positive loop group along some Cartier divisor, cf.~\refpr{vector.extension} and \refex{groups} below.
\xexam

\subsubsection{Descent for motives}
We now study consequences of descent.
Let $\tau$ be a Grothendieck topology on $\AffSch_S$ such that each covering family has a refinement by a finite covering family.
We now fix $S$, and write $\Stk^\tau=\Stk_S^\tau$ for the category of $\tau$-stacks. By definition \cite[Def.~6.2.2.6]{Lurie:Higher}, this is the full subcategory of $\PreStk$ of objects $X\co (\AffSch_S)^\opp\r \infty\str \Gpd$ which commute with finite coproducts, and such the natural map
\begin{equation}\label{sheaf_condition}
X(T)\;\r\; \lim \left( \xymatrix{X(U) \ar@<-0.5ex>[r] \ar@<0.5ex>[r] & X(U\x_T U) \ar@<-1ex>[r] \ar[r] \ar@<1ex>[r]  &  X(U\x_T U\x_T U) \ar@<-0.5ex>[r] \ar@<0.5ex>[r] \ar@<-1.5ex>[r] \ar@<1.5ex>[r] &  \dots} \right)
\end{equation}
is an equivalence for all $\tau$-covers $U\r T$ in $\AffSch_S$. We denote the \emph{sheafification} (or localization) functor
$\PreStk \r \Stk^\tau$ by $X \mapsto X^\tau$ which is left adjoint to the inclusion $\Stk^\tau\subset \PreStk$, cf.~\cite[Prop.~6.2.2.7]{Lurie:Higher}.

We are interested in topologies $\tau$ contained in the $h$-topology for which we recall the definition from \cite[Def.~8.1]{Rydh:Submersions}.

\defi The $h$-topology on $\AffSch_S$ (resp.~on $\Sch_S$) is the minimal Grothendieck topology generated by the following covering families
\begin{itemize}
\item Families of open immersions $\{U_i \stackrel {f_i}\r T\}$ such that $T=\cup f_i(U_i)$.
\item Finite families $\{U_i\stackrel {f_i} \r T\}$ such that $\sqcup f_i\co \sqcup U_i\r T$ is {\em universally subtrusive} (i.e., a $v$-cover) and of finite presentation.
\end{itemize}
\xdefi

By \cite[Thm.~8.4]{Rydh:Submersions}, any $h$-covering $\{U_i\r T\}$ of a quasi-compact scheme can be refined by a finite covering family $\{V_i\r T\}$. Thus, as each object in $\AffSch_S$ is quasi-compact, the $h$-topology on $\AffSch_S$ is generated by finite covering families.


\theo
\label{theo--descent.prestacks}
Suppose $\tau$ is a Grothendieck topology which is contained in the $h$-topology, and such that each covering family admits a refinement by a finite covering family \textup{(}e.g., the fppf, étale or Zariski topology, but also the qfh, cdh or h topology\textup{)}.
For any prestack $X\in \PreStk_S$, $!$-pullback along the natural map $X\to X^\tau$ yields an equivalence
$$\DM(X^\tau) \stackrel \cong \lr \DM(X).\eqlabel{DM.tau}.$$
\xtheo

\pf
Since the map $X\to X^\tau$ is an equivalence after $h$-sheafification, we may assume that $\tau$ equals the $h$-topology. Now let $T\in \AffSch_S$, and suppose that $u: U\r T$ is an $h$-cover of schemes with $U$ being quasi-compact. We claim that the natural map
$$
\DM(T)\;\r\; \lim\DM(C^{\bullet}_u)\eqlabel{claim1}
$$
is an equivalence, where $C^\bullet_u$ is the \v Cech nerve of $u$. If $T$ (and hence $U$) is of finite type over $S$, this is a special case of \refth{DM.descent}.
Now the proof in \cite[Prop.~3.12.1]{Raskin:D-modules} carries over. We briefly recall Raskin's argument: as $U\r T$ is an $h$-cover of qcqs schemes, one may treat the case of proper surjective morphisms, and finitely presented Zariski coverings separately, cf.~\cite[Thm.~8.4]{Rydh:Submersions}.
One shows that $u^!$ has a left adjoint $u_!$ (for $u$ proper), respectively a right adjoint denoted $u_*$ (for $u$ being a quasi-compact Zariski covering map).
The existence of these adjoints follows from proper, resp. smooth base change for finite type schemes.
Moreover, these adjoints $u_!$ (resp. $u_*$) satisfy base change with respect to $f^!$ for arbitrary maps of qcqs schemes.
Then, the Beck-Chevalley type argument already used in the proof of \refth{DM.descent} finishes the proof of \refeq{claim1}.

The universal property of the sheaf category $\Stk^\tau_S$ \cite[Prop.~5.5.4.15, Def.~6.2.2.6]{Lurie:Higher} states that it is the localization of $\PreStk_S$ with respect to maps $\colim C^\bullet_u \r T$ from the \v Cech nerve of $h$-coverings $u$ as above, and with respect to finite coproducts $\coprod j(T_i) \r j(\coprod T_i)$, where $j\co \AffSch_S \r \PreStk_S$ is the Yoneda embedding and $T_i \in \AffSch_S$.
By \refeq{claim1}, the functor $\DM^!\co \PreStk_S^\opp \r \DGCat_\cont$ sends the map $\colim C^\bullet_u \r T$ to equivalences.
Moreover, $\DM$ sends (finite) coproducts of affine schemes to products, so that $\DM$ factors over the sheafification $X \mapsto X^\tau$.
\xpf

\exam
Let $f \co X \r Y$ be a schematic $h$-covering of prestacks, i.e., for any affine scheme $T \r Y$, $X \x_Y T$ is a scheme, which is an $h$-cover of $T$ in the usual sense.
Let $C_f^\bullet$ be the \v Cech nerve of $f$.
Then there is an equivalence$$\DM^!(Y) \stackrel \cong \lr \lim \DM^!(C_f^\bullet). \eqlabel{DM.Cech}$$
Indeed, this follows from \refre{DM.prestacks} iii) and \refth{descent.prestacks} using that $\colim C_f^\bullet\r Y$ is an equivalence after $\tau$-sheafification.
\xexam

For later use we record the following lemma.

\lemm
\label{lemm--DM.G/H}
Let $\tau$ be a Grothendieck topology as in \refth{descent.prestacks}.
Let $H\subset G$ be an inclusion of group $\tau$-stacks, and consider the quotient $X := (G/H)^\tau$.
Let $e : S / H \r G \backslash X$ be the map in $\PreStk$ induced by the base point $S \r X$.
Then there is an equivalence
$$e^!\co \DM(G \backslash X) \r \DM( S/H).$$
\xlemm

\pf
The map $G/H \r X$ is an equivalence after $\tau$-stackification, hence so is
$e\co S / H = G \backslash G / H \r G \backslash X$.
The lemma follows from \refth{descent.prestacks}.
\xpf

In the definition of $G \backslash X$, there is no consideration of a topology on $\AffSch_S$.
Consequently, $G \backslash X$ is only a prestack.
Now suppose $G\in \Stk^\tau$ for some Grothendieck topology $\tau$. The category of $\tau$-stacks with $G$-action $\Stk_G^\tau$ is the full subcategory of $\PreStk_G$ whose objects are $\tau$-stacks.
The following general statement about principal bundles in an \ii-topos is well-known, see e.g. \cite[Prop.~3.7]{NikolausSchreiberStevenson:Principal} for the implication i)$\Rightarrow$ii).

\lemm \label{lemm--torsors}
Let $X\in\Stk^\tau$ equipped with the trivial $G$-action. The following data are equivalent:\smallskip\\
i\textup{)} An object $P\in \Stk_G^\tau$ together with an equivalence $\al\co (G\backslash P)^\tau\simeq X$.\smallskip\\
ii\textup{)} An effective epimorphism $p\co P\to X$ in $\Stk_G^\tau$ such that $G\x P\simeq P\x_XP$, $(g,p)\mapsto (g p,p)$ is an equivalence.\smallskip\\
In this case, the object $P\to X$ is called a $G$-torsor in the $\tau$-topology.
\xlemm
\pf Let $P\to X$ be any map in an $(\infty,1)$-topos with \v Cech nerve $P^\bullet$. By \cite[Cor.~6.2.3.5]{Lurie:Higher}, $P\to X$ is an effective epimorphism if and only if the canonical map $\colim P^\bullet\to X$ is an equivalence. Now consider a pair as in i). Being a left adjoint, $\tau$-sheafification commutes with colimits so that $\pi\co P\to (G\backslash P)^\tau$ is an effective epimorphism (the \v Cech nerve of $P\to G\backslash P$ is the bar construction, and $G\backslash P$ its colimit). Hence, passing to the \v Cech nerve of $\pi$ the equivalence $\al$ induces on the first simplicial level the equivalence $G\x P\simeq P\x_XP$. Conversely, having $P\to X$ as in ii),
we use that both the bar construction and the \v Cech nerve are 1-coskeletal and conclude that all simplicial levels in the \v Cech nerve are equivalent to the bar construction. Taking colimits, we find a pair as in i).
\xpf

A morphism between $G$-torsors $P\to P'$ over $X$ is a morphism in the slice category $\Stk_G^\tau/X$. A $G$-torsor equivalent to $G\x X$ is called trivial. Clearly, this happens if and only if $P\to X$ admits a section (use the equivalence in \refle{torsors} ii)).

\lemm\label{lemm--loc.triv}
Let $T$ be an affine scheme. Then any $G$-torsor $P\to T$ in the $\tau$-topology is $\tau$-locally on $T$ trivial.
\xlemm
\pf
View $T\simeq (G\backslash P)^\tau$ as an element in $s_T\in (G\backslash P)^\tau(T)$. By the description of sheafification in \cite[§6.5.3, p.~673 and proof of Prop.~6.2.2.7]{Lurie:Higher}, there exists a $\tau$-cover $T'\to T$, and a lift $s_{T'}$ to $G(T')\backslash P(T')$, and thus $P(T')\not = \varnothing$. This shows that the base change $P_{T'}\to T'$ is trivial.
\xpf

\coro \label{coro--equivalence.torsors}
Let $X$ be any $\tau$-stack. Then every morphism of $G$-torsors $P\to P'$ over $X$ is an equivalence.
\xcoro
\pf
This can be tested after pullback $T\to X$ for $T\to \AffSch_S$. By descent, we may work $\tau$-locally on $T$. By \refle{loc.triv}, we may assume that both torsors $P, P'$ are trivial. Then the corollary is clear.
\xpf

Similarly to \StP{04WL}, 
we consider for each $T\in \AffSch_S$ the \ii-category $[G\backslash X]^\tau(T)$ defined as the full subcategory of the slice $\Stk^\tau_G/X\x T$ of those objects $b\x a\co P\to X\x T$ where $a$ is a $G$-torsor in the $\tau$-topology.
By \cite[Cor.~2.4.7.12]{Lurie:Higher}, the formation $T \mapsto \Stk^\tau_G / X \x T \in \Cat_\infty$ is functorial in the $\infty$-sense, so that $[G\backslash X]^\tau$ is a presheaf of \ii-categories, and by \refco{equivalence.torsors} it takes values in \ii-groupoids.
Also being a $\tau$-sheaf, we see that $[G\backslash X]^\tau$ defines $\tau$-stack.

\prop\label{prop--sheafification.iso}
Let $X\in\Stk_G^\tau$. In the above situation,
the natural map given by the trivial $G$-torsor on $X$,
$$X \r [G\backslash X]^\tau,$$
is a $G$-torsor in the $\tau$-topology, so that \textup{(}by \refle{torsors}\textup{)} $(G\backslash X)^\tau=[G\backslash X]^\tau$.
In particular, if $\tau$ is contained in the $h$-topology \textup{(}\refth{descent.prestacks}\textup{)}, there is an equivalence
$$\DM([G\backslash X]^\tau) \stackrel \cong \lr \DM(G \backslash X).$$
\xprop
\pf
We check \refle{torsors} ii). Clearly, $G\x X\simeq X\x_{[G\backslash X]^\tau}X$, and it remains to show that $X \r [G\backslash X]^\tau$ is an effective epimorphism. By \cite[Prop.~7.2.1.14]{Lurie:Higher}, this can be tested on the $0$-trucation, i.e., on its underlying $1$-topos. But a map of ordinary $\tau$-sheaves is an effective epimorphism if it is an epimorphism. Thus, \refle{loc.triv} implies the proposition.
%
\xpf

\rema
We apply \refpr{sheafification.iso} to quotients like $L^+G\backslash LG/L^+G$, i.e., the double quotient of the loop group by the positive loop group. In virtue of \refsect{torsors} below each $L^+G$-torsor for the fpqc topology admits sections \'etale locally so that it is enough to consider \'etale sheafifications, cf.~\refle{DM.double.tau} below.
\xrema

\subsection{Motives on ind-Artin stacks}
\label{sect--DM.Artin}
In \refsect{intersection} below, we will construct intersection motives on moduli stacks of shtukas. The following framework is convenient for our constructions.

Let $\IndArt_S^\lft$ be the category of strict ind-Artin stacks ind-(locally of finite type) over $S$. By definition, every object $X\in \IndArt_S$ admits a presentation over a countable filtered index set
$$X=\colim_{i\in I}X_i\eqlabel{presentation.Ind.Artin}$$
by $S$-Artin stacks locally of finite type as defined in \StP{026O} with transition maps $t_{i,j}\co X_i\r X_j$ being closed immersions for all $i\leq j$.
The category $\IndArt_S^\lft$ is, by definition, a full subcategory of the $(2,1)$-category of presheaves of ordinary groupoids on $\AffSch_S$.
As was mentioned in \refre{prestacks.examples}, we regard it as a full subcategory of $\PreStk_S$.
In \refeq{presentation.Ind.Artin}, $\colim$ denotes the colimit of presheaves of ordinary groupoids.
The inclusion $\tau_{\le 1} \PreStk_S \subset \PreStk_S$ preserves filtered colimits \cite[Cor. 5.5.7.4]{Lurie:Higher}, so we will not distinguish between them.
Any object \refeq{presentation.Ind.Artin} is automatically a sheaf of groupoids in the fppf topology\footnote{Even in the fpqc topology if the diagonal of each $X_i$ is quasi-affine.} on $\AffSch_S$ because each $X_i$ is by definition, and the colimit by \cite[Lem.~4.2.6]{EmertonGee:StackyImages} (each object in $\AffSch_S$ is quasi-compact).

\lemm
\label{lemm--Lurie.co.limit}
\textup{(}Lurie \textup{\cite[Lem.~1.3.3, 1.3.6]{Gaitsgory:Generalities}}\textup{)}
\nts{Another exposition for abelian categories is in Barlev, ``Limits of categories, and sheaves on ind-schemes''}
Let $I$ be an \ii-category and $F: I \r \DGCat_\cont$ a functor.
For a map $\alpha: i \r j$ in $I$, the right adjoint of $F(\alpha)$ \textup{(}which exists by the adjoint functor theorem\textup{)} is denoted $G(\alpha)$.
\begin{enumerate}
\item
The evaluation functors $\lim_{I^\opp} G \stackrel{\ev_i} \lr G(i)$ admit left adjoints.
These left adjoints assemble to an equivalence of \ii-categories
$$\colim_I F \stackrel \cong \lr \lim_{I^\opp} G.$$
Here and below, the colimit is taken in the $\infty$-category $\DGCat_\cont$ of presentable DG-categories with continuous \textup{(}i.e., colimit-preserving\textup{)} functors.

\item
\label{item--Lurie.co.limit.map}
If $I$ is filtered and the $G(\alpha)$ are also continuous, then the following composition of the natural functors and this equivalence, $F(j) \stackrel{\ins_j} \lr \colim_I F \cong \lim_{I^\opp} G \stackrel{\ev_i}\lr G(i)$, can be computed as
$$\colim_{k \in I, \alpha: j \r k, \beta: i \r k} G(\beta) \circ F(\alpha).$$
\end{enumerate}
\xlemm


\prop
\label{prop--f_!.ind-Artin}
Let $f: X \r Y$ in $\IndArt_S^\lft$ be any map.
\begin{enumerate}
\item
\label{item--f!.exists}
The functor $f^! : \DM(Y) \r \DM(X)$ has a left adjoint $f_!$.
\item
\label{item--f!.proper}
If $f$ is ind-proper, then $f_!$ satisfies base change with respect to $g^!$ for any map $g : Y' \r Y$ of prestacks.
\item
If $f$ is representable by a closed immersion, $f_!$ is fully faithful.
\end{enumerate}
\xprop

\pf
\refit{f!.exists}:
All \ii-categories in sight are presentable, so we can apply the adjoint functor theorem \cite[Cor.~5.5.2.9]{Lurie:Higher} once we know that $f^!$ preserves all limits and filtered colimits.

If $f$ is a map between $S$-schemes locally of finite type, we pick a (possibly infinite) Zariski cover $V = \bigsqcup V_i \stackrel v \r X$ by open subschemes $V_i \in \AffSch^\ft_S$, and similarly $u: U \r Y$ together with a map $g: V \r U$ compatible with $f$.
\refth{descent.prestacks} gives an equivalence $\DM(Y) = \lim \DM(C^\bullet_u)$ and likewise for $X$.
Applying \refle{Lurie.co.limit}.i) to the composite
\[
\Delta \;\overset{C^\bullet_v}{\lr}\; \Sch_S^\opp\; \overset{\DM^!}{\lr}\; \DGCat_\cont
\]
we see that the forgetful functor $\lim_n \DM(C^n_v) \r \DM(V)$ has a left adjoint, so it preserves limits.
Being a functor in $\DGCat_\cont$, it also preserves colimits.
Being conservative, it therefore creates (co)limits.
To check preservation of (co)limits under $f^!$ we may therefore replace $f$ by $g$.
Using $\DM(\bigsqcup V_i) = \bigsqcap \DM(V_i)$, we further reduce to the case that $X$ and $Y$ are in $\AffSch^\ft_S$, in which case we know the desired properties of $f^!$ by \refsy{motives}, iv) and vi).

If $f$ is a map between algebraic spaces locally of finite type over $S$, we choose an étale cover $u\co U \r Y$ by a scheme $U$ and an étale cover $v\co V \r X$ and a map $g\co V \r U$ so that $u \circ g = f \circ v$.
Using the étale descent equivalences $\DM(X) = \lim \DM(C^\bullet_v)$ and likewise for $Y$, we repeat the above argument and reduce the claim to the previously considered case.

This reasoning can be repeated with Artin stacks instead of algebraic spaces. Here we use that any smooth cover of an Artin stack $u\co U\to X$ by an algebraic space defines an effective epimorphism of \'etale stacks (by \cite[Prop.~7.2.1.14]{Lurie:Higher} this can be checked on the $0$-th truncation), so that the natural map $\colim C_u^\bullet\to X$ is an equivalence after \'etale sheafification. Then by \refth{descent.prestacks} and \refre{DM.prestacks}, we obtain $\DM(X)=\DM((\colim C_u^\bullet)^\et)=\DM(\colim C_u^\bullet)=\lim\DM(C_u^\bullet)$, and we repeat the reasoning as before.

Finally, suppose both $X$ and $Y$ are ind-Artin stacks.
Choosing suitable presentations of $X$ and $Y$, the map $f$ is a colimit of maps $f_i\co X_i\r Y_i$, $i\in I$ over the same index set, cf.~\S\ref{sect--ind.schemes}. Again, the evaluation functors $\ev^X_i : \DM(X) \r \DM(X_i)$ and likewise for $Y_i$ preserve (co)limits.
They satisfy $\ev^X_i f^! = f_i^! \ev^Y_i$.
The family of the $\ev^?_i$ for all $i$ is conservative, so that again preservation of (co)limits under $f^!$ can be checked for $f_i^!$ instead.

The full faithfulness for a closed immersion $f$, i.e., $f^! f_! = \id$, is again checked on each $X_i$ in an ind-presentation \refeq{presentation.Ind.Artin}, and then on smooth, resp.~étale atlases.
We then conclude using the corresponding property for motives on finite type $S$-schemes (combine \refsy{motives}.ix) and x)).

\refit{f!.proper}:
If $f$ is a proper map of locally finite type $S$-schemes, then the left adjoint $f_!$ in the first step is given by
$$\lim_n (\DM(C^n_u) \stackrel {g^n_!} \lr \DM(C^n_v)),$$
where $g^n$ is the $n$-th simplicial level of the map between the two \v Cech nerves.
Indeed, this functor is well-defined by the proper base change (for finite type $S$-schemes) and it is left adjoint to $f^!$ since this is true in each level of the \v Cech nerve.
Applying this argument to each step in the above construction, we obtain our claim for $f$ being an ind-proper map in $\IndArt^\lft_S$.

This way, the claim $f^! f_!= \id$ for a closed immersion $f$ in $\IndArt_S^\lft$ also reduces to the same claim for $f$ in $\AffSch^\ft_S$, where we know it from \refsy{motives}.ix) and x).
\xpf

The preceding statements allow to rephrase motives on ind-Artin stacks in a way which is more closely reminiscent of the usual definition of derived categories on ind-Artin stacks.

\coro
\label{coro--Ind.Artin.co.limit}
For any presentation of an ind-Artin stack $X$ as in \textup{\refeq{presentation.Ind.Artin}}, the category of motives on $X$ can be computed in two ways:
$$\colim_{t_!} \DM(X_i) = \lim_{t^!} \DM(X_i) = \DM(X).$$
If an $S$-group prestack $G$ acts on $X$ and the presentation is $G$-equivariant, then there is an equivalence
$$\colim_{t_!} \DM(G \backslash X_i) \stackrel \cong \lr \lim_{t^!} \DM(G \backslash X_i) \stackrel \cong \lr \DM(G \backslash X).\eqlabel{DM.G.colim}$$
\textup{(}In both statements, the colimit is taken in $\DGCat_\cont$.
Transition functors at the left are the $(t_{ij})_!$.
In the middle, the transition functors are the $t_{ij}^!$.\textup{)}
\xcoro

\pf
The formulations using the limits follow from \refre{DM.prestacks} iii) because $G \backslash X = \colim (G \backslash X_i)$. Then, use \refle{Lurie.co.limit}.
\xpf

By \refco{Ind.Artin.co.limit}, motives $M \in \DM(X)$ can informally be thought of as sequences $M_i \in \DM(X_i)$, together with equivalences $M_i \r t_{i,j}^! M_j$ for $i\leq j$.
These equivalences are subject to higher coherence conditions.
Combining \refle{Lurie.co.limit}.\refit{Lurie.co.limit.map} and the full faithfulness of the $(t_{i,j})_!$, a motive of the form $\ins_i(N)$ is given by $t^!_{i,j} N$ in degrees $j \le i$ and $(t_{i,j})_! N$ for $j \ge i$.
We say that \emph{$M$ is supported on $X_i$} if it is of the form $M = \ins_i(N)$ for some $N \in \DM(X_i)$.

\lemm
\label{lemm--compact.motives.Ind.Artin}
For any ind-Artin stack $X$, the category $\DM(X)$ is compactly generated.
An object $M \in \DM(X)$ is compact iff it is of the form $M = \ins_i(N)$ for some $i$ and some compact object $N \in \DM(X_i)^\comp$.
\textup{(}Thus, it is supported on $X_i$, and is a compact object there.\textup{)}
\xlemm

\pf
We retrace the proof of \refpr{f_!.ind-Artin}: if $X$ is an algebraic space with atlas $v: V \r X$, $\DM(X) = \colim \DM_! (C^\bullet_v)$ is compactly generated by \refeq{colim.DGCat.general}.
Here we use that the !-pushforwards and !-pullbacks along the maps in the \v Cech nerve $C^\bullet_v$ preserve compact objects.
From there, we obtain the claim for Artin stacks $X$ in the same vein.
Similarly, for an ind-Artin stack, use \refco{Ind.Artin.co.limit}.
\nts{The following proof, in the context of abelian categories, is due to Barlev, ``Limits of categories, and sheaves on Ind-schemes'', Proposition 1.14:
On the one hand, $\ins_i$ preserves compact objects since its right adjoint $\ev_i : M \mapsto M_i$ preserves filtered colimits (being a functor in $\DGCat_\cont$).
Conversely, since the individual categories $\DM(X_i)$ are compactly generated, we can write $M = \colim \ins_i M_i$, where $M_i \in \DM(X_i)^\comp$. Since $M$ is compact, this equivalence factors over some $\ins_i M_i$, making $M$ a retract of some $\ins_i M_i$ and therefore of the claimed form.
For any (not necessarily compact) object $M \in \DM(X)$, the map $M \r \colim_i \ins_i (M_i)$ is an equivalence by \refle{Lurie.co.limit} ii), showing that $\DM(X)$ is compactly generated.}
\xpf

Since the pushforwards $(t_{i,j})_!$ are fully faithful, the mapping space $\Hom_{\DM(X)}(M, N)$ between two compact objects is given by $\Hom_{\DM(X_i)}(M_i, N_i)$ for some $i>\!\!>0$.
Summarizing this discussion, we can say that a compact object in $\DM(X)$ is nothing but a motive $M \in \DM(X_i)^\comp$ for some $i\in I$ (and it is identified with its image under $(t_{i,j})_!$ for $j \ge i$).

The following theorem compares the motivic functor $f_!$ with its counterpart for étale $\ell$-adic sheaves due to Liu and Zheng. Liu--Zheng's work is a \ii-categorical refinement of constructions by Laszlo--Olsson \textup{\cite{LaszloOlsson:SixAdic}}.
The formalism for $\Zl$-adic étale sheaves in \cite[§2.3]{LiuZheng:EnhancedAdic} extends to the case of $\D_\et(X, \Ql)$ in view of \refth{D.et!.sheaf} and the fact that for smooth maps (which are the ones needed to cover Artin stacks by schemes) $f^*$ and $f^!$ are equivalent up to twist and shift.

We only state the comparison for Artin stacks as opposed to ind-Artin stacks, since the authors in op.~cit.~do not consider ind-objects.
We will apply the theorem to identify the $\ell$-adic realization of the intersection cohomology motives of moduli stacks of shtukas with the $\ell$-adic intersection cohomology complex as for example considered in \cite[Déf.~4.1]{Lafforgue:Chtoucas}.

\theo
\label{theo--f!.Artin}
Let $f\co X\to Y$ be a map in $\Art_S^\lft$, and let $\ell$ be a prime invertible on $S$.

\begin{enumerate}
\item
\label{item--realization.calX}
The $\ell$-adic realization functor $\rho := \rho_\ell$ for motives on $S$-schemes of finite type (see \refsect{realization}) extends to an $\ell$-adic realization functor
$$\rho_X: \DM(X) \r \D_\et(X, \Ql) := \Ind (\D^\bound_\cstr(X, \Ql))$$
taking values in the ind-completion of the derived \ii-category of $\ell$-adic constructible sheaves constructed in \textup{\cite[§2.3]{LiuZheng:EnhancedAdic}}.

\item
\label{item--rho.f!}
The square
$$\xymatrix{
\DM(X) \ar[d]^{f_!} \ar[r]^{\rho_X} & \D_\et(X, \Ql) \ar[d]^{f_!} \\
\DM(Y) \ar[r]^{\rho_Y} & \D_\et(Y, \Ql)
}\eqlabel{rho.f!}$$
commutes up to equivalence, i.e., there is an equivalence $f_! \circ \rho_{X} \stackrel \cong \r \rho_Y \circ f_!$.
\end{enumerate}
\xtheo

\pf
We will instead show these claims when
$\D_\et(-, \Ql)$ is replaced by $\D(-, \Z/m)$, $m := \ell^n$ and when
$\DM$ is replaced by its integral analogue, denoted $\DM_\h$ in \cite[Def.~5.1.3]{CisinskiDeglise:Etale} (alternatively, in view of \cite[Cor.~5.5.7]{CisinskiDeglise:Etale}, one may use Ayoub's category $\DA_\et$ of étale motives without transfers \cite[§3]{Ayoub:Realisation} if $\ell$ is odd and the $\ell$-cohomological dimension of all residue fields in $S$ is finite).
The rather formal extension to $\Zl$- and then $\Ql$-adic coefficients is omitted.
The claim for $\rho_X$ as stated above follows upon taking the ind-completion, using that $\DM(X)$ is compactly generated (\refle{compact.motives.Ind.Artin}).

As was recalled in the beginning of \refsect{realization}, the properties of $\DM$ listed in \refsy{motives}, hold for the integral $\DM$ as well, except for xiii) and xv), and provided that we replace the word ``compact'' in \refsy{motives} by ``constructible''.

Recall that, for a finite type scheme $U / S$,
there is a pair of adjunctions so $\rho_U = i_* \circ a^*$
(see \cite[4.5.3, 5.5.3]{CisinskiDeglise:Etale} or \cite[§5]{Ayoub:Realisation} under the above-mentioned condition on $S$):
$$
\xymatrix{
\DM(U) \ar@<1ex>[r]^{a^*} &
\DM(U, \Z/m) \ar@<1ex>[l]^{a_*} \ar@<-1ex>[r]_{i_*}^\cong &
\D_\et(U, \Z/m) \ar@<-1ex>[l]_{i^*}
}$$
The functor $a^*$ is obtained by applying the (derived) tensor product $- \t_\Z \Z/m$ to the coefficients, its right adjoint $a_*$ is the forgetful functor.
The right hand equivalence is Ayoub's generalization of Suslin-Voevodsky's rigidity.

The proof of both \refit{realization.calX} and \refit{rho.f!} proceeds in two steps: a first step in which $X$ is an algebraic space, and a second step in which $X$ is an Artin stack.
In the first step, there is an \'etale covering $u\co U \r X$ such that the transition maps in the \v Cech nerve $U^\bullet$ are \'etale maps of locally finite type $S$-schemes.
In the second step, there is similarly a smooth covering by algebraic spaces. We will only spell out the first step in the sequel, the second being analogous.

Throughout, we use the descent equivalence (\refth{descent.prestacks}):
$$\DM(X) \cong \lim \DM^!(U^\bullet) \cong \lim \DM^*(U^\bullet).\eqlabel{DM.calX}$$
At the right, the limit is taken over the same diagram, but using *-pullbacks instead.
These two limits are equivalent by relative purity.

\refit{realization.calX}: Recall from \cite[§0.1]{LiuZheng:Enhanced} that
$$\D_\et(X, \Z/m) \defined \lim \D_\et^*(U^\bullet, \Z/m).\eqlabel{Det.calX}$$
Using \refeq{DM.calX}, to construct $\rho_X$ it remains to observe that $\rho$ (on the category of $S$-schemes) is compatible with *-pullback.\nts{This commutation is the most elementary one, see \cite[Lemme A.3]{Ayoub:Realisation}.}


\refit{rho.f!}:
Using \refeq{DM.calX} and \refeq{Det.calX} (for a smooth covering of $Y$) it is enough to show that \refeq{rho.f!} commutes in case $Y$ is a scheme.
Consider the diagram
$$\xymatrix{
&
\DM(U^\bullet, \Z/m) \ar@<1ex>[r]^{i_{U^\bullet, *}}
&
\D_\et(U^\bullet, \Z/m) \ar@<1ex>[l]_\cong^{i_{U^\bullet}^*}
\\
\DM(X) \ar@<-1ex>[d]_{f_!} \ar@<1ex>[r]^{a^*_X} &
\DM(X, \Z/m) \ar@<-1ex>[d]_{f_!} \ar@<1ex>[l]^{a_{X, *}} \ar@<1ex>[r]^{i_{X, *}} \ar@<-1ex>[u]^{u^!} &
\D_\et(X, \Z/m) \ar@<-1ex>[d]_{f_!} \ar@<-1ex>[u]^{u^!} \ar@<1ex>[l]_\cong^{i_X^*} \\
\DM(Y) \ar@<-1ex>[u]_{f^!} \ar@<1ex>[r]^{a^*_Y} &
\DM(Y, \Z/m) \ar@<-1ex>[u]_{f^!} \ar@<1ex>[l]^{a_{Y, *}} \ar@<1ex>[r]^{i_{Y, *}}&
\D_\et(Y, \Z/m). \ar@<-1ex>[u]_{f^!} \ar@<1ex>[l]_\cong^{i_Y^*}
}$$
The $(-)^*$ and $(-)_!$ functors are the left adjoints, the others the right adjoints.
To show the commutation of the large bottom rectangle involving left adjoints, it is enough to show the commutation of the two individual bottom squares composed of right adjoints.
For the left, this follows from the natural transformation $a_*: \DM^!(-, \Z/m) \r \DM^!(-, \Z)$ of functors $(\AffSch^{\ft}_S)^\opp \r \DGCat_\cont$, which in its turn follows from the transformation $a^*: \DM_!(-, \Z) \r \DM_!(-, \Z/m)$ of functors $\AffSch_S^\ft \r \DGCat_\cont$ given by \cite[Prop.~6.2(b)]{Ayoub:Realisation} or \cite[Cor.~5.5.4]{CisinskiDeglise:Etale}.
To show the commutation of $f^!$ and $i^*$, we use that $f^! : \D_\et(Y, \Z/m) \r \D_\et(X, \Z/m)$ is (by functoriality of $f^!$) the unique functor whose composition with $u^!: \D_\et(X, \Z/m) \r \D_\et(U^\bullet, \Z/m)$ is the functor $(f \circ u)^!$.
Thus, the commutativity of the lower right hand square of right adjoints is equivalent to the commutativity of the square involving the !-pullbacks along the map (of simplicial schemes) $f \circ u$.
This, in turn, holds again by adjunction and the above-mentioned result in {\em loc.~cit.}
\xpf

\rema
\label{rema--l.adic.prestacks}
It would be interesting to apply \refde{DM.prestacks} to the functor
$$\DM_\et^!(-, \Z/\ell^n) \text{ or similarly with } \Zl, \Ql$$
in order to obtain a six-functor formalism of torsion or $\ell$-adic sheaves in great generality.
Another interesting question seems to investigate the relation of D-modules on prestacks as in \cite{Raskin:D-modules} and $\HdR$-modules on prestacks, where $\HdR$ is the motivic ring spectrum representing de Rham cohomology.
\xrema

\refpr{f_!.ind-Artin} allows to conveniently define the motive of an ind-Artin stack ind-(locally of finite type) over $S$. This drops the exhaustiveness condition in \cite[Def.~2.17]{HoskinsLehalleur:Formula} at the expense of getting only a motive in the étale topology, which is coarser than the motive in the Nisnevich topology constructed by the authors in op.~cit.

\defi\label{defi--motive.ind.scheme}
The \emph{motive of an ind-Artin stack} $f\co X \r S$ in $\IndArt_S^\lft$ is defined as the object
$$\M(X) \defined f_! f^! 1_S \in \DM(S).$$
For a presentation $X=\colim X_i$, the motive can be computed as $\M(X) = \colim \M(X_i)$, where $\M(X_i) := (f_i)_! (f_i)^! 1_S$, $f_i\co X_i \r S$.
\xdefi

\exam\label{exam--motive.grass}
Let $\Gr$ be the affine Grassmannian (resp.~$\Fl$ the full affine flag variety) attached to a split reductive group over the spectrum $S$ of any field, cf.~\S \ref{sect--loop.grps} below.
The standard computation of a proper $S$-scheme $X$ stratified by affine spaces ($X = \bigsqcup_{w \in W} \A^{d_w}_S$) gives $\M(X) = \bigoplus_w 1(d_w)[2d_w]$.
Using the localization sequences for motives on ind-schemes (\refth{motives.Ind-schemes}), this extends to ind-schemes, reproducing the computations in \cite[Cor.~23]{Bachmann:Affine}, with one summand for each affine space of a certain dimension $d_\la$ (resp.~$d_w$) occuring in the stratification of $\Gr$ (resp.~$\Fl$) by Iwahori orbits:
$$\eqalign{
\M(\Gr) & \,=\, \bigoplus_{\la\in X_*(T)}  1(d_\la)[2d_\la], \cr
\text{(resp.}\;\;\; \M(\Fl) & \,=\, \bigoplus_{w\in W} 1(d_w)[2d_w].\;\text{)}}$$
\xexam

\subsection{Motives on ind-schemes}
\label{sect--DM.ind-schemes}

In this section, we restrict our attention to motives on ind-schemes which are the objects of main interest in the present manuscript.
Important examples are affine Grassmannians and affine flags varieties, see \refsect{Stratifications.flag}.

Let $\IndSch_S^{\ft}$ be the category of strict ind-schemes of ind-(finite type) over $S$. Every object $X\in \IndSch_S^{\ft}$ admits a presentation
$$X=\colim_{i\in I}X_i\eqlabel{presentation.Ind.scheme}$$
by $S$-schemes of finite type with transition maps $t_{i,j}\co X_i\r X_j$ being closed immersions for all $i\leq j$.
The category $\IndSch_S^\ft$ is, by definition, a subcategory of the (ordinary) category of presheaves on $\AffSch_S$. We regard it as a full subcategory of $\PreStk_S$ by \refre{prestacks.examples}. Note that $\IndSch_S^{\ft}$ is a full subcategory of $\IndArt_S^\lft$, so that the results of \refsect{DM.Artin} carry over to strict ind-schemes of ind-finite type over $S$. In particular, the colimit in \refeq{presentation.Ind.scheme} is the ordinary colimit of presheaves of sets, and each ind-scheme is an fpqc sheaf on $\AffSch_S$. Further, the computation of motives on ind-schemes reduces to the one of motives on schemes as in \refco{Ind.Artin.co.limit} and \refle{compact.motives.Ind.Artin}.

In order to define Whitney-Tate stratified ind-schemes (see \refdele{Whitney.Tate.condition}), we will use the six functor formalism for motives on ind-schemes supplied by the next theorem.

\theo
\label{theo--motives.Ind-schemes}
Motives on ind-schemes of ind-\textup{(}finite type\textup{)} satisfy the properties\footnote{Property xiii) does not carry over.} i\textup{)}-xii\textup{)}, xiv\textup{)}-xvii\textup{)} listed in \refsy{motives}, with the following adjustments:
\begin{itemize}
\item
The functor $f^*$ is defined \textup{(}and left adjoint to $f_*$\textup{)} if the following condition is satisfied: there is a presentation $Y = \colim_i Y_i$ such that the underlying reduced locus $(X \x_Y Y_i)_\red$ is a finite type $S$-scheme \textup{(}as opposed to an ind-scheme\textup{)}.

\item
If $X$ is componentwise quasi-compact, then $f^* 1_S$ is a monoidal unit, for the structural map $f: X \r S$.
In general, $\DM(X)$ does not have a monoidal unit, cf.~Example \ref{exam--monoidal.unit}, but still has $\t$.
\nts{
Suppose there was a monoidal unit $1 \in \DM(X)$.
Then the natural evaluation functor $\DM(X) \r \DM(X_n)$ maps it to a monoidal unit: $ev(1) \t A = ev(1) \t i^* i_* A$, where $i: X_n \r X$ is the closed immersion \textup{(}of ind-schemes\textup{)}.
Using the projection formula for the closed immersion $i$, this can be computed as $i^* (1 \t i_* A) = i^* i_* A = A$.
However, since $\DM(X) = \lim_{i^!} \DM(X_i)$, and $i^! 1 \ne 1$ in general \textup{(}i.e., if the transition immersions $i$ are not also open immersions\textup{)}, there is no such object.}

\item
The term ``smooth'' has to be replaced with ``schematic smooth'' in iii\textup{)}, xii\textup{)}.
The term ``proper'' has to be replaced by ``ind-proper'' in iii\textup{)}, iv\textup{)}. The term ``open immersion'' \textup{(}resp.~``closed immersion''\textup{)} has to be replaced with ``schematic open immersion'' \textup{(}resp.~``schematic closed immersion''\textup{)} in iv\textup{)}, ix\textup{)} \textup{(}resp.~in ix\textup{)}\textup{)}. Item x\textup{)} holds for a Cartesian diagram of ind-schemes whenever the corresponding functors are defined.
\item
For the descent statement in xiv\textup{)}, we require a schematic $h$-covering $X \r Y$.
\end{itemize}
\xtheo

\pf
Throughout, let $f: X \r Y$ be a map in $\IndSch_S^\ft$.
Choosing suitable presentations of $X$ and $Y$, the map $f$ is a colimit of maps $f_i\co X_i\r Y_i$, $i\in I$ over the same index set, cf.~\S\ref{sect--ind.schemes}. We write $\ev_i : \DM(X) \r \DM(X_i)$ and likewise for $Y_i$.

The functor $f^!$ exists for any map of prestacks.
It has a left adjoint $f_!$ by \refpr{f_!.ind-Artin} which satisfies $(f_i)_! \ev_i = \ev_i f_!$. Here we use that the $t_{ij}: X_i \r X_j$ are proper.

Tensor product and internal Hom extend by virtue of the presentation $\DM(X) = \colim_{t_*} \DM(X_i)$ and the formulas $t_* M \t t_*N = t_*(M \t N)$ and $\IHom(t_* M, t_* N) = t_* \IHom(M, N)$, valid for a closed immersion $t$ \cite[Thm.~2.4.50(5)]{CisinskiDeglise:Triangulated}.

\refco{Ind.Artin.co.limit} says that $\DM|_{\IndSch^\ft_S}$ is also the left Kan extension of $\DM$ on $\Sch^\ft_S$, when equipped with $*$-pushforwards. We therefore obtain a functor $f_*$ for any map of ind-schemes.
It satisfies $\ins_i (f_i)_* = f_* \ins_i$.

Now we construct $f^*$ under the above assumption.
We note that $X_\red=\colim_i\tilde X_i$, $\tilde X_i:=(X \x_Y Y_i)_\red$ is a presentation by finite type $S$-schemes.
Using that $\DM(X \x_Y Y_i)=\DM(\tilde X_i)$ by localization, we get an equivalence
$$\DM(X) = \colim_{i} \DM(\tilde X_{i}).$$
Hence, we may define $f^*$, noting that the $*$-pushforwards along the closed immersions given by the transition maps commute with $f^*$ by base change.
One checks that the functor $f^*$ is left adjoint to $f_*$.

The existence of $f_\sharp$ for a schematic smooth map $f$ follows immediately from the case of schemes: after possibly replacing $X_i$ by $X\x_YY_i$, the map $f$ is the colimit of the smooth maps $f_i\co X_i\r Y_i$, $i\in I$.

Finally, we define the subcategories $\DM(X)^{\comp, \wt \le 0}$ to be the colimit (in $\Cat_\infty$) of the categories  $\DM(X_i)^{\comp, \wt \le 0}$ and similarly for ``$\wt \ge 0$'', using that $t_*=t_!$ is weight-exact. The existence of a weight truncation triangle $M^{\wt \le 0} \r M \r M^{\wt \ge 1}$ for some compact object $M \in \DM(X)^\comp$ follows from the above description of compact objects: $M$ is supported on some $X_i$, and a weight truncation triangle of $M$ in $\DM(X_i)$ gives rise to one in $\DM(X)$.
The weight structure on compact objects extends to one on $\DM(X) = \Ind(\DM(X)^\comp)$ by \cite[Prop.~1.3.5]{BondarkoLuzgarev:Relative}.
\xpf

\exam\label{exam--monoidal.unit}
There is no compact monoidal unit in $\DM(\A^\infty)$, where $\A^\infty = \colim (\A^0 \stackrel 0 \r \A^1 \stackrel{\id \x 0} \r \A^2 \dots)$ is an infinite-dimensional affine space: by \refle{compact.motives.Ind.Artin} it would be supported on some $\A^i$. However, using the complementary open immersion $u$ of the inclusion $t_{i, i+1}$,  tensoring a motive supported on $\A^i$ with one of the form $u_! u^* M$, $M \in \DM(\A^{i+1})$ gives zero.
Indeed, by construction of $\t$ on $\DM(\A^\infty)$, the inclusion functors $\DM(\A^i) \r \DM(\A^\infty)$ are monoidal.
\xexam


We finish this section by extending the exterior product of motives to the case of certain pro-algebraic group actions on ind-schemes.
This will be needed in \refsect{intersection}.

\prop
\label{prop--boxtimes}
For $j=1, \dots, n$, let $X^{(j)} = \colim_i X_i^{(j)}$ in $\IndSch_S^\ft$ equipped with an action of a strictly pro-algebraic group $G^{(j)} = \lim G_i^{(j)}$ so that the $G^{(j)}$-action on $X_i^{(j)}$ factors over $G_i^{(j)}$.
We assume that the kernels $\ker (G^{(j)} \r G^{(j)}_i)$ are split pro-unipotent.
Then there is a functor
$$\boxtimes\co \DM\left(G^{(1)} \backslash X^{(1)}\right) \x \dots \x \DM\left(G^{(n)} \backslash X^{(n)}\right) \r \DM \left(\sqcap_j G^{(j)} \backslash X^{(j)}\right).$$
Forgetting the actions of the $G^{(j)}$, and restricting to objects supported on some $X_{i_j}^{(j)}$, this functor agrees with the usual exterior product.
\xprop

\pf
By \refco{Ind.Artin.co.limit}, \refpr{DM.G.homotopy.invariant}, and \refle{DM.G.BarC}, there are equivalences of $\infty$-categories
$$\eqalign{
\DM_{G^{(j)}}(X^{(j)}) & = \colim_i \DM_{G^{(j)}} (X^{(j)}_i) \cr
& = \colim_i \lim_{i' \ge i} \DM_{G_{i'}^{(j)}} (X^{(j)}_i) \cr
& = \colim_i \lim_{i' \ge i} \lim \DM(\BarC(G_{i'}^{(j)}, X^{(j)}_i)).}$$
To define a functor with the asserted properties it is therefore enough to observe that $\boxtimes$ commutes with the !-pushforward along the closed immersions used to form the strict ind-schemes $X^{(j)}$ and then that $\boxtimes$ commutes with the !-pullback along the smooth action and projection maps in the bar construction $\BarC(G_{i'}^{(j)}, X^{(j)}_i)$.
\nts{
We have natural isomorphisms $(f \x g)_! (- \boxtimes -) \stackrel \cong \r (f_! -) \boxtimes (g_! -)$ for any separated map. This follows by showing it separately for proper maps and open embeddings.
By adjunction, we get maps $(f \x g)^! (- \boxtimes -) \gets (f^! -) \boxtimes (g^! -)$.
We claim these are isomorphisms whenever $f$ and $g$ are smooth.
It suffices to assume $g=\id$, by factoring $(f \x g)=(f \x \id)(\id \x g)$, cf. the argument in \cite[Lemma 2.3.4]{JinYang:Kuenneth}.
Consider the diagram whose left half is cartesian:
$$\xymatrix{
X \ar[d] f & X \x Y \ar[l]_x \ar[d]^{f \x \id} \ar[r]^y & Y \ar@{=}[d] \\
X' & X' \x Y \ar[l]_{x'} \ar[r]^{y'} & Y.
}$$
We have the following computations
$$\eqalign{f^! - \boxtimes - &:=  x^* f^! - \t y^* \cr
& \stackrel 1 = (f \x \id)^! x'^* - \t y^* - \cr
& = (f \x \id)^! x'^* - \t (f \x \id)^* y'^ - \cr
& \stackrel 2 \r (f \x \id)^! (x'^* - \t y'^* -).}$$
The equality 1 holds by the relative purity isomorphism using the smoothness of $f$.
The map 2 is part of the projection formula. It is an isomorphism (for smooth $f \x \id$) again by purity.
}
\xpf

\section{Stratified Tate motives}
\label{sect--DTM}

In this section, we discuss stratified Tate motives, i.e., motives which are Tate motives on each stratum of a given stratification.
We work in the generality of Tate motives on stratified ind-schemes relative to a base scheme, which extends the work of Soergel and Wendt \cite[§§3-4]{SoergelWendt:Perverse}.

\nota
\label{nota--S.nochmal}
Our base scheme $S$ is as in \refno{S}.
By convention, all \mbox{ind-}schemes are strict
$S$-ind-schemes of ind-finite type.
Recall from \refsect{ind.schemes} that we only consider ind-schemes indexed by countable index sets.
\xnota

\subsection{Definitions and elementary construction principles}
\label{sect--DTM.definitions}

\defi
\label{defi--stratified.dfn}
i) A \emph{stratified ind-scheme} is a map of ind-schemes
\[
\iota\co X^+=\bigsqcup_{w\in W}X_w\,\r\, X
\]
such that $\iota$ is bijective on the underlying sets, each stratum $X_w$ is a scheme, the restriction to each stratum $\iota|_{X_w}$ is representable by a quasi-compact immersion and the topological closure of each stratum $\iota(X_w)$ is a union of strata.\smallskip\\
ii) A \emph{map of stratified ind-schemes} is a commutative diagram of (ind-)schemes
\begin{equation}\label{strat.map.square}
\xymatrix{
X^+ \ar[r]^{\iota_X} \ar[d]^{\pi^+} & X \ar[d]^\pi \\
Y^+ \ar[r]^{\iota_Y} & Y,}
\end{equation}
where $\pi$ is a schematic map of finite type, and $\pi^+$ maps each stratum in $X^+$ into a stratum in $Y^+$. The latter condition is automatically satisfied whenever the strata are connected.
\xdefi

\rema
\label{rema--explain.stratified.Ind.scheme}
i)
If $X$ happens to be a finite type $S$-scheme, then $W$ is necessarily finite.
For an ind-scheme, $W$ may be countably infinite. However, all the $X_w$ are necessarily of finite type.
\smallskip\\
ii)
By the localization sequence in Synopsis \ref{syno--motives} ix), the category $\DM(X)$ only depends on the underlying reduced ind-scheme structure. After possibly replacing $X$ (resp.~$X^+$) by their reduced sub-ind-schemes, we may and do assume $X$ and $X^+$ to be reduced.
\smallskip\\
iii)
We do not in general assume the strata $X^+=\bigsqcup_{w\in W}X_w$ to be regular (or smooth over $S$).
\smallskip\\
iv)
Let $\iota\co X^+=\bigsqcup_{w\in W}X_w\r X$ be a stratified ind-scheme.
Let $X=\on{colim}_iX_i$ be any ind-presentation. For each $w\in W$, the map $\iota|_{X_w}\co X_w\r X$ factors as $X_w\to X_i\subset X$ for some $i>\!\!>0$ because $X_w$ is quasi-compact (being a finite type $S$-scheme). Hence, the scheme-theoretic image $\overline{X_w}\subset X_i$ of $\iota|_{X_w}$ is a separated $S$-scheme of finite type, and its underlying topological space agrees with the topological closure of $\iota(X_w)$, cf.~\StP{01R8}. Hence, the base change map $\ol{X_w}^+:=X^+\x_{X}\ol{X_w}\r \ol{X_w}$ is a stratified scheme, and $\ol{X_w}^+=\bigsqcup_{v\leq w}X_{v}$ is an (automatically finite) union of strata. There is a presentation
\begin{equation}\label{stratified.Ind-pres}
X\,=\,\colim_{w\in W}\ol{X_w},
\end{equation}
where $W$ is partially ordered by the closure relations of the strata. We can think about $\iota\co X^+\r X$ as being the colimit of the stratified schemes $\ol{X_w}^+=\bigsqcup_{v\leq w}X_{v}\r\ol{X_w}$.
\xrema

All stratified (ind-)schemes we will encounter in \refsect{DTM.Fl} are cellular in the following sense:

\defi
\label{defi--cellular}
An \emph{$S$-cell} is an $S$-scheme isomorphic to $\bbV(\calE) \x \bigsqcap_{i=1}^r \bbV^\x (\calE_i)$ for some vector bundles
$\calE$, $\calE_i$ on $S$ (see \refeq{VE} for notation).
A \emph{cellular $S$-scheme} $X$ is a separated $S$-scheme of finite type which is smooth and admits a stratification into cells.
A \emph{cellular stratified $S$-ind-scheme} $\iota\co X^+=\bigsqcup_{w\in W}X_w \r X$ is a stratified $S$-ind-scheme where $X_w$ is a cellular $S$-scheme for every $w\in W$.
\xdefi

\exam \label{exam--stratified.dfn}
The affine Grassmannian $\Gr_G$ for a split reductive group scheme $G$ over $S$, equipped with its stratification by $L^+G$-orbits, and more general ind-schemes are shown to be cellular in \refsect{Stratifications.flag}.

Our notion of cellularity is less restrictive than the one in \cite{DuggerIsaksen:Motivic}, say, which requires a stratification by affine spaces.
In particular any split reductive group $G$ over $S$ is cellular in our sense by means of the Bruhat stratification.
\xexam

We just write $\iota\co X^+\r X$ or even $X$ whenever $X^+=\bigsqcup_{w\in W}X_w$ or $\iota$ are clear from the context, and likewise for stratified maps. Given a stratified map \eqref{strat.map.square}, we have a commutative diagram of ind-schemes
\begin{equation}\label{stratified.diagram}
\xymatrix{
X^+ \ar@/^2pc/[rr]^{\iota_X} \ar[r]^f \ar[dr]_{\pi^+} & \tilde X \ar[d]^{\tilde \pi} \ar[r]^{\tilde \iota_Y} & X \ar[d]^\pi \\
& Y^+ \ar[r]^{\iota_Y} & Y, 
}
\end{equation}
where the square is cartesian. It is easily checked that $\tilde \iota_Y\co \tilde X=\bigsqcup_{w\in W_Y}(Y_{w}\x_YX)\r X$ and $f\co X^+=\bigsqcup_{w\in W_X}X_w\r \tilde X$ are again stratifications. Thus, a stratified map amounts to possibly refining the preimage of the stratification on $Y$.

\defi
\label{defi--Tate.geometry}
For a scheme $X$ of the form $X = \bigsqcup_{w\in W} X_w$, with $X_w \to S$ of finite type, the category of \emph{Tate motives}
$$\DTM(X) \subset \DM(X)$$
is the stable cocomplete sub-$\infty$-category generated by the objects $1_X(n)$, for $n \in \Z$.

A map $\pi: X \r Y$ between such ind-schemes is called a \emph{Tate map} if $\pi_* 1_X \in \DTM(Y)$. This holds true if and only if the adjunction \refeq{adjunction.star} restricts to an adjunction
$$\pi^*: \DTM(Y) \rightleftarrows \DTM(X) : \pi_*.$$
\xdefi

\rema
\label{rema--DTM.large}
By definition, $\DTM(X)$ is large in the sense that it contains arbitrary coproducts. This implies in particular that it is idempotent complete, i.e., stable under taking direct summands.
For example, if $\iota\co X_w \r X$ is the inclusion of a connected component of $X$, then $\DTM(X)$ contains $\iota_* 1_{X_w}(n)$ because it is a direct summand of $1_X(n)$. We will often apply this remark in the case where $X$ is the disjoint union of strata in some scheme.
If $X$ has infinitely many (finite type) connected components $X_w$, $1_X$ is not compact.
Indeed, by \refle{compact.motives.Ind.Artin} a motive is compact iff its support is contained in finitely many $X_w$'s and is a compact object in the usual sense there.
\xrema

\exam
For any vector bundles $\calE, \calE_1,\ldots,\calE_b$ on $Y$, the projection $\pi\co X := \bbV(\mathcal E) \x \bigsqcap_{i=1}^b \bbV^\x (\mathcal E_i) \r Y$ (see \refeq{VE} for notation) is a Tate map, by the definition of Tate objects, homotopy invariance and localization \cite[4.20]{Deglise:Around}.
\nts{If the $\mathcal E_i$ are trivial bundles, then we have $\pi_\sharp 1 = \bigotimes_{i=1}^b \left (1 \oplus 1(\rk \mathcal E_i)[2 \rk \mathcal E_i-1] \right ).$
In general, the map $a$ in
$\M(\bbV^\x E_i) \r \M(\bbV \mathcal E_i) \stackrel a \r \M(Y)(e)[2e]$ might not vanish (examples given in \cite[Cor.~4.21]{Deglise:Around}), so the extension does not split in general.}
\xexam

The following condition, introduced by Soergel and Wendt \cite[\S 4]{SoergelWendt:Perverse} for Tate motives on stratified schemes (as opposed to ind-schemes), ensures a well-behaved notion of stratified Tate motives:

\defilemm
\label{defilemm--Whitney.Tate.condition} Let $\iota\co X^+=\bigsqcup_{w\in W}X_w \r X$ be a stratified ind-scheme. The following are equivalent:\smallskip\\
i) $\iota^* \iota_* 1_{X^+} \in \DTM(X^+)$,\smallskip\\
ii) $\iota^! \iota_! 1_{X^+} \in \DTM(X^+)$.\smallskip\\
If either i) or ii) holds true, then $\iota\co X^+\to X$ is called a \emph{Whitney-Tate stratification}. In this case, the following subcategories a)-d) of $\DM(X)$ are all the same:\smallskip\\
a) $\langle \iota_{w,*} 1_{X_w}(n)\,;\, w\in W\rangle$, where $\iota_w :=\iota|_{X_w}: X_w \r X$ denotes the inclusion of a stratum,
and ``$\langle \str\rangle$'' denotes the stable sub-$\infty$-category generated under arbitrary (homotopy) colimits, suspensions and desuspensions. (In the language of triangulated categories, this corresponds to the closure under arbitrary shifts, extensions, and coproducts.)
\smallskip\\
b) $\langle \iota_{w,!} 1_{X_w}(n) \,;\, w\in W \rangle$,\smallskip\\
c) $\{M \in \DM(X), \iota^! M \in \DTM(X^+) \}$,\smallskip\\
d) $\{M \in \DM(X), \iota^* M \in \DTM(X^+) \}$.\smallskip\\
This category is denoted $\DTM(X, X^+)$ or just $\DTM(X)$ if the stratification is clear from the context.
\xdefilemm

\pf
For finite type schemes $X$ with regular $X^+$, this is due to \cite[\S4]{SoergelWendt:Perverse}.
The regularity assumption is unnecessary, by replacing the usage of Verdier duality by localization arguments such as the cofiber sequence $i^! \r i^* \r i^* j_* j^*$ (obtained by applying $i^*$ to \refeq{localizationOne}) for two complementary closed and open embeddings $i$ and $j$.
The extension from finite type schemes to ind-schemes is formal, using the following remark:
write $X=\colim_w\ol {X_w}$ as in \eqref{stratified.Ind-pres}. Condition i) is equivalent to the condition $\mathrm{i)_w}$ for all $w\in W$: $\iota_w^* \iota_{*} 1 \in \DTM(X_w)$.
Using that the closure $\ol{X_w}$ in $X$ consists only of finitely many strata, Condition $\mathrm{i)_w}$ for $X$ is equivalent to Condition $\mathrm{i)_w}$ for $\ol{X_w}$.
Similarly for ii), which reduces the equivalence i) $\Leftrightarrow$ ii) to the case of schemes, and accomplishes the proof.
The agreement of a)--d) can also be reduced to the case of schemes in a similar manner.
\xpf

\rema
Following up \refre{DTM.large}, $\DTM(X, X^+)$ is again cocomplete.
The consideration of these large categories is merely a matter of convenience.
We could instead consider its subcategory of compact objects, which by \refle{compact.motives.Ind.Artin} consists precisely of those Tate motives (in the above sense) whose support is contained in finitely many strata and is compact there.

\nts{The category $\DTM(X, X^+)$ is the Ind-completion of its subcategory of compact objects, and all the functors we consider in this paper preserve compact objects.}
\xrema

\exam
\label{exam--G.Whitney.Tate}
Let $k$ be a field, and let $G$ be a split reductive $k$-group with Borel subgroup $B$. Then the Bruhat stratification by $B$-orbits on $G/B$ is a cellular Whitney-Tate stratification by \cite[Prop.~4.10]{SoergelWendt:Perverse}.
We will reprove and extend this statement to the case of partial affine flag varieties in \refth{Fl.WT}.
\nts
{In addition $G$, with the preimage stratification along $p: G \r G/B$ is also Whitney-Tate. It satisfies \eqref{Reunion} since the projection $p$ is smooth and surjective, and therefore a quotient map, so that $p^{-1}(\ol {Y_w}) = \ol{p^{-1}(Y_w)}$.
The strata $G_w := p^{-1} (Y_w)$ are cells, see e.g. Milne's script on ``Algebraic groups'', Thm. 22.72 with Thm 22.64.
$G$ is Whitney-Tate stratified since $p$ is smooth (\refex{basic.WT}\refit{preimage.stratified}).}
\xexam

\rema
\label{rema--WT.smooth.duality}
Let $X^+\r X$ be a Whitney-Tate stratified ind-scheme with $X^+$ being a regular scheme. Then $\DTM(X) \subset \DM(X)$ is stable under the dualising functor $\Du$ (\refsy{motives} viii)). Indeed, $\Du(\iota^* \iota_* 1)=\iota^! \iota_! \Du(1) = \iota^! \iota_! 1(\dim X^+)[2\dim X^+]$ where $\dim X^+\co |X^+|\to \bbZ_{\geq 0}$ is viewed as a locally constant function.
\xrema

\defi
\label{defi--Whitney-Tate.map}
A stratified map $\pi : (X, X^+) \r (Y, Y^+)$ of Whitney-Tate stratified ind-schemes is a \emph{Whitney-Tate map} if \refeq{adjunction.star} restricts to an adjunction
$$\pi^* : \DTM(Y, Y^+) \rightleftarrows \DTM(X, X^+): \pi_*.$$
\xdefi

\rema
\label{rema--WT.maps}
i) Definition \ref{defi--Whitney-Tate.map} is in effect only a condition on $\pi_*$ since $\pi^*$ preserves Tate motives for any stratified map: using the notation of \eqref{stratified.diagram}, $\iota_X^* \pi^* \iota_{Y,!} 1_{Y^+} = f^* \tilde \iota_Y^* \tilde \iota_{Y,!} 1_{\tilde X}$.
By localization, $1_{\tilde X}$ lies in the smallest subcategory generated by $f_! 1$ under extensions and retracts.
Indeed, each connected component of $\tilde X$ is of finite type and is therefore stratified by finitely many $X_w$. We conclude the claim from the localization sequence \refeq{localization}.
Hence the above motive is obtained by extensions and direct summands from $\iota_X^* \iota_{X,!} 1 \in \DTM(X^+)$.\smallskip\\
ii) In addition, if $X^+$ and $Y^+$ are regular, then the Whitney-Tate condition on $\pi$ is equivalent to the existence of an adjunction
$$\pi_! : \DTM(X, X^+) \rightleftarrows \DTM(Y, Y^+) : \pi^!.$$
This follows from \refre{WT.smooth.duality}.
Finally, if in addition $\pi$ is smooth, then the Whitney-Tate condition can also be expressed using $\pi_\sharp$ (using the equivalence $\pi_\sharp = \pi_! (d)[2d]$, $d$ being the relative dimension of $\pi$).
\xrema

\exam
\label{exam--basic.WT}
i) If a schematic smooth map $\pi : X \r Y$ of finite type and a Whitney-Tate stratification $\iota_Y\co Y^+ \r Y$ is given,
then the \emph{preimage stratification} $\iota_X\co X^+ := X \x_Y Y^+ \r X$ is again Whitney-Tate.
Indeed, $\iota_X^* \iota_{X, *} 1 = \pi^{+,*} \iota_Y^* \iota_{Y,*} 1$ using smooth base change, which is a Tate-motive since $Y$ is Whitney-Tate.
\smallskip\\
ii) If, in addition, $\pi^+ : X^+ \r Y^+$ is a Tate map, then $\pi$ is a Whitney-Tate map.
\xexam

The following lemmas give relations between (partial) Whitney-Tate properties of source and target of a proper map.
They will be used to show that partial affine flag varieties are Whitney-Tate, cf.~\refsect{DTM.Fl}.

\lemm
\label{lemm--Tate.up.down}
Let $\pi\co X \r Y$ be a map of stratified ind-schemes such that $\pi^+$ is a Tate map.
We assume that either $\pi^+$ is smooth or that $X^+$ and $Y^+$ are both regular.
\textup{(}We do not assume the stratifications on $X$ or $Y$ are Whitney-Tate.\textup{)} If $M \in \DM(X)$ is such that $\iota_X^! M \in \DTM(X^+)$, then $\iota_X^! \pi^! \pi_* M\in \DTM(X^+)$.
\xlemm

\pf
We use the notation of \eqref{stratified.diagram} and compute
$\iota_X^! \pi^! \pi_* M =
(\pi^+)^! \tilde \pi_* \tilde \iota_Y^! M$.
Any $N \in \DTM(\tilde X)$ lies in the subcategory generated by the summands (corresponding to the connected components of $X^+$) of $f_* f^! N$ by localization.
We may thus consider $(\pi^+)^! \tilde \pi_* f_* f^! \tilde \iota_Y^! M = (\pi^+)^! \pi^+_* \iota_X^! M$ instead. This is a Tate motive by one of the assumptions.
\xpf

\lemm
\label{lemm--Tate.proper.descent}
Let $\pi: X \r Y$ be a proper map of stratified ind-schemes such that the map $\pi^+: X^+ \r Y^+$ has a section $s^+\co Y^+\to X^+$. Assume that $s^+$ is an open and closed immersion which identifies the strata of $Y^+$ with some strata of $X^+$. Further, assume that $X$ is Whitney-Tate, and that $\pi^+$ is a Tate map between regular schemes.
Then $Y$ is also Whitney-Tate, and $\pi$ is a Whitney-Tate map.
\xlemm

\pf
We have to show that $\iota_Y^* \iota_{Y,*} 1 \in \DTM(Y^+)$. Using the notation of \eqref{stratified.diagram}, we compute
\begin{align}
\iota_Y^* \iota_{Y,*} 1 &=\iota_Y^* \pi_* \tilde \iota_{Y,*} f_* s^+_* 1 \nonumber \\
       &=\tilde \pi_* \tilde \iota_Y^* \tilde \iota_{Y,*} f_* s^+_* 1 \ \text{(proper base change).} \nonumber
\end{align}
The map $s^+$ identifies, by assumption, the strata of $Y^+$ with strata of $X^+$, and therefore $s^+_* 1 \in \DTM(X^+)$.
Let $M \in \DM(\tilde X)^\comp$ be any compact object.
The localization property of motives (Synopsis \ref{syno--motives} ix), together with an induction on the finite number of strata on which $M$ is supported) implies that $\tilde \pi_! M$ is an extension of direct summands of $\pi^+_! f^* M$.
It is thus enough to show that $\pi^+_! f^* \tilde \iota_Y^* \tilde \iota_{Y,*} f_* s^+_* 1$ is a Tate motive. This holds since $\pi^+_!$ and $\iota_X^* \iota_{X, *} = f^* \tilde \iota_Y^* \tilde \iota_{Y,*} f_*$ preserve Tate motives by assumption and \refre{WT.maps}. This shows that $Y$ is Whitney-Tate.

In order to see that $\pi$ is Whitney-Tate, we use that $\pi_* \tilde \iota_{Y,*} f_* 1 = \pi^+_* \iota_{Y,*} 1$ is a Tate motive on $Y$ because $\pi^+$ and $\iota_Y$ are both Tate maps by assumption and the previous step, respectively. This proves the lemma.
\xpf

\lemm
\label{lemm--smooth.detects.Tate}
Let $\pi: X \r Y$ be a smooth stratified map of stratified Whitney-Tate ind-schemes such that $X^+$ and $Y^+$ are regular, $\pi^+$ is a Tate map, and such that in the notation in \eqref{stratified.diagram} the adjunction map $\tilde \pi_\sharp 1 \r 1$ admits a section \textup{(}this holds true if $\tilde \pi$ or, a fortiori, $\pi^+$ admits a section\textup{)}. Then $\pi^*$ preserves and detects Tate motives in the following sense: for $M\in \DM(Y)$ one has
$$
\pi^* M \in \DTM(X) \Leftrightarrow M \in \DTM(Y).$$
\xlemm

\pf
The implication $\Leftarrow$ follows from \refre{WT.maps}.
Conversely, we use the projection formula \cite[1.1.26]{CisinskiDeglise:Triangulated} (extended to motives over ind-schemes as discussed in \refsect{DM.ind-schemes}) and smooth base-change to see that the adjunction map
$$\tilde \pi_\sharp \tilde \pi^* 1 \t \iota_Y^* M = \tilde \pi_\sharp \tilde \pi^* \iota_Y^* M = \iota_Y^* \pi_\sharp \pi^* M \r \iota_Y^* M$$
admits a section by assumption.
Since $\pi_\sharp$ preserves stratified Tate motives (by assumption and \refre{WT.maps} ii)), we are done.
\xpf

\defi
\label{defi--DTM.G}
Suppose an ordinary presheaf of $S$-groups $G$ acts on a stratified ind-scheme $X$. If the stratification on $X$ is Whitney-Tate, then we define the category $\DTM_G(X)$ of \emph{$G$-equivariant stratified Tate motives} as the homotopy pullback
\[
\DTM_G(X)\defined \DTM(X) \x_{\DM(X)} \DM(G \backslash X),
\]
i.e., as the full subcategory of $\DM(G\backslash X)$ whose underlying object in $\DM(X)$ is in $\DTM(X)$.
\xdefi

\rema
The category $\DTM_G(X)$ is defined even if the $G$-action on $X$ does not preserve the stratification. For example, we can consider the category $\DTM_{L^+G}(\Gr_G)$ where $\Gr_G$ is equipped with the stratification into either $L^+G$-orbits or Iwahori orbits. In the latter case, the $L^+G$-action does not preserve the strata, but it turns out that the resulting categories are the same.
\xrema

The notation $\DTM_G(X)$ (as opposed to $\DTM(G \backslash X)$) highlights the fact that the category of stratified Tate motives does depend on the given presentation of the prestack quotient $G \backslash X$.
While descent does hold for $\DM$, it does not hold for $\DTM$: being a Tate motive is a property of motives which does not in general descend.
Here is, however, a sufficient condition which ensures such a descent behavior.

\prop
\label{prop--DTM.G/H}
Let $H\subset G$ be an inclusion of ordinary $\tau$-sheaves of $S$-groups where $\tau$ is a Grothendieck topology as in \refth{descent.prestacks}.
Let $X := (G/H)^\tau$ be the quotient of $\tau$-sheaves which we assume to be a smooth finite type $S$-scheme.
We suppose $X$ carries a Whitney-Tate stratification such that the base point $S \r X$ is a stratified map, i.e., factors through $X^+$.
Then the equivalence $\DM_G(X) = \DM_H(S)$ established in \refle{DM.G/H} restricts to an equivalence
$$\DTM_G(X) \overset{\simeq}{\lr} \DTM_H(S).$$
In particular, the left hand category is insensitive to the choice of the stratification \textup{(}under the above assumption\textup{)}.
\xprop

\pf
We have a commutative diagram of prestacks, where the vertical maps are the standard quotient maps and $\pi$ is the structural map of $S$, and $\sim$ indicates that the map $\gamma$ becomes an isomorphism after $\tau$-sheafification so that $\gamma^!$ is an equivalence by \refth{descent.prestacks}:
$$\xymatrix{
X \ar[d]^{\alpha} \ar[r]^\pi & S \ar[d]^{\beta} \\
G \backslash X \ar[r]^\gamma_{\sim} &  S/H.
}\eqlabel{X.S.GX.etc}$$
We have the following commutative diagram of \ii-categories:
$$\xymatrix{
\DTM_H(S) \ar[rr] \ar@{^{(}->}[dd] \ar[dr] & & \DTM_G(X) \ar@{^{(}-}[d] \ar[dr] \\
& \DTM(S) \ar@{^{(}->}[dd] \ar[rr]^(.4){\pi^!} & \ar[d] & \DTM(X) \ar@{^{(}->}[dd] \\
\DM(S/H) \ar@{-}[r] \ar[dr]^{\beta^!} & \ar[r]^(.25){\gamma^!}_(.25)\cong & \DM(G \backslash X) \ar[dr]^{\alpha^!} \\
& \DM(S) \ar[rr]^(.4){\pi^!} & & \DM(X).
}$$
The left and right face is cartesian by definition of equivariant Tate motives.
The front face is cartesian by \refle{smooth.detects.Tate} (which uses the assumption on the stratified section).
We conclude that the back square is cartesian,
so we get our claim.
\xpf

\exam\label{exam--descent.Tate.not}
The existence of a section in the above proposition is crucial: consider a finite Galois extension $K / k$ with Galois group $G$
and let $X := \Spec K \stackrel \pi \r \Spec k =: S$.
Then $G \backslash X = S$, but the natural functor $\DTM(S) \r \DTM_G(X)$ (whose composite with the forgetful functor to $\DTM(X)$ is just $\pi^!$) is fully faithful, but not essentially surjective, since $1 \in \DTM_G(X)$ is not in the image.
In fact, $\DTM_G(X)$ can be identified with the category $M$ of Artin-Tate motives on $S$ such that $\pi^! M$ is a Tate motive.
\xexam

\defi
\label{defi--stratified.G.action}
A \emph{stratified $G$-action} of an ordinary group presheaf $G$ on a stratified ind-scheme $(X, X^+)$ is an action $G\x_S X\to X$ which restricts to an action (necessarily in a unique way) $ G \x_S X^+ \r X^+$.
If $G$ is algebraic and fibrewise connected, this is equivalent to requiring that each stratum $X_w$ is $G$-stable.
\xdefi

Let $G=\lim_{i\geq 0}G_i$ be strictly pro-algebraic, cf.~\refsect{algebraic.grps}. By taking suitable unions of $\ol{X_w}$, $w\in W$ there exists a presentation $X=\colim_{i\geq 0}X_i$ with the following properties (cf.~Lemma \ref{lemm--adm.strata}): for every $i\geq 0$, the $G$-action on $X_i$ factors through $G_i$, and the map $\iota_i\co X^+\x_XX_i\to X_i$ is a stratified scheme of the form $X_i^+=\bigsqcup_{w\in W_i}X_w$ for a suitable finite subset $W_i\subset W$.

\prop
\label{prop--DTM.G}
Let $X = \colim_i X_i$ be a stratified Whitney-Tate ind-scheme equipped with a stratified $G$-action of a strictly pro-algebraic group $G = \lim G_i$ so that the $G$-action on $X_i$ factors over $G_i$.
We assume that the kernels $\ker (G \r G_i)$ are split pro-unipotent.
Then there are equivalences of $\infty$-categories
$$\eqalign{
\DTM_G(X) & = \lim_i \DTM_G (X_i) \cr
& = \lim_i \lim_{j \ge i} \DTM_{G_j} (X_i) \cr
& = \lim_i \lim_{j \ge i} \lim \DTM^! (\BarC(G_j, X_i)),
}$$
where all the limits are formed by $!$-pullbacks.
\xprop

\pf
The first equivalence follows from the definitions and the commutation of homotopy pullbacks of \ii-categories with limits.
The second equivalence follows from \refpr{DM.G.homotopy.invariant}.
(The restriction functors $\DTM_{G_j}(X_i) \r \DTM_{G_{j'}}(X_i)$ are equivalences for all $j' \ge j \ge i$, but we keep them in order to be able to take the colimit over $i$.)
To simplify notation, we now write $G$ for $G_j$ and $X$ for $X_i$.
For each $n\geq 0$, the stratification $G^n\x_S X^+\r G^n\x_SX$ obtained by pulling back the one on $X$ is Whitney-Tate by \refex{basic.WT} i), so the categories $\DTM(G^n\x_SX)$ are well-defined.
In addition, all maps appearing in $\BarC(G,X)$ are maps of stratified schemes since the action (resp.~projection) map $G\x_SX\to X$ is smooth and is stratified by assumption (resp.~construction).
Hence, the !-pullback along these maps  preserves the subcategories $\DTM(G^n\x_SX)$ by Remark \ref{rema--WT.maps} i).
Thus $\DTM^!(\BarC(G, X))$ is well-defined.

By \refle{lim.equivalence}, there is an equivalence of $\infty$-categories
\[
\lim \DTM^!(\BarC(G, X))= \lim \left(\DTM(X) \x_{\DM(X)} \DM^!(\BarC(G,X))\right).
\]
Indeed, in both limits, the category at the 0-vertex is $\DTM(X)$.
The right hand side computes $\DTM_G(X)$ by the commutation of limits with pullbacks.
\xpf

\subsection{Stratified mixed Tate motives}
\label{sect--MTM}

\nota
\label{nota--BS.vanishing}
Throughout \refsect{MTM}, $S$ is a scheme satisfying the conditions in \refno{S}, and moreover satisfies the Beilinson-Soulé vanishing conjecture as in \textup{\refeq{BS.vanishing}}.

Moreover, $\iota\co X^+ \r X$ is a cellular Whitney-Tate stratified ind-scheme in the sense of Definitions \ref{defi--cellular}, \ref{defilemm--Whitney.Tate.condition}.
\xnota

Recall from \cite[Prop.~1.2.1.16]{Lurie:HA},
that a t-structure amounts to giving a full subcategory
$\C^{\ge 0} \subset \C$ which is closed under extensions and such that the inclusion functor $\incl$ admits a left adjoint $\tau^{\ge 0}$.
An exact functor $F$ between two stable categories with t-structures is \emph{left t-exact} if it preserves the ``$\ge 0$''-subcategories.
The subcategory $\C^{\le -1} := \ker \tau^{\ge 0} = \{M \in \C, \tau^{\ge 0} M = 0\}$ agrees with
the subcategory spanned by the objects $M$ such that the mapping space $\Hom_\C(M, \C^{\ge 0})$ is contractible.
Then $F$ is called \emph{right t-exact} if it preserves the ``$\le -1$'' subcategories.
A functor is \emph{t-exact} if it is both right- and left-t-exact.

\subsubsection{The motivic t-structure}
In this section we introduce the motivic $t$-structure and use it to cut down the category $\DTM(X)$, for a Whitney-Tate stratified scheme $X$, to the abelian category $\MTM(X)$ of stratified mixed Tate motives.
The consideration of Tate motives in this paper has two reasons:\smallskip\\
i) The motivic $t$-structure on $\DM(X)$, predicted by the ``standard'' conjectures seems to be out of reach at the moment.
By contrast, this $t$-structure is known to exist for the subcategory $\DTM(X) \subset \DM(X)$, under some (rather severe) restrictions on the base scheme $X$, see \refle{t-structure.stratum}.\smallskip\\
ii) The realization functor is known to be conservative on Tate motives, see \refle{Tate.conservative}.\smallskip\\
Recall that the \emph{Beilinson-Soulé vanishing conjecture} holds for $S$ if
$$
\Hom_{\DM(S)}(1, 1(n)[m]) = K_{2n-m}(S)^{(n)}_\Q \stackrel ! = 0 \eqlabel{BS.vanishing}
$$
holds for $m < 0$ and also for $m = 0$ and $n \ne 0$.

\exam\label{exam--BS.exam}
The Beilinson--Soulé conjecture is known by Quillen's, Borel's, and Harder's work, for $S$ being the spectrum of a finite field, a number field or localizations of its ring of integers, and finally for a smooth curve over a finite finite or its function field.
Since $K$-theory commutes with filtered colimits, it also holds for filtered colimits of such rings; for example, for the algebraic closures $\ol{\bbF}_p$ or $\ol \Q$.
See the references cited in \cite[Rmk.~3.10]{SoergelWendt:Perverse} and \cite[1.6]{DeligneGoncharov:Groupes}.
\xexam

The following lemma is a mild extension of \cite[Thm.~3.7]{SoergelWendt:Perverse}:
we consider infinite disjoint unions of cellular schemes, and we also allow non-compact objects.
The proof is the same as in loc.~cit., using in addition that a $t$-structure on a stable \ii-category yields a $t$-structure on its ind-completion.
\nts{This fact is proved in Proposition 1.1.3 in a note by Campbell ``Ind-coherent sheaves I''.}

\lemm
\label{lemm--t-structure.stratum}
With $S$ as in \refno{BS.vanishing}, every cellular $S$-scheme also satisfies the Beilinson-Soulé vanishing conjecture. If $X^+ = \bigsqcup_{w\in W} X_w$ is a possibly infinite disjoint union of cellular schemes, then one has:\smallskip\\
i\textup{)} The category $\DTM(X^+)$ carries a unique $t$-structure such that both the dualising functor $\Du_{X^+}$ and the twisting functors $M \mapsto M(n)$, $n \in \Z$ are exact.
Equivalently, the heart is the subcategory of $\DM(X)$ which is generated by means of extensions and arbitrary coproducts by the objects
$$1_{X_w}(n)[d_w], \ (n \in \Z), \eqlabel{motivic.t.structure}$$
where $d_w\co|X_w|\to\bbZ_{\geq 0}$ is the relative dimension of $X_w\to S$ viewed as a locally constant function. The $t$-structure is called the motivic $t$-structure, and its heart $\MTM(X^+)$ is the $\bbQ$-linear abelian category of mixed Tate motives.
This t-structure restricts to one on the subcategory $\DTM(X)^\comp$ of compact objects.
\smallskip\\
ii\textup{)}
The category $\DTM(X^+)$ has a weight structure whose heart is generated -- by means of coproducts and extensions -- by the objects $1(n)[2n]$, $n \in \Z$.
\smallskip\\
iii\textup{)}
The $t$-structure is transversal to the weight structure in the sense of \cite[§1.2]{Bondarko:WeightHearts}. In particular, any compact object $M \in \MTM(X^+)^\comp$ has a functorial \emph{weight filtration}, i.e., a finite sequence of subobjects
$$0 = M_0 \subset M_1 \subset \dots \subset M_n = M,$$
with $M_i \in \MTM(X^+)^\comp$ such that $M_i / M_{i-1}$ is a finite direct sum of $1_{X_w}[d_w]((d_w-i)/2)$ \textup{(}if $d_w - i$ is odd, this object is to be interpreted as 0\textup{)} for $w\in W$.
\xlemm

\nts
{\pf
Replacing $X$ by its connected components, the category $\DTM(X)$ is a Tate category in the sense of \cite[Definition 1.1]{Levine:Tate}.
Moreover, the vanishing conjection in \cite[Theorem 1.5]{Levine:Tate}
$$\Hom_{\DM(X)}(1, 1(p)[q])=0$$
for $p > 0$ and $q \le 0$ is satisfied.
For $X=S$, this is just Beilinson-Soulé, because of \refsy{motives}xiii).
If $f: X = \bbV(\calE) \x \bbV^\x (\calE_1) \r S$, $e_1 := \rk \calE_1$, use the cofiber sequence
$1_S(e_1)[2e_1-1] \r f_\sharp 1 \r 1_S$ in order to show that the conditions in \cite[Definition 1.1]{Levine:Tate} for $X$ boil down to the same conditions for $S$.
(It is not necessary that the extension splits for the argument to work.)
Finally, do an induction on the number of strata: if $Z \stackrel i \r X \stackrel j \gets U$, $Z$ is a (smooth) stratum of codimension $c$, we have a distinguished triangle $j_! 1_U \r 1_X \r i_! i^! 1 = i_* 1(-c)[-2c]$. This leads to a distinguished triangle $M(U) \r M(X) \r M(Z)(-c)[-2c]$.
Altogether $p_\sharp p^* 1$ is an extension of motives of the form $1(n)[2n]$, $n \in \Z$, whence the claim.
\xpf
}

The $t$-structures on Tate motives on individual strata can be glued together:

\coro
\label{coro--t-structure}
The category $\DTM(X, X^+)$ carries the \emph{motivic $t$-structure} which is glued from the motivic $t$-structures on the strata $X^+=\bigsqcup_{w\in W}X_w$ \textup{(}cf.~\refle{t-structure.stratum}\textup{)}. That is, the motivic $t$-structure satisfies
\begin{eqnarray}
\DTM(X)^{\le 0} & = & \{M \in \DTM(X), \iota^*M \in \DTM(X^+)^{\le 0}\} = \{M, \iota_w^* M \in \DTM(X_w^+)^{\le 0} \text{ for all }w\}, \nonumber
\\
\DTM(X)^{\ge 0} & = & \{M \in \DTM(X), \iota^! M \in \DTM(X^+)^{\ge 0}\} = \{M, \iota_w^! M \in \DTM(X_w^+)^{\ge 0} \text{ for all }w\}. \nonumber
\end{eqnarray}
The heart $\MTM(X) = \MTM(X, X^+)$ of the motivic $t$-structure is the category of \emph{\textup{(}stratified\textup{)} mixed Tate motives}. It is abelian and $\bbQ$-linear.
This $t$-structure is conservative \textup{(}a.k.a.~non-degenerate\textup{)}: an object $M \in \DTM(X)$ is 0 if and only if ${\motH}^i M=0$ for all $i\in \bbZ$ \textup{(}truncation with respect to the motivic $t$-structure\textup{)}.
This t-structure restricts to one on the subcategory $\DTM(X)^\comp$ of compact objects.
\xcoro

\pf
As in \cite[Thm.~10.3]{SoergelWendt:Perverse}, we can glue the individual motivic $t$-structures on the strata $X_w$, using \cite[Thm.~1.4.10]{BBD}.
The conservativity of the t-structure is immediate from the fact that $M=0$ iff $\iota_w^* M = 0$ for all $w\in W$ iff $\iota_w^! M = 0$ for all $w\in W$ and the conservativity of the $t$-structure on $\DTM(X_w)$.
\xpf

\rema
\label{rema--classical.t-structure} In the proof of \refpr{equivariant.MTM}, we will also use the \emph{classical $t$-structure} on $\DTM(X)$.
The name ``classical'' refers to the analogy with the classical (resp.~standard) $t$-structure on sheaves.
For an individual stratum $X_w$, $w\in W$, we declare the classical $t$-structure on $\DTM(X_w)$ to be the unique $t$-structure such that $1_{X_w}(n)$, $n \in \Z$, is in its heart. This differs from \refeq{motivic.t.structure} only by a shift. Then, the same way as for the motivic $t$-structure, the classical $t$-structure on $\DTM(X)$ is obtained by glueing the $t$-structures on the strata $X_w$.
If necessary, we distinguish between these $t$-structures by using ``$\cl$'' (e.g., $\tau^{\ge 0, \cl}$ for the classical one), and ``$\mot$'' for the motivic one.
\xrema

The following lemma allows to lift the exactness properties of the six functors from the $\ell$-adic world to stratified Tate motives.

\lemm
\label{lemm--Tate.conservative}
We assume that $S$ \textup{(}in addition to the conditions in \refno{BS.vanishing}\textup{)} admits an $\ell$-adic realization functor $\rho_\ell$ \textup{(}see \refsy{motives} xvii\textup{)}\textup{)}.
The restriction of $\rho_\ell$ to Tate motives,
$$\rho_\ell\co \DTM(X) \r \D_\et(X, \Ql)$$
is conservative and creates the motivic $t$-structure from the perverse $t$-structure on $\D_\et(X, \Ql)$
\textup{(}i.e., $M \in \DTM(X)^{\le 0}$ is equivalent to $\rho_\ell(M) \in \D_\et(X, \Ql)^{\le 0}$ and likewise for ``$\ge 0$''\textup{)}.
It also creates the classical $t$-structure on $\DTM(X)$ from the classical $t$-structure on the right hand category.
In particular, for $M \in \DTM(X)$ the following properties a\textup{)}-b\textup{)} are equivalent:\smallskip\\
a\textup{)} $M \in \MTM(X)$ \textup{(}resp.~$M$ is in the heart of the classical $t$-structure\textup{)},\smallskip\\
b) $\rho_\ell (M)$ is an $\ell$-adic perverse sheaf \textup{(}resp. $\rho_\ell (M)$ is an $\ell$-adic `honest' sheaf\textup{)}.
\xlemm

\pf
Since $\iota^*$ is conservative and creates the t-structure, we may replace $X$ by a connected component of $X^+$ and assume $X$ is a cellular scheme.
Now, the restriction to compact objects, $\rho_\ell|_{\DTM(X)^\comp}$ is exact (both for the motivic and the classical t-structure), which follows right from the definitions.
This implies the exactness also on the Ind-completion of $\DTM(X)^\comp$, which is $\DTM(X)$.
\nts{Here we are using the following lemma: if $F: \C \r \calD$ is an exact functor, $\D$ is cocomplete, then the natural functor $F: \Ind(\C) \r \D$ is also exact: clearly $(\Ind \C)^{\ge 0} = \Ind(\C^{\ge 0})$ is preserved.
Also $(\Ind \C)^{\le 1} := \ker (\tau^{\ge 0} : \Ind \C \r \Ind \C^{\ge 0})$ is preserved, since $\tau^{\ge 0}$ commutes with $F$,
since they are both continous (preserve filtered colimits), and continuous functors out of $\Ind \C$ are determined by their restriction to $\C$, and there they do commute since $F$ is exact.}

To check conservativity of $\rho_\ell$ we note that the t-structure on $\DTM(X)^\comp$ is non-degenerate.
This implies that the t-structure on its Ind-completion is again non-degenerate.
By \cite[Prop.~1.3.7]{BBD}, the family of the cohomology functors $\H^i: \DTM(X) \r \MTM(X)$ ($i \in \Z$) is therefore conservative, so it is enough to show the conservativity of $\rho_\ell |_{\MTM(X)}$.
Any mixed motive $M \in \MTM(X)$ is the filtered colimit of its compact subobjects $N \subset M$.
Using the exactness of $\rho_\ell$, $\rho_\ell(N) \subset \rho_\ell(M)$, so the conservativity of $\rho_\ell$ on $\MTM(X)^\comp$ implies the one on $\MTM(X)$.
Using \refle{t-structure.stratum}.ii), the conservativity of $\rho_\ell|_{\MTM(X)^\comp}$ can be shown as in \cite[Thm.~3.9]{Wildeshaus:Notes} (which considers the special case $X = \Spec k$).
\xpf

\rema
\label{rema--realization.functor}
Since no $\ell$ is invertible in $\Z$, \refle{Tate.conservative} does not literally apply to $S = \Spec \Z$.
However, \refle{Tate.conservative} can be extended to such cases, too.
More precisely, suppose that\smallskip
\begin{enumerate}
\item[$\bullet$]
$S$ is as in \refno{S}, and satisfies the Beilinson-Soulé vanishing conjecture
\item[$\bullet$]
there are primes $\ell_1 \ne \ell_2$ such that $S[\ell_1^{-1}]$, $S[\ell_2^{-1}]$ satisfy the Beilinson-Soulé vanishing conjecture.
\end{enumerate}
\smallskip
\noindent Then
the functor
$$\DTM(X) \lr \DTM(X[\ell_1^{-1}]) \x \DTM(X[\ell_2^{-1}]) \stackrel{\rho_{\ell_2} \x \rho_{\ell_1}} \lr \D_\et(X[\ell_1^{-1}], \Q_{\ell_1}) \x \D_\et(X[\ell_2^{-1}], \Q_{\ell_2})$$
has the same properties as $\rho_\ell$ in the statement above: it is conservative, and creates the motivic t-structure from the perverse t-structures on the right hand categories.
Note that the first functor is conservative by Zariski descent for $\DM$ and creates the motivic t-structure.
\nts{$j^*$ on $\DTM$ is t-exact for the motivic t-structure and also the pair $(j_1^*, j_2^*)$ is conservative, which implies that the functor also detects the t-structure: an object $M$ is $\ge 0$ iff $\tau^{\le -1} M =0$, applying $(j_1^*, j_2^*)$ commutes with $\tau^{\le -1}$ and is conservative.}

By a slight abuse of language, we still refer to the situation above by saying that ``\emph{$S$ admits an $\ell$-adic realization functor}''.
\xrema

\coro
\label{coro--exactness}
Suppose $S$ admits an $\ell$-adic realization functor. Let $\pi\co X\to Y$ be a Whitney-Tate map of cellular Whitney-Tate stratified ind-schemes. 
If $\pi$ is quasi-finite, then $\pi^*$ and $\pi_!$ \textup{(}resp.~$\pi^!$ and $\pi_*$\textup{)} are right $t$-exact \textup{(}resp.~left $t$-exact\textup{)} for the motivic $t$-structure.
\xcoro
\pf This follows from \refre{realization.functor} and the classical statement for $\ell$-adic sheaves \cite[§4.2.4]{BBD}.
\xpf

The following statement is a refinement of the detection of Tate motives (\refle{smooth.detects.Tate}).

\lemm
\label{lemm--smooth.detects.t-structure}
Suppose $\pi\co X\to Y$ is a stratified map of cellular Whitney-Tate stratified ind-schemes such that $\tilde \pi _\sharp 1 \r 1$ has a section. If $\pi$ is smooth of constant relative dimension $d$, then $\pi^*[d]$ preserves and detects the motivic $t$-structure in the following sense:
$$\eqalign{
\pi^* M[d] \in \DTM(X)^{\ge 0} & \Leftrightarrow M \in \DTM(Y)^{\ge 0} \cr
\pi^* M[d] \in \DTM(X)^{\le 0} & \Leftrightarrow M \in \DTM(Y)^{\le 0}.
}$$
\xlemm

\pf
Using smooth base-change, 
we immediately reduce to $Y = Y^+$ (and hence $\pi = \tilde \pi$), and further to the case where $Y$ is a cellular scheme.
The implication ``$\Leftarrow$'' then reduces to the observation that $\pi^*[d] 1_Y[d_Y] = 1_X[d_X]$ lies in $\MTM(X)$ where $d_Y$ (resp.~$d_X$) is the relative dimension of $Y$ (resp.~$X$).

Conversely, assume $\pi^* M[d] \in \DTM(X)^{\le 0}$ or equivalently, $\pi^* M[d]$ is left orthogonal to all objects in $ \DTM(X)^{\geq 1}$.
Since $\pi^* [d]$ is exact, this includes in particular all $\pi^*[d]M'$ with $M' \in \DTM(Y)^{\geq 1}$.
Then, $\Hom(M, M') \r \Hom(\pi^* M, \pi^* M') = \Hom(\pi_\sharp \pi^* M, M')$ is injective since $\pi_\sharp \pi^* M = \pi_\sharp 1 \t M$ contains $M$ as a direct summand by assumption.
Therefore $M$ is orthogonal to $M'$, i.e., $M \in \DTM(Y)^{\le 0}$.
An analogous argument works for $\ge 0$ which implies the lemma.
\xpf

\lemm
\label{lemm--pi*.fullyfaithful.MTM}
Let $\pi : X \r Y$ be a smooth surjective map of schemes of relative dimension $d$ with \emph{connected} fibers.
Let $Y$ be equipped with a cellular Whitney-Tate stratification, and equip $X$ with the preimage stratification \textup{(}\refex{basic.WT} i\textup{))}.
If $X$ is a cellular $Y$-scheme \textup{(}\refde{cellular}\textup{)}, then
$$\pi^![-d] = \pi^*[d](d)\co \MTM(Y) \r \MTM(X)$$
is a fully faithful functor.
\xlemm

\pf
We may assume that $Y$ (and hence $X$ \StP{0378}) is connected.
We proceed by induction on the number of cells in $X$.
If $X$ is a single cell (over $Y$), then localization, homotopy invariance, and the Beilinson--Soulé vanishing for $Y$ show the following isomorphism for $M, N \in \MTM(Y)$:
$$\Hom_{\DM(X)}(\pi^* M, \pi^* N) = \Hom_{\DM(Y)}(M, N).\eqlabel{Hom.X.Y}$$
For the inductive step consider a minimal cell $Z \stackrel i \r X \stackrel j \gets U := X \backslash Z$.
Since $X$ is connected, and has at least two strata by assumption, the codimension $c := \codim_X Z > 0$ is positive.
We use relative purity to compute the localization triangle: $i_! i^! \pi^* N = i_* i^* \pi^* N (-c)[-2c] \r \pi^* N \r j_* j^* \pi^* N$.
Let $\pi_Z := \pi \circ i$, $\pi_U := \pi \circ j$.
Applying $\Hom_{X}(\pi^* M, -)$ gives a $4$-term exact sequence
\begin{align*}
& \Hom_Z(\pi_Z^* M, \pi_Z^* N(-c)[-2c]) \r \Hom_{X}(\pi^* M, \pi^* N) \r \Hom_U(\pi_U^* M, \pi_U^* N)  \\
& \r \Hom_Z(\pi_Z^* M, \pi_Z^* N(-c)[-2c+1]).
\end{align*}
The outer terms vanish by the Beilinson--Soulé condition for $Z$ since $-2c+1 < 0$ which follows from the corresponding condition for $Y$ by \refle{t-structure.stratum}.
\xpf

\subsubsection{Equivariant mixed Tate motives}

\defi
\label{defi--MTM.G}
Let $G$ be an ordinary presheaf of groups acting on $X$ over $S$.
We define the category of \emph{$G$-equivariant mixed Tate motives} as the homotopy pullback
$$\MTM_G(X) \defined \MTM(X) \x_{\DM(X)} \DM(G\backslash X),$$
or equivalently as the full subcategory of $\DM(G\backslash X)$ (\refde{DM.G}) of those objects whose underlying motive is a mixed Tate motive, with respect to the given stratification on $X$.
\xdefi

The preceding definition works even if the $G$-action is incompatible with the stratification. However, to prove that $\MTM_G(X)$ is abelian we need that the $G$-action respects the stratification:

\prop
\label{prop--MTM.G}
In the situation of \refpr{DTM.G}, assume in addition that each $G_i$ is cellular.
Then $\DTM_G(X)$ admits a $t$-structure such that the forgetful functor $\DTM_G(X)\r \DTM(X)$ is $t$-exact.
Its heart identifies with the category $\MTM_G(X)$ defined above.
It is a $\bbQ$-linear abelian category and can be computed as
$$\MTM_G(X) = \colim_i \lim_{j \ge i} \MTM_{G_j}(X_i).$$
The colimit is taken in the bicategory of cocomplete categories and continuous functors \textup{(}or in $\DGCat_\cont$\textup{)}. Transition functors are the pushforwards along $X_i \r X_{i'}$.
The limit is formed using the restriction functors along $G_j \r G_{j'}$.

The category $\MTM_G(X)$ is compactly generated, i.e., $\MTM_G(X) = \Ind(\MTM_G(X)^\comp)$ is the ind-completion of the subcategory of compact objects.
The latter is given by a similar formula, namely
$$\MTM_G(X)^\comp = \colim_i \lim_{j \ge i} \MTM_{G_j}(X_i)^\comp,$$
where now, however, the colimit is taken in the bicategory of categories with not necessarily continuous functors.
\xprop

\rema
Colloquially speaking, a compact mixed $G$-equivariant Tate motive on $X$ is therefore simply a $G_j$-equivariant mixed Tate motive on some $X_i$, where $j \ge i$ is arbitrary.
\xrema

\pf By \refpr{DTM.G}, we have
\begin{equation}\label{Bar_Construction}
\DTM_G(X) \,=\, \colim_i \lim_{j \ge i} \lim \DTM^!(\BarC(G_j, X_i)).
\end{equation}
As explained there, the outer colimit is using *-pushforward along the closed immersions $t_{i,i'}: X_i \r X_{i'}$.
The middle limit uses the restriction functors (which are equivalences) and the functors in the limit are !-pullbacks (along the maps $f$ in $\BarC(G_j, X_i)$, which are smooth maps).
The latter limit does not change up to equivalence if we replace these pullback functors $f^!$ by $f^![-d]$, where $d$ is the relative dimension of $f$ (which is finite, since $G_j$ is of finite type).
Both these shifted transition functors and also the $(t_{ij})_*$ are t-exact by Lemma \ref{lemm--smooth.detects.t-structure} and \refco{exactness}, respectively.
Using \refle{t.structure.limit}, we get a t-structure whose heart is as claimed above.
\nts{If we were to include the section maps in $\BarC(G_i, X_i)$, we could argue like this: for maps of the form $e^*[-g_i]$, $g_i$ the relative dimension of the $S$-scheme $G_i$, we use that the generators $\iota_{1, *} 1_{G \x X^+}(n)[g_i+ \dim X^+]$ of $\DTM(G \x X)^{\ge 0}$ are mapped to $e^* \iota_{1,*} 1(n)[\dim X^+] = \iota_* 1(n)$, which lie in $\DTM(X)^{\ge 0}$ and similarly with ``$\le 0$''. The commutation of $e^*$ and $\iota_*$ works here since $e^* \iota_{1,*} 1 = e^* \iota_{1,*} p_2^{+,*} 1_{X^+} = e^* p_2^* \iota_* 1 = \iota_* 1$, where we use that $p_2$ is smooth to commute $p_2^*$ past $e_*$.}
\xpf

\lemm
\label{lemm--t.structure.limit}
Consider a diagram of stable $\infty$-categories $\C_i$ and exact functors $F_{ij}\co \C_i \r \C_j$ between them.
We suppose that all $\C_i$ are equipped with $t$-structures $\C_i^{\ge 0}$ \textup{(}we use cohomological notation\textup{)} and that the functors $F_{ij}$ are left t-exact, i.e., preserve the ``$\ge 0$''-subcategories.
Let
$$\C := \lim_{F_{ij}} \C_i.$$
\nts{\textup{(}recall from \cite[§1.2.1]{Lurie:HA} that this means their homotopy categories have $t$-structures in the sense of \cite{BBD}\textup{)}.}

\begin{enumerate}
\item
\label{item--t-structure.lim}
The subcategory $\C^{\ge 0} := \lim_{F_{ij}} \C_i^{\ge 0}$ on $\C$ determines a t-structure.
\textup{(}Thus, $M \in \C^{\ge 0}$ iff all the projections $p_i (M) \in \C_i$ under the canonical maps $p_i: \C \r \C_i$ lie in $\C_i^{\ge 0}$.\textup{)}
If the $F_{ij}$ are in addition right t-exact, then the $p_i$ also create the ``$\le 0$'' part of the t-structure.

\item
\label{item--t-structure.colim}
Suppose in addition that $I$ is filtered, that the $\C_i$ are presentable, and that the $F_{ij}$ have left adjoints $G_{ij}: \C_j \r \C_i$ which are t-exact and fully faithful.
Using \refit{t-structure.lim} and \refle{Lurie.co.limit}, we consider the induced t-structure on $\colim_{G_{ij}} \C_i \stackrel \cong \r \C$.
Here, $\colim$ denotes the colimit in the \ii-category of presentable \ii-categories with continuous functors.
Then the canonical functors $\C_i \r \colim_{G_{ij}} \C_i$ are t-exact.
\nts{The assumption that the $G_{ij}$ are fully faithful should be unnecessary, provided that the t-structures are compactly generated.}
\end{enumerate}
\xlemm

\pf
\refit{t-structure.lim}:
The adjunctions $(\tau_i^{\ge 0}, \incl_i)$ for the $\C_i$'s propagate to one on $\C$, and the limit of the inclusion functors is fully faithful since the counit map of the adjunction for $\C$ is the limit of the counit maps for the $\C_i$, hence an equivalence, so that $\incl$ is indeed fully faithful.
The subcategory $\C^{\ge 0} \subset \C^0$ is closed under extensions since the $p_i$ preserve finite (co)limits, in particular extensions.
If the $F_{ij}$ are right exact, then the full subcategory $\C^{\le -1} := \lim_{F_{ij}} \C_i^{\le -1} \subset \C$ agrees with $\ker (\tau^{\ge 0} = \lim \tau_i^{\ge 0} : \C \r \C^{\ge 0})$.
\nts{Here, use that an object $M \in \lim \C_i$ is 0 iff all the $p_i(M)=0$. Indeed, $M=0$ iff $\id_M = 0$ in the mapping space $\Hom(M, M)$.
This implies $\id_{p_i(M)} = p_i(\id_M) = 0$. Conversely $p_i(M) = 0$ clearly implies $M=0$.
Now, let $M \in \C$. Write $M = (M_i)$, $M_i = p_i(M)$.
Then $M \in \ker \tau^{\ge 0}$ iff $(\tau^{\ge 0}_i M_i) = \tau^{\ge 0} (M) = 0$ iff all the $\tau^{\ge 0}_i M_i = 0$ iff $M_i \in \C_i^{\le -1}$.}

\refit{t-structure.colim}:
By \refle{Lurie.co.limit}.\refit{Lurie.co.limit.map} and the full faithfulness assumption, the composition $\ins_i: \C_i \r \colim_{G_{ij}} \C_i \cong \lim_{F_{ij}} \C_i$ satisfies $p_j (\ins_i (M)) = F_{ij} (M)$ for $j < i$ and $G_{ij} (M)$ for $j \ge i$.
If $M \in \C_i^{\ge 0}$, then both $F_{ij}(M)$ and $G_{ij}(M)$ are in $\C_j^{\ge 0}$ using the left exactness assumptions. Thus $\ins_i$ is left t-exact.
To show the right t-exactness of $\ins_i$, we show that $\Hom_\C(\ins_i M, N) = \lim_j \Hom_{\C_j} (p_j(\ins_i M), p_j (N)) = 0$ for any $M \in \C_i^{\le -1}$ and $N \in \C^{\ge 0}$.
We may restrict the limit to $j > i$, in which case the $j$-th term reads $\Hom_{\C_j}(G_{ij} M, p_j(N))$ which is contractible since $p_j(N) \in \C_j^{\ge 0}$ and $G_{ij}(M) \in \C_j^{\le -1}$.
\xpf

\rema
\label{rema--classical.equivariant}
To connect the category $\MTM_G(X)$ more closely to classical notions, we suppose for simplicity of notation that $G$ is algebraic and $X$ a scheme.
The proof above shows that
$$\MTM_G(X) = \lim \left (\MTM(X)\xymatrix{ \ar@<-0.5ex>[r]_{p_2^![-d]}  \ar@<0.5ex>[r]^{a^![-d]} &} \MTM(G \x_S X) \xymatrix{ \ar@<1ex>[r] \ar@<-1ex>[r] \ar[r] &} \MTM (G \x_S G \x_S X) \xymatrix{ \ar@<1.5ex>[r] \ar@<-1.5ex>[r] \ar@<0.5ex>[r] \ar@<-0.5ex>[r] &} \cdots \right),$$
where $d = \dim G/S$.
Using the t-exactness of $e^![d]$, $e: X \r G \x_S X$, and the cofinality of the subcategory $(\Delta^+)^\opp$ of injective maps in $\Delta^\opp$\nts{\cite[Lemma 6.5.3.7]{Lurie:Higher}}, we may equivalently form the limit over the full cosimplicial diagram also involving the maps built using the unit sections.
On the other hand, since $\MTM$ is an ordinary category, the limit does not change if we drop all terms $\MTM(G^{\x n} \x X)$ for $n \ge 3$.
Thus, an object in $\MTM_G(X)$ is a datum
$$(M_0, M_1, M_2, \varphi_e, \varphi_a, \varphi_{p_2}, \varphi_{\id \x a}, \varphi_{m \x \id}, \varphi_{p_{23}})$$
where $M_0 \in \MTM(X)$, $M_i \in \DM(G^i \x X)$ for $i=1,2$ and $\varphi_a : a^![d] M_0 \r M_1$, $\varphi_{\id \x a} : (\id \x a)^![d] M_1 \r M_2$ etc.~are equivalences.
Up to equivalence, we may further replace $a^![d]$ by $a^*[-d]$, which description is equivalent to the standard definition of say equivariant perverse sheaves in \cite[III.15, p.~187]{KiehlWeissauer:Weil}.
\xrema

The following proposition is a motivic variant of a well-known statement about equivariant perverse $\ell$-adic sheaves \cite[III.15, p.~188]{KiehlWeissauer:Weil}.
It allows us to easily construct equivariant mixed Tate objects, since we only have to check the existence of an isomorphism, as opposed to verifying higher coherences.

\prop
\label{prop--equivariant.MTM}
Let $G$ be a \emph{fibrewise connected} smooth $S$-affine $S$-group whose underlying scheme is cellular.
Let $X$ be a cellular Whitney-Tate stratified scheme with a stratified $G$-action.
Then the forgetful functor
$$\MTM_G(X) \r \MTM(X)$$
is fully faithful and its image consists precisely of those motives $M \in \MTM(X)$ such that there is an isomorphism $a^! M[-\dim G] \cong p^! M[-\dim G]$ in $\MTM(G \x_S X)$ where $a\co G\x_S X\to X$ \textup{(}resp.~$p\co G\x_S X\to X$\textup{)} is the action \textup{(}resp.~projection\textup{)}.
We call this condition the naive equivariance condition.
\xprop

\pf
Fix $n$, and denote $H:=G^{\x_S n}$.
Applying \refle{pi*.fullyfaithful.MTM} to the projection map $p\co  H\x_S X \r X$ we see that $p^!\co \MTM(X) \r \DM(H \x_S X)$ is fully faithful.
This implies that $!$-pullback along the unit map $X \r H \x_S X$ is fully faithful, hence our claim by \refle{lim.equivalence}.
\xpf

\coro
\label{coro--MTM.G.trivial}
In the situation of \refpr{equivariant.MTM}, suppose $G$ acts trivially on $X$. Then there is an equivalence $\MTM_G(X) \r \MTM(X)$.
We therefore get
$$\H^1_{\mathrm{mot}}(BG, \Q) \defined \Hom_{\DM(G \setminus S)}(1, 1[1])=\Ext^1_{\MTM_G(X)}(1, 1) = \Ext^1_{\MTM(S)}(1, 1) = (K_{-1}(S) \t \Q)^{(0)} = 0.$$
\xcoro

\pf
For a trivial action, the condition $a^! M \cong p^! M$ is vacuous.
We are done using that $\Ext^1$ in the heart of an t-structure agrees with homomorphisms in the original triangulated category \cite[(1.1.5)]{DeligneGoncharov:Groupes}.
The identification with $K$-theory uses the regularity of $S$.
\xpf

In the following proposition we do not need to assume that $G$ or $H$ is cellular since we just work with the definition of $\MTM_G(X)$ and only use the t-structure on non-equivariant motives.

\prop
\label{prop--ballaballa}
Suppose we are in the situation of \refpr{DTM.G/H}: $H\subset G$ is an inclusion of ordinary $\tau$-sheaves of $S$-groups, $X := (G/H)^\tau$ be the quotient of $\tau$-sheaves which we assume to be a smooth finite type $S$-scheme equipped with a cellular Whitney-Tate stratification such that the base point $S \r X$ factors over $X^+$.
Let $d := \dim X / S$, and denote by $e\co S/H \r G\backslash X$ be the map induced from the base point.
Then the equivalence in \refpr{DTM.G/H} restricts to an equivalence
$$e^! [d]\co \MTM_G(X) \overset{\simeq}{\lr} \MTM_H(S).$$
\xprop

\pf
We construct the inverse of the equivalence.
Let $\pi\co X\to S$ be the structure map.
The functor $\pi^![-d]$, being $t$-exact and conservative, creates the motivic $t$-structure, i.e., for $M \in \DTM(S)$ we have $\pi^![-d] M \in \MTM(X)$ iff $M \in \MTM(S)$.
We then conclude using \refpr{DTM.G/H}.
\xpf

We now obtain a convenient description of generators of certain equivariant categories of Tate motives.

\prop
\label{prop--generators.DTM.G}
Let $H \subset G$ be an inclusion of smooth $S$-affine $S$-groups, and suppose that the \'etale sheaf quotient $X=(G/H)^\et$ is a scheme equipped with a cellular Whitney-Tate stratification such that the base point $S \r X$ factors over $X^+$. Further, assume that $H$ is fibrewise connected and that its underlying scheme is cellular.
Then the shifted !-pullback along the map $S \r H \backslash S \r G / X$ induces an equivalence
$$\MTM_G(X) \stackrel \cong \r \MTM(S).$$
In particular, $\MTM_G(X)$ \textup{(}resp.~$\DTM_G(X)$\textup{)} is generated by means of coproducts and extensions \textup{(}resp.~by means of colimits and arbitrary shifts\textup{)} by motives of the form $1_{X}(n)[d]$, $n \in \Z$, $d:=\dim X/S$.
\xprop

\pf
This follows from \refpr{ballaballa}, \refco{MTM.G.trivial} and the fact that $\MTM(S)$ is generated by the motives $1(n)$.
\xpf

We obtain the following corollary in the case $X$ has several strata.

\coro
\label{coro--DTM.G.X.generators}
Suppose a smooth $S$-affine $S$-group $G$ acts on a finite type scheme $X$ over $S$.
Suppose that it carries a cellular Withney-Tate stratification $X^+= \sqcup_{w\in W}X_w\to X$ where each stratum has the form $X_w=(G/H_w)^\et$ and satisfies all conditions in \refpr{generators.DTM.G} \textup{(}in particular $H_w$ is fibrewise connected\textup{)}.
We write $\iota_w\co G \backslash X_w \r G \backslash X$ for the map of prestacks \textup{(}whose étale sheafifications are Artin stacks\textup{)} induced by the strata inclusions.
Then $\DTM_G(X)$ is generated, by means of colimits and shifts by the objects $(\iota_w)_! 1(n)$, where $\iota_w: X_w \r X$ is the inclusion and $n \in \Z$, and $(\iota_w)_!$ is the left adjoint of $\iota_w^!$.
\xcoro

\pf
Let $X_0 \subset X$ be an open stratum, and its complement $X_1$ stratified by the remaining strata.
By \refle{functoriality.equivariant}, we have the following adjoints (left adjoints are depicted above their right adjoints)
$$\xymatrix{
\DM(G \backslash X_1) \ar[rr]|{i_!} & &
\DM(G \backslash X) \ar[rr]|{j^!} \ar@/_1pc/[ll]_{i^*} \ar@/^1pc/[ll]^{i^!} & &
\DM(G \backslash X_0) \ar@/_1pc/[ll]_{j_!} \ar@/^1pc/[ll]^{j_*}
}.$$
Moreover, $(i^!, j^!)$ is conservative and $i^* i_! = \id$, $j^! j_! = \id$, so that the localization cofiber sequence in \refeq{localization} carries over to the equivariant setting.
By construction, the functors have their usual meaning if we forget the $G$-equivariance, so they preserve the subcategories $\DTM_G(\str) \subset \DM(G \backslash \str)$.
Thus $\DTM_G(X)$ is generated by $j_! \DTM_G(X_0)$ and $i_! \DTM_G(X_1)$.
This allows an induction on the number of strata, the case of a single stratum being \refpr{generators.DTM.G}.
\xpf

\subsection{Simple objects}

In this section, $S$ is a scheme satisfying the conditions in \refno{S} which is moreover regular, satisfies the Beilinson-Soulé vanishing conjecture as in \textup{\refeq{BS.vanishing}} and admits an $\ell$-adic realization functor $\rho_\ell$ in the sense of \refre{realization.functor}.

In order to describe the simple objects in the category $\MTM(X)$, we need to introduce the middle extension functor $j_{!*}$. This follows closely the classical theory \cite{BBD}.

Let $\iota\co X^+=\bigsqcup_{w\in W}X_w\to X$ be a cellular Whitney-Tate stratified $S$-scheme of finite type, and denote $\iota_w:=\iota|_{X_w}$. Let $j\co U\to X$ be an open immersion, and assume $U^+:=U\x_XX^+=\bigsqcup_{w\in W_U}X_w$ for some subset $W_U\subset W$. In particular, $U$ is Whitney-Tate, and $j$ is a Whitney-Tate map of cellular stratified Whitney-Tate schemes.

For an object $A\in \MTM(U)$, we have $j_!A\in \DTM^{\leq 0}(X)$ (resp. $j_*A\in \DTM^{\geq 0}(X)$) by definition of the $t$-structures. Hence, the natural map $j_!A\to j_*A$ factors as $j_!A\to {\motH}^0j_!A\to {\motH}^0j_*A\to j_*A$ where ${\motH}^0$ denotes the $0$-th truncation with respect to the motivic $t$-structure on $\DTM(X)$.

\defi
The \emph{middle extension} of $A$ along $j$ is defined as the image
\begin{equation}\label{middle-extension}
 j_{!*}A\defined \on{im}({\motH}^0j_!A\to {\motH}^0j_*A)\in \MTM(X).
\end{equation}
\xdefi

By \refco{t-structure}, $j_{!*}$ preserves compact objects.

\lemm
\label{lemm--middle.extension}
Under the $\ell$-adic realization \textup{(}cf.~Lemma \ref{lemm--Tate.conservative}\textup{)}, one has $\rho_\ell(j_{!*}A)\simeq {j}_{!*}(\rho_\ell(A))$. In particular, the middle extension $j_{!*}A$ is the unique extension $B\in \DTM(X)$ of $A$ with the property
\[
\text{$^{\cl}\H^i (\iota_w^*\rho_\ell(B))=0$ for $i\geq -\dim(X_w / S)$\;\; and\;\; $^{\cl}\H^i (\iota_w^!\rho_\ell(B))=0$ for $i\leq -\dim(X_w / S)$,}
\]
where $^{\cl}\H^i$ denotes the $i$-th cohomology with respect to the classical $t$-structure on $\D_\et(X, \Ql)$.
\xlemm
\pf
This follows from the parallel $\ell$-adic statement \cite[Prop 2.1.9]{BBD} and \refre{realization.functor} because the functors ${\motH}^0$, $j_!$ and $j_*$ commute with $\rho_\ell$, as does the formation of images and kernels in the abelian category $\MTM(X)$.
\xpf

\lemm
\label{lemm--intermediate.simple}
If the $S$-scheme $X$ is cellular \textup{(}hence smooth\textup{)} of relative dimension $d$, then one has $1_X[d]=j_{!*}1_U[d]\in \MTM(X)$. Additionally, if $X$ is irreducible, then $1_X[d]$ is a simple object in $\MTM(X)$.
\xlemm

\pf
Since $X$ (and all the strata $X_w$) are smooth, one has $\iota_w^! 1 = 1(-\codim_X X_w)[-2\codim_X X_w]$, so the first claim follows from \refle{middle.extension}.
If $X$ is irreducible, then $1_X[d]$ is simple by adapting \cite[Lemma 4.3.3]{BBD} to our set-up.
\xpf

Let $\iota\co X^+=\bigsqcup_{w\in W}X_w\to X$ be a stratified $S$-ind-scheme. As in \eqref{stratified.Ind-pres}, we write  $X$ as the colimit of the stratified $S$-schemes $\ol{X_w}^+=\bigsqcup_{v\leq w}X_v\to \ol{X_w}$. The map $\iota|_{X_w}\co X_w\to X$ factors as
\[
X_w\overset{j_w}{\longrightarrow} \ol{X_w}\overset{i_w}{\longrightarrow} X,
\]
where $j_w$ is a dense open immersion and $i_w$ a closed immersion. Note that if $X^+\to X$ is cellular Whitney-Tate stratified, then $\ol{X_w}^+\to \ol{X_w}$ is cellular Whitney-Tate stratified.

\defi
In the above situation, the \emph{intersection motive} is defined for each $w \in W$, $n\in \bbZ$ as
\begin{equation}\label{intersection.complex}
\on{IC}_w(n)\defined i_{w,*}\left(j_{w, !*}1_{X_w}(n)[d_w]\right)\in \MTM(X)^\comp,
\end{equation}
where $d_w$ is the relative dimension of the cellular $S$-scheme $X_w$.
(\refco{exactness} and $i_{w,*}=i_{w,!}$ shows $\IC_w(n)$ is a mixed Tate motive.)
\xdefi

\rema
We emphasize that the existence of intersection motives (without assuming any standard conjectures) in this special situation is guaranteed by the cellularity assumptions and the Beilinson--Soulé conjecture.
Another, somewhat orthogonal case where the intersection motive exists is the case of singular proper surfaces \cite{Wildeshaus:Pure}.
\xrema

\theo
\label{theo--simple.objects}
Let $X^+=\bigsqcup_{w\in W}X_w\to X$ be a cellular Whitney-Tate stratified ind-scheme.\smallskip\\
i\textup{)} The category of compact objects $\MTM(X)^{\on{c}}$ is Artinian and Noetherian: every object is of finite length.\smallskip\\
ii\textup{)} If $X_w$ is irreducible for each $w\in W$, then the twisted intersection motives $\on{IC}_w(n)\in \MTM(X)^{\on{c}}$ are simple. Additionally, if $X_w$ is a cell for each $w\in W$, then the simple objects in $\MTM(X)^{\on{c}}$ are precisely the intersection motives $\on{IC}_w(n)$ for $w\in W$, $n\in \bbZ$.
\xtheo
\pf
Again i) is immediate from $\rho_\ell$ being conservative, cf.~\refre{realization.functor}. For ii), we reduce to the case where $X$ is a cellular Whitney-Tate stratified $S$-scheme.
The intersection motives $\IC_w(n)$ are simple because they map to simple objects under the $\ell$-adic realization. As in \cite[\S 4.3.4]{BBD} one can proceed by Noetherian induction to see that every simple object is obtained in this way if each $X_w$ is a cell: Let $j\co U\to X$ be an open stratum with closed complement $i\co X\backslash U\to X$, and assume by induction that ii) holds for objects in $i_*\MTM(X\backslash U)^{\on{c}}\subset \MTM(X)^{\on{c}}$. If $A\in \MTM(X)^{\on{c}}$, then $j^*A\in \MTM(U)^{\on{c}}$ is by construction a successive extension of twisted $1_U[d](n)$ with $d$ being the relative dimension of $U$ and $n \in \Z$ arbitrary (the category $\MTM(U)$ is a category of Tate type, then apply \cite[Thm.~1.4 (iii)]{Levine:Tate}). Hence, the simple constituents of $j_{!*}j^*A$ are of the desired form.
The exact sequences in $\MTM(X)$,
\[
\begin{aligned}
& 0\to i_*{\motH^{-1}}(i^*A)\to \motH^0 j_!j^*A\to A\to i_*{\motH^{0}}(i^*A)\to 0\\
& 0\to i_*{\motH^{0}}(i^!A)\to A\to \motH^0 j_*j^*A\to i_*{\motH^{1}}(i^!A)\to 0
\end{aligned}
\]
give that the cokernel of $j_{!*}j^*A\subset \im(A\to \motH^0 j_*j^*A)$ lies in the category $i_*\MTM(X\backslash U)^{\on{c}}$. Part ii) follows by induction.
\xpf

\section{Loop groups and their flag varieties}\label{sect--loop.grps}

In this section, we study loop groups and their flag varieties associated with Chevalley groups $G$ over $\Z$. We then gather some results about the partial affine flag variety $\Fl=LG/\calP$ associated with a parahoric subgroup $\calP\subset LG$. A final goal is to show that $\Fl$ has the structure of a cellular stratified ind-scheme in the sense of Definition \ref{defi--stratified.dfn}.
Results over general base schemes $S$ are deduced in \refsect{chevalley.base.change} by base change.

\subsection{Group-theoretic notation} \label{sect--loop.group.dfn}
We fix a Chevalley group scheme $G$ over $\Z$, i.e., a smooth affine $\Z$-group scheme whose geometric fibers are connected reductive groups, and which admits a maximal torus defined over $\Z$, cf.~\cite[\S6.4]{Conrad:Groups}. We fix a maximal $\Z$-torus $T\subset G$ which is automatically split, cf.~\cite[Exam.~5.1.4]{Conrad:Groups} (because the Galois module $X^*(T_{\bar{\bbQ}})$ is necessarily unramified, and hence trivial). Let $B\subset G$ be a Borel subgroup defined over $\Z$ and containing $T$. We obtain a Borel triple of smooth $\bbZ$-group schemes
\begin{equation}\label{Chevalley_Triple}
T\,\subset\, B\,\subset\, G.
\end{equation}
{\it i\textup{)} Cocharacters.} There is the natural pairing of finitely generated free $\bbZ$-modules $\lan\str,\str\ran\co X^*(T)\x X_*(T)\,\r\, \bbZ$ where $X^*(T):=\Hom(T,\bbG_{m,\bbZ})$ (resp.~$X_*(T):=\Hom(\bbG_{m,\bbZ},T)$) is the group of characters (resp.~of cocharacters) defined over $\bbZ$. \smallskip\\
{\it ii\textup{)} Roots.} Let $R\subset X^*(T)$ be the roots associated with $(G,T)$, and let $R_+$ the subset of positive roots defined by $B$. \smallskip\\
{\it iii\textup{)} Affine roots.} Let $\scrA:=X_*(T)\otimes \bbR$.
The roots $R$ are regarded as linear maps on $\scrA$.
Adding integers to their values gives the set $\calR:=R+\bbZ$ of affine roots which are then affine linear maps $\scrA\to \bbR$. \smallskip\\
{\it iv\textup{)} Standard apartment.} The vector space $\scrA$ equipped with the simplicial structure defined by the hyperplanes $\ker(\al)$ for $\al\in \calR$ is called \textit{the standard apartment}. The connected components of $\scrA\backslash (\cup_{\al\in\calR}\ker(\al))$ are called {\em alcoves}.
For any alcove $\bba$, its closure $\bar{\bba}$ is a disjoint union of facets $\bbf$ which are locally closed by convention and may have dimension ranging from $0$ to $\dim_\bbR(\scrA)$ (e.g.~if $\bar{\bba}$ is a triangle, then it is decomposed into three vertices, three edges and one alcove).
Thus, $\scrA$ decomposes into a disjoint union of facets.
There is a unique alcove $\bba_0$ called the {\em base alcove} which lies in the chamber of $\scrA$ defined by $R_+$, and which contains $0$ in its closure.
A point in $\scrA$ is called {\em special} if every hyperplane $\ker(\al)$, $\al\in \calR$ is parallel to a hyperplane passing through that point.
The base point $0\in \scrA$ is always special.
Note that if $\ker(\al)$ contains a point $x\in \scrA$, then it contains the unique facet $\bbf$ with $x\in \bbf$.\smallskip\\
{\it v\textup{)} Weyl groups.} Let $W_0$ denote the Weyl group of the root system $R$ which acts on $\scrA$ by linear transformations. The Iwahori-Weyl group (or extended affine Weyl group) $W:=X_*(T)\rtimes W_0$ acts on $\scrA$ by affine linear transformations permuting transitively the set of alcoves in $\scrA$. For each $\al\in \calR$, we have the reflection $s_\al\in W$ along the hyperplane $\ker(\al)$. The group $W$ acts on $\calR$ via $w\al\co \scrA\to \bbR$, $x\mapsto \al(w^{-1}x)$. We have the relation $w s_\al w^{-1}=s_{w\al}$ for all $\al\in\calR$. For each facet $\bbf\subset \scrA$, we denote by $W_\bbf\subset W$ the subgroup generated by the reflections $s_\al$ such that $\bbf\subset \ker(\al)$, i.e., $\al|_\bbf\equiv 0$. We remark that $W_\bbf$ is finite. \smallskip\\
{\it vi\textup{)} Dominant cocharacters.} The monoid of dominant cocharacters is
\begin{equation}\label{dom_weights}
X_*(T)_+\defined \{\la\in X_*(T)\;|\; \lan a,\la\ran\geq 0\; \forall a\in R_+\}.
\end{equation}
This monoid is equipped with the dominance partial order defined by: $\la\leq\mu$ if and only if $\mu-\la$ is a sum of positive coroots with coefficients in $\bbZ_{\geq 0}$.

\subsection{Parahoric subgroups}\label{sect--loop.definitions}

The \emph{loop group $LG$} is the group functor on the category of rings
$$LG\co R\longmapsto G(R\rpot{\varpi}),\eqlabel{Loop_Grp_Dfn}$$
where $R\rpot{\varpi}$ denotes the ring of Laurent series in the formal variable $\varpi$. Since $G$ is affine and of finite type, the loop group $LG$ is representable by an ind-affine ind-scheme, cf.~\cite[\S1.a]{PappasRapoport:LoopGroups} (or \cite[Lem.~3.2]{HainesRicharz:TestFunctionsWeil} in greater generality). In particular, it is an fpqc sheaf on the category of rings, cf.~\S\ref{sect--ind.schemes} below.

We are interested in certain pro-algebraic closed subgroups $\calP\subset LG$, called {\it parahoric subgroups}. These subgroups should be regarded as infinite-dimensional analogues of parabolic subgroups in linear algebraic groups. We first give the guiding examples. The general notion defined in Lemma \ref{lemm--parahoric.defi} is needed in the proof of \refth{Fl.WT}. It requires some Bruhat-Tits theory \cite{BruhatTits:Groups2}.

\exam \label{exam--parahoric}
i) The {\it positive loop group $L^+G$} is the group functor
\[
L^+G\co  R\longmapsto G(R\pot{\varpi}),
\]
where $R\pot{\varpi}\subset R\rpot{\varpi}$ is the subring of formal power series. Then as presheaves $L^+G=\on{lim}_{i\geq 0}G_i$ with $G_i(R)=G(R[\varpi]/(\varpi^{i+1}))$, and hence $L^+G$ is represented by a pro-algebraic $\bbZ$-group, see \S \ref{sect--algebraic.grps} for conventions on pro-algebraic groups. The inclusion $L^+G\subset LG$ is relatively representable by a closed immersion, and makes $L^+G$ a closed $\bbZ$-subgroup functor of $LG$.\smallskip\\
ii) Example i) is generalized as follows. For a standard parabolic subgroup $P\subset G$ (i.e.~$P$ contains $B$), let $\calP\subset L^+G$ (resp.~$\calP_i\subset G_i$) be the preimage of $P\subset G$ under the reduction map $L^+G\to G$, $\varpi\mapsto 0$ (resp.~$G_i\to G_0=G$). Then $\calP=\on{lim}_{i\geq 0}\calP_i$ is a pro-algebraic closed $\bbZ$-subgroup scheme of $LG$. If $P=B$, then the parahoric subgroup $\calB:=\calP$ is called the {\it standard Iwahori subgroup}.
\xexam

Let $k$ be a field. Recall the classical notion of parahoric subgroups in $LG\otimes k$.
Here and below $LG \t k$ is the restriction of $LG$ to the category of $k$-algebras;
it can be computed as $L(G \t k)$.
Let $\bbf\subset \scrA$ be a facet, and let $\calG_{\bbf,k}$ be the associated parahoric $k\pot{\varpi}$-group scheme, that is, the neutral component of the unique algebraic $k\pot{\varpi}$-group scheme such that the generic fiber is $G\otimes k\rpot{\varpi}$, and such that the $k\pot{\varpi}$-points are the pointwise fixer of $\bbf$ in $G(k\rpot{\varpi})$ (under its action on the Bruhat-Tits building).
In particular, this defines for any $k$-algebra $R$ a subgroup $\calG_{\bbf,k}(R\pot{\varpi})\subset G(R\rpot{\varpi})$. The {\it parahoric $k$-subgroup $\calP_{\bbf,k}\subset LG\otimes k$ associated with $\bbf$} is the group functor on the category of $k$-algebras defined by
\begin{equation}\label{parahoric}
\calP_{\bbf,k}\co R\mapsto \calG_{\bbf,k}(R\pot{\varpi}).
\end{equation}
Again $\calP_{{\bbf,k}}$ is a pro-algebraic $k$-group given by the inverse limit of the Weil restriction of scalars $\calP_{{\bbf,k},i}:=\on{Res}_{k_i/k}(\calG_{{\bbf,k}}\otimes_{k\pot{\varpi}} k_i)$ for $k_i=k[\varpi]/(\varpi^{i+1})$. In particular, to any facet $\bbf\subset \scrA$ we have associated the family of pro-algebraic groups $\bbf \mapsto \{\calP_{{\bbf,k}}\}_k$ where $k$ ranges over all fields. For the notion of strictly pro-algebraic groups, we refer the reader to \refsect{algebraic.grps}.

\lemm \label{lemm--parahoric.defi}
For any facet $\bbf\subset \scrA$, there exists a unique flat closed subscheme $\calP_\bbf\subset LG$ such that for every field $k$ one has $\calP_\bbf\otimes k=\calP_{{\bbf,k}}$ as subgroups of $LG\otimes k$. The group scheme $\calP_\bbf$ is a strictly pro-algebraic $\bbZ$-group scheme with connected fibers.
It is called the \emph{parahoric subgroup of $LG$ associated with $\bbf$}.
\xlemm
\pf
By \cite[\S4.2.2]{PappasZhu:Kottwitz} applied with $\bbZ\pot{\varpi}$ as a base ring (cf.~also \cite[Lem.~2.1]{HainesRicharz:Normality}), there exists an algebraic $\bbZ\pot{\varpi}$-group scheme $\calG_\bbf$ with connected fibers such that $\calG_\bbf\otimes \bbZ\rpot{\varpi}=G\otimes\bbZ\rpot{\varpi}$, and $\calG_{\bbf}\otimes k\pot{\varpi}=\calG_{{\bbf,k}}$ for any field $k$. We define $\calP_\bbf$ as the functor on the category of rings given by $R\mapsto \calG_\bbf(R\pot{\varpi})$. Then $\calP_{\bbf}=\lim_{i\geq 0}\calP_{\bbf,i}$ is a pro-algebraic $\bbZ$-group with connected fibers where
\begin{equation}\label{Weil_resitriction}
\calP_{\bbf,i}\defined \on{Res}_{\bbZ_i/\bbZ}(\calG_{\bbf}\otimes_{\bbZ\pot{\varpi}} \bbZ_i),
\end{equation}
for $\bbZ_i=\bbZ[\varpi]/(\varpi^{i+1})$. Note that each $\calP_{\bbf,i}$ is an algebraic $\bbZ$-group with connected fibers, cf.~the proof of \cite[Lem.~2.11 (ii)]{Richarz:AffGrass}, and hence $\calP_\bbf$ has connected fibers as well, cf.~Lemma \ref{lemm--pro.group} i).

In particular, $\calP_{\bbf}\subset LG$ is a flat closed subscheme such that $\calP_\bbf\otimes k=\calP_{{\bbf,k}}$ for any field $k$. This shows existence. Being a flat closed subscheme, $\calP_{\bbf}$ agrees with the flat closure (=scheme-theoretic image) of $\calP_{\bbf,\bbQ}\subset LG\otimes \bbQ$ inside $LG$. This shows uniqueness, and the lemma follows.
\xpf

\rema \label{rema--bounded_subsets}
More generally, for every subset $\Omega\subset \scrA$ whose projection onto the semisimple part $\scrA_{\on{ss}}$ is bounded (cf.~\cite[\S3.4.1]{Tits:Corvallis}), there exists a smooth affine $\bbZ\pot{\varpi}$-group scheme $\calG_\Omega$ with connected fibers by \cite[\S4.2.2]{PappasZhu:Kottwitz}, resp.~\cite[Lem.~2.1]{HainesRicharz:Normality}. The associated $\bbZ$-subgroup schemes $\calP_\Omega\subset LG$ are strictly pro-algebraic with connected fibers, and satisfy the favorable property $\calP_{\Omega}\cap\calP_{\Omega'}=\calP_{\Omega\cup\Omega'}$, cf.~the proof of Lemma \ref{lemm--orbit.flag} i) below.
\xrema

The general notion of parahoric groups relates to Example \ref{exam--parahoric} as follows. We have $\calP_0=L^+G$ for the base point $0\in \scrA$. The group $\calP_{\bba_0}=\calB$ is the standard Iwahori, and the parahoric subgroups from Example \ref{exam--parahoric} ii) correspond to the finitely many facets $\bbf\subset \scrA$ whose closure contains $0$, and which are itself contained in the closure of $\bba_0$. More generally, we have $\calP_{\bbf}\subset \calP_{\bbf'}$ if and only if $\bbf'$ is contained in the closure of $\bbf$.
The following lemma relates to  Proposition \ref{prop--DM.G.homotopy.invariant}, and we record it for later use.

\lemm\label{lemm--affine.proj}  Let $\bbf\subset \scrA$ be a facet. For each $i\geq 0$, the kernel $\ker(\calP_{\bbf,i+1}\to \calP_{\bbf,i})$ is a vector group of dimension $\dim(G/\bbZ)$. In particular, $\ker(\calP_{\bbf}\to \calP_{\bbf,0})$ is split pro-unipotent in the sense of \refde{unipotent}, and each $\calP_{\bbf,i}$ is a cellular $\bbZ$-scheme in the sense of \refde{cellular}.
\xlemm
\pf
By Proposition \ref{prop--vector.extension} (cf.~also Example \ref{exam--groups} iii.a)), each kernel $U_i:=\ker(\calP_{\bbf,i+1}\to \calP_{\bbf,i})$ is a vector group of dimension $d:=\dim(G/\bbZ)$, and thus $\ker(\calP_{\bbf}\to \calP_{\bbf,0})$ is split pro-unipotent.

In order to show that each $\calP_{\bbf,i}$ is cellular, we consider the projection $\calP_{\bbf,i}\to \calP_{\bbf,0}$. This is isomorphic to a relative affine space: by induction on $i$ it is enough to show that the $U_{i}$-torsor $\calP_{\bbf,i+1}\to\calP_{\bbf,i}$ is trivial which holds by \refpr{unipotent}. Hence, the map is on the underlying schemes isomorphic to $\calP_{\bbf,i}\x_{\bbZ} U_{i}\to \calP_{\bbf,i}$. As every vector bundle on $\Spec \bbZ$ can be trivialized,
we get $U_{i}\simeq \bbA^{d}_\bbZ$.
Hence, it is enough to show that $\calP_{\bbf,0}=\calG_{\bbf}\otimes_{\bbZ\pot{\varpi}}\bbZ$ is cellular. By construction \cite[\S4.2.2 (a)]{PappasZhu:Kottwitz}, the special fiber $\calG_{\bbf}\otimes_{\bbZ\pot{\varpi}}\bbZ$ is also given by base changing the schematic root data defining $\calG_{\bbf}$, and thus admits a semidirect product (Levi) decomposition into a split unipotent $\bbZ$-group scheme and a split reductive $\bbZ$-group scheme. As every split reductive $\bbZ$-group scheme is cellular by the Bruhat decomposition (e.g.~\cite[\S13]{Jantzen:Reps} over $\bbZ$), the lemma follows.
\xpf

\rema\label{rema--general.cellular}
\refle{affine.proj} holds with the same proof for the more general subgroups $\calP_\Omega\subset LG$ from \refre{bounded_subsets}. This is needed in \refle{orbit.flag} below in order to control the stabilizers, e.g., of the $L^+G$-action on the affine Grassmannian.
\xrema

For two facets $\bbf,\bbf'\subset \scrA$ and for any field $k$, the combinatorics of the double coset $\calP_{\bbf'}(k)\backslash LG(k)/\calP_\bbf(k)$ are determined by the double classes $W_{\bbf'}\backslash W/W_\bbf$ in the Iwahori-Weyl group as follows, cf.~\S\ref{sect--loop.group.dfn} v) for notation. We have an identification as abstract groups
\begin{equation}\label{IW_Indentify}
W \,=\, \on{Norm}_G(T)(\bbZ\rpot{\varpi})/T(\bbZ\pot{\varpi}).
\end{equation}
This identification is compatible with the decomposition $W=X_*(T)\rtimes W_0$, and works as follows. We have $X_*(T)=T(\bbZ\rpot{\varpi})/T(\bbZ\pot{\varpi}), \la\mapsto \varpi^{-\la}$ as subgroups of \eqref{IW_Indentify} where we refer to \cite[\S3.2]{dHL:Frobenius} for a discussion of the sign in the identification. Further, as $(G,T)$ is split, the Weyl group scheme $\on{Norm}_G(T)/T=\underline{W}_0$ is constant by \cite[Prop 5.1.6]{Conrad:Groups}. Thus, its $\bbZ\pot{\varpi}$-valued (resp.~$\bbZ\rpot{\varpi}$-valued) points identify with $W_0$ as a subgroup (resp.~quotient) of \eqref{IW_Indentify} which gives the above identification compatible with the semidirect product structure.

Further, the subgroup $W_\bbf\subset W$ associated with a facet $\bbf\subset \scrA$ is identified with
\begin{equation}\label{IW_Sub_Indentify}
 W_\bbf\,=\,\left(\on{Norm}_G(T)(\bbZ\rpot{\varpi})\cap \calP_\bbf(\bbZ)\right)/T(\bbZ\pot{\varpi}).
\end{equation}
For example, if $\bbf$ is an alcove, then $W_\bbf=\{*\}$ is trivial. On the other extreme, if $\bbf=0$ is the base point (or any other special point), then $W_0$ is the finite Weyl group.

\lemm \label{lemm--double.orbit}
Let $\bbf, \bbf'\subset \scrA$ be facets. For any field $k$, there is a bijection of sets
\[
W_{\bbf'}\backslash W/ W_\bbf \,\r\, \calP_{\bbf'}(k)\backslash LG(k)/\calP_\bbf(k), \;\; W_{\bbf'}wW_\bbf\mapsto \calP_{\bbf'}(k)\dot{w}\calP_\bbf(k),
\]
where $\dot{w}\in LG(k)$ denotes the image of a representative of $w\in W$ under the map $LG(\bbZ)\to LG(k)$.
\xlemm
\pf Since $T$ is split, the natural map $X_*(T)=\Hom(\bbG_{m,\bbZ},T)\to \Hom(\bbG_{m,k},T\otimes k)=:X_*(T\otimes k)$ is an isomorphism for any field $k$. The discussion above implies that the natural map
\[
W\to \left(\on{Norm}_G(T)(k\rpot{\varpi})\cap \calP_\bbf(k)\right)/T(k\pot{\varpi})=:W_k
\]
is an isomorphism.
For each facet $\bbf$, we also have the subgroup $W_{{\bbf,k}}\subset W_k$ which under the isomorphism $W=W_k$ is identified with $W_\bbf$.
Thus, we obtain $W_{\bbf'}\backslash W/ W_\bbf=W_{\bbf',k}\backslash W_k/ W_{\bbf,k}$.
Now the lemma follows from the work of Bruhat-Tits, cf.~\cite[App., Prop.~8]{PappasRapoport:LoopGroups} (or \cite[Thm.~1.4]{Richarz:Iwahori}).
\xpf

The following important special cases of these double cosets relate to the partial affine flag varieties  introduced in Definition \ref{defi--flag.variety} below.

\exam \label{exam--double.quot}
i) Let $\bbf'=\bba_0$, i.e., $\calP_{\bbf'}=\calB$ is the standard Iwahori subgroup. The $\calB$-orbits on $\Fl_\bbf$ are parametrized by the infinite set $W/W_\bbf$. In particular, if $\bbf=0$ is the base point, i.e., $\Fl_0=\Gr_G$ is the affine Grassmannian, then $W/W_0=X_*(T)$ is the group of cocharacters.\smallskip\\
ii) Let $\bbf'=\bbf=0$, i.e., $\calP_{\bbf'}=\calP_\bbf=L^+G$. The $L^+G$-orbit on $\Gr_G$ are parametrized by the set $W_0\backslash W/W_0=W_0\backslash X_*(T)=X_*(T)_+$ of dominant cocharacters defined in \eqref{dom_weights}.
\xexam

Let us recall some more structure of the group $W$. The choice of the base alcove $\bba_0$ equips $W$ with the structure of a quasi-Coxeter group with length function $l\co W\to \bbZ_{\geq 0}$ and Bruhat-Chevalley partial order ``$\leq$'' as follows. The group $W$ acts on the apartment $\scrA$ by affine linear transformations permuting transitively all alcoves. According to the choice of base alcove $\bba_0$, we have the finite set of {\it simple affine reflections} $\bbS\subset W$ which are given by the reflections along the walls of $\bba_0$. The subgroup $W_{\aff}\subset W$ generated by $\bbS$ is the affine Weyl group in the sense of \cite[VI, \S2.1]{Bourbaki:Lie456} associated with the based root system $R$. In particular, the pair $(W_\aff,\bbS)$ is a Coxeter group. If $\on{Stab}_{\bba_0}=\{w\in W\,|\, w\cdot \bba_0=\bba_0\}$ viewed as a subgroup of $W$, then there is a semidirect product decomposition
\begin{equation}\label{Quasi_Coxeter}
W\,=\, W_{\aff}\rtimes \on{Stab}_{\bba_0}.
\end{equation}
Hence, every element $w\in W$ admits a decomposition
\begin{equation}\label{decomp_element}
w\,=\, s_1\cdot \ldots\cdot s_q\cdot \tau_w,
\end{equation}
for some $s_1,\ldots,s_q\in \bbS$ and unqiue $\tau_w\in \on{Stab}_{\bba_0}$. The decomposition \eqref{decomp_element} is called {\it reduced} whenever the number $q\in \bbZ_{\geq 0}$ is minimal among all decompositions of $w$. Reduced decompositions are not unique, but $q$ only depends on $w$, and not on the choice of reduced decomposition.

The {\it length of $w\in W$} is defined to be the unique number $l(w):=q\in \bbZ_{\geq 0}$ in some reduced decomposition \eqref{decomp_element}. The partial order on $W$ is defined  by the requirement $v\leq w$ if and only if $\tau_v=\tau_w$, and $v$ arises by deleting some of the $s_i$ in a reduced decomposition of $w$.

Let $\bbf,\bbf'\subset \scrA$ be facets contained in the closure of $\bba_0$. Then the quasi-Coxeter structure of $W$ induces on the double classes
\begin{equation}\label{double_classes}
W_{\bbf'}\backslash W/ W_\bbf
\end{equation}
a length function $l=l(\bbf',\bbf)$ and a partial order $\leq\,=\,{^{\bbf'}\!\!\!\!\leq^\bbf}$, cf.~\cite[Lem.~1.6 ff]{Richarz:Schubert}.


\exam\label{exam--length.function.exam}
If $\bbf'=\bbf=0$, then $W_0\backslash W/W_0=X_*(T)_+$, cf.~Example \ref{exam--double.quot} ii). By e.g.~\cite[Cor.~1.8]{Richarz:Schubert}, the length function is computed as
\begin{equation}\label{length_Grass}
l\co X_*(T)_+\,\r\, \bbZ_{\geq 0}, \;\; \mu\mapsto \lan2\rho, \mu\ran,
\end{equation}
where $2\rho:=\sum_{a\in R_+}\! a \in X^*(T)$ is the sum of the positive roots. The partial order on $X_*(T)_+$ specializes to the dominance order described in \eqref{dom_weights}. Similarly, if $\bbf'=\bba_0$, and $\bbf=0$, then $W/W_0=X_*(T)$, and the length function is computed as
\begin{equation}\label{length_Grass_Iwahori}
l\co X_*(T)\,\r\, \bbZ_{\geq 0}, \;\; \mu\mapsto \lan2\rho, \mu^{\on{dom}}\ran-\#\{a\in R^+\;|\; \lan a,\mu\ran <0\},
\end{equation}
where $\mu^{\text{dom}}$ is the unique dominant representative in $W_0\cdot \mu$.
\xexam

\subsection{Stratifications on affine flag varieties}
\label{sect--Stratifications.flag}

Fix two facets $\bbf,\bbf'\subset \scrA$, and denote by $\calP:=\calP_\bbf, \calP':=\calP_{\bbf'}$ the associated parahoric subgroups. 

\defi\label{defi--flag.variety}
The \emph{\textup{(}partial\textup{)} affine flag variety $\Fl=\Fl_{\bbf}$ associated with $\bbf$} is the \'etale sheaf quotient
\[
\Fl\defined (LG/\calP)^\et.
\]
\xdefi

Let us note that the quotient map $LG\to \Fl$ admits sections Zariski locally which follows from \cite[Def.~5 ff.]{Faltings:Loops}. Thus $\Fl(R)=LG(R)/\calP(R)$ for any local ring $R$, and $\Fl$ agrees with the Zariski sheafification of the functor $R\longmapsto LG(R)/\calP(R)$. Here we use that the group $G$ is split.

By \cite[Lem.~2.1 ff]{HainesRicharz:Normality}, the sheaf $\Fl$ is representable by an ind-projective ind-scheme over $\Z$. If $\bbf=\bba_0$ is the base alcove, then $\calP=\calB$ is the standard Iwahori subgroup, and $\Fl$ is {\it the full affine flag variety}. If $\bbf=0$ is the base point, then $\calP_0=L^+G$, and $\Fl=\Gr_G$ is {\it the affine Grassmannian}.
Further, the affine flag variety is equipped with a transitive left action of the loop group
\begin{equation}\label{flag_act}
LG\times \Fl\to \Fl, \;\;\;(g,x)\mapsto g\cdot x.
\end{equation}
The restriction of the $LG$-action to the parahoric subgroup $\calP'$ is well-behaved in the following sense.

\begin{Lemm} \label{lemm--ind.pres}
There exists a $\calP'$-stable presentation $\Fl=\on{colim}_i\Fl_{i}$ where $\Fl_{i}$ are projective $\bbZ$-schemes, and the $\calP'$-action on each $\Fl_{i}$ factors through some $\calP'_{j}$ for $j>\!\!>0$.
\end{Lemm}
\begin{proof} The lemma immediately follows from Lemma \ref{lemm--adm.finite.type} because $\Fl$ is ind-projective over $\bbZ$, and in particular of ind-finite type.
\end{proof}


\defi\label{defi--Schubert}
For $w\in W_{\bbf'}\backslash W/W_\bbf$, the {\it Schubert scheme} $\Fl^{\leq w}={^{\bbf'}\!\!\Fl}_\bbf^{\leq w}$ is the scheme-theoretic image of the map
\begin{equation}\label{Schubert_Map}
\calP'\to \Fl,\;\; p'\mapsto p'\cdot \dot{w}\cdot e,
\end{equation}
where $e\in \Fl(\bbZ)$ is the base point, and $\dot{w}\in LG(\bbZ)$ is a representative of $w$.
\xdefi

Let us justify the definition. It is clear from \eqref{IW_Sub_Indentify} that $\Fl^{\leq w}$ is independent of the choice of $\dot{w}$. Write $\Fl=\on{colim}_i\Fl_{i}$ as in Lemma \ref{lemm--ind.pres}. Then there exists an $i>\!\!>0$ with $\dot{w}\in \Fl_{i}(\bbZ)$, and hence the map \eqref{Schubert_Map} factors through $\Fl_{i}$ defining a map $\calP'\to \Fl_{i}$ of quasi-compact separated schemes (hence a quasi-compact map). By \StP{01R8}, its scheme-theoretic image $\Fl^{\leq w}$ is the closed subscheme of $\Fl$ defined by the quasi-coherent ideal sheaf $\on{ker}(\calO_{\Fl_i}\to \calO_{\calP'})$, and in particular, a projective $\bbZ$-scheme.

Usually, Schubert schemes (resp.~Schubert varieties) in partial affine flag varieties are defined over fields as reduced orbit closures \cite[Def.~8.3]{PappasRapoport:LoopGroups}.
In the case of fields, the definition of Schubert varieties as reduced orbit closures agrees with the definition via scheme theoretic images as in \refde{Schubert}.
However, it is the latter notion which behaves well over more general base schemes, cf.~\S\ref{sect--chevalley.base.change} below.
Here is the relation between Schubert varieties over fields with Schubert schemes over the integers.

\lemm\label{lemm--base.change.schubert.field}
For any field $k$, the underlying reduced locus of the base change $\Fl^{\leq w}\otimes_\bbZ k$ is the Schubert variety over $k$ associated with the element $w\in W_{\bbf'}\backslash W/W_{\bbf}$.
\xlemm
\pf
By \refle{double.orbit}, we obtain for each class $w$ a unique Schubert variety $\Fl_k^{\leq w}$ over $k$.
We claim that the natural closed immersion $\Fl_k^{\leq w}\subset \Fl^{\leq w}\otimes_\bbZ k$ is an isomorphism on reduced loci.
We need to show that this map is an equality on the underlying topological spaces.
For this fix a reduced expression of $w$ as in \eqref{decomp_element}, and consider the Demazure resolution $D(w)\to  \Fl^{\leq w}$ over the integers \cite[Def.~5 ff.]{Faltings:Loops}.
Since $D(w)\to  \Fl^{\leq w}$ is surjective, the base change $D(w)\otimes_\bbZ k\to  \Fl^{\leq w}\otimes_\bbZ k$ is surjective as well \StP{01S1}.
Since $D(w)\to \Spec(\bbZ)$ is smooth, its formation commutes with base change so that $D(w)\otimes_\bbZ k\to \Fl_k^{\leq w}$ is the Demazure resolution over $k$.
Thus, $\Fl_k^{\leq w}\subset \Fl^{\leq w}\otimes_\bbZ k$ is surjective, i.e., an equality on the underlying topological spaces.
\xpf

The following lemma is basic for the study of $\calP'$-orbits in $\Fl^{\leq w}$.

\lemm \label{lemm--orbit.flag}
Let $w\in W_{\bbf'}\backslash W/W_\bbf$.\smallskip\\
i\textup{)} The stabilizer of $\calP'$ in $\dot{w}\cdot e\in \Fl(\bbZ)$ is representable by a closed $\bbZ$-subgroup $\calP_w'\subset \calP'$. It is strictly pro-algebraic $\calP_w'=\lim_i\calP_{w,i}'$, and the scheme underlying each $\calP_{w,i}'$ is fibrewise connected and cellular.\smallskip\\
ii\textup{)} The \'etale sheaf-theoretic image
\[
{\Fl}^w:=\calP'\cdot \dot{w}\cdot e\,\subset\, \Fl^{\leq w},
\]
is representable by an open subscheme which is smooth, fibrewise geometrically connected and dense over $\Z$. It agrees with the \'etale quotients $\Fl^w= (\calP'/\calP'_w)^\et= (\calP'_i/\bar{\calP}'_{w,i})^\et$ for $i>\!\!>0$ where $\bar{\calP}'_{w,i}\subset \calP'_i$ is a fibrewise connected smooth cellular closed $\bbZ$-subgroup scheme.\smallskip\\
iii\textup{)} For each $v\leq w$, there is a quasi-compact immersion $\Fl^v\to \Fl^{\leq w}$, and one has as sets
\[
\Fl^{\leq w}\,=\,\bigsqcup_{v\leq w}\Fl^v,
\]
where ``$\leq$'' denotes the partial order on $W_{\bbf'}\backslash W/W_{\bbf}$.
\xlemm
\begin{proof} For i), the stabilizer is given by the closed $\bbZ$-subgroup $\calP_{\bbf'}\cap \calP_{w\bbf}\subset LG$ where we used that $\calP_{w\bbf}=\dot{w}\calP_{\bbf}\dot{w}^{-1}$. As in Remark \ref{rema--bounded_subsets}, we have for any subset $\Omega\subset \scrA$ whose projection onto the semisimple part $\scrA_{\on{ss}}$ is bounded, an algebraic $\bbZ\pot{\varpi}$-group scheme $\calG_\Omega$ with connected fibers. We apply this to $\Omega=\bbf'\cup w\bbf$, and we claim that one has $\calP_{\bbf'}\cap \calP_{w\bbf}\,=\, \calP_{\bbf'\cup w\bbf}$ as closed subgroups of $LG$. This implies i) because $\calP_\Omega$ is the strictly pro-algebraic $\bbZ$-group scheme given by the functor $R\mapsto \calG_\Omega(R\pot{\varpi})$, and satisfies the asserted properties by \refle{affine.proj} and \refre{general.cellular}. It remains to prove the claim. Recall that, infact for any $\Omega$, the group scheme $\calG_\Omega$ is constructed from the rational group law on the big open cell
\[
\calG_\Omega^o:=\calU_\Omega^-\x\calT\x \calU_\Omega^+\,\to\, \calG_\Omega, \;\;\; (u^-,t,u^+)\mapsto u^-\cdot t\cdot u^+,
\]
where $\calT:=T\otimes_\bbZ\bbZ\pot{\varpi}$, and $\calU_\Omega^\pm$ are split unipotent algebraic $\bbZ\pot{\varpi}$-group schemes, cf.~\cite[\S4.2.2]{PappasZhu:Kottwitz} applied with $\bbZ\pot{\varpi}$ as a base ring. Here $\calG_\Omega^o\subset \calG_\Omega$ is an open subscheme which is fibrewise dense over $\bbZ\pot{\varpi}$ and carries a rational group law compatible with the group law on $\calG_\Omega$. Denote by $\calP_\Omega^o$ the functor $R\mapsto \calG_\Omega^o(R\pot{\varpi})$. We claim that $\calP_\Omega^o\subset \calP_\Omega$ is relatively representable by an open immersion which is fibrewise dense over $\bbZ$. Indeed, the same construction as in \eqref{Weil_resitriction} applies so that $\calP_\Omega^o=\lim_{i\geq 0}\calP_{\Omega,i}^o$ is an inverse limit of smooth affine $\bbZ$-schemes. Note that each $\calP_{\Omega,i}^o\subset \calP_{\Omega,i}$ is an open subscheme which is fibrewise dense over $\bbZ$. In passing to the limit, we need to make sure that the open subsets are not getting too small. Since $\calP_{\Omega,i}^o$ contains the unit section, the construction in the proof of Proposition \ref{prop--vector.extension} below applies to show that $\ker(\calP_{\Omega,i}^o\to \calP_{\Omega,i-1}^o)=\ker(\calP_{\Omega,i}\to \calP_{\Omega,i-1})$ for all $i\geq 1$ (the left hand side are by definition the sections which map to the unit section). Here we used that under the open immersion $\calG_{\Omega}^o\subset \calG_\Omega$ the associated vector bundles \eqref{vector_bundle} agree. It follows that $\ker(\calP_{\Omega}^o\to \calP_{\Omega,0}^o)=\ker(\calP_{\Omega}\to \calP_{\Omega,0})$ which implies the claim. Now from the construction of $\calG_\Omega^o$ it is immediate that $\calP_{\bbf'}^o\cap \calP_{w\bbf}^o=\calP_{\bbf'\cup w\bbf}^o$. Thus, the natural closed immersion $\calP_{\bbf'\cup w\bbf}\subset \calP_{\bbf'}\cap \calP_{w\bbf}$ is an equality over the open and fibrewise dense subscheme $\calP_{\bbf'\cup w\bbf}^o$. This immediately implies $\calP_{\bbf'\cup w\bbf}=\calP_{\bbf'}\cap \calP_{w\bbf}$, and shows i).

Part ii) is proven the same way as \cite[Cor.~3.14]{Richarz:AffGrass}. In the reference, the base is a discrete valuation ring, but the same argument works over the Dedekind ring $\bbZ$ as well. Also we note that for $i>\!\!>0$, the split pro-unipotent kernel $U_i=\ker(\calP'\r\calP_i')$ lies in $\calP'_w$, so that $\bar{\calP}'_{w,i}:=(\calP'_w/U_i)^\et\subset \calP_i'$ is fibrewise connected, cellular and smooth over $\bbZ$.

Part iii) is deduced the same way as for example in \cite[Prop.~2.8 (i)]{Richarz:Schubert} (which is over fields) using the existence of Demazure resolutions.
\end{proof}

Now assume that $\bbf,\bbf'\subset \scrA$ are contained in the closure of the base alcove $\bba_0$, and consider the double classes $W_{\bbf'}\backslash W/W_\bbf$ equipped with its length function $l$ and partial order $\leq$ as in \eqref{double_classes} above.
By Lemma \ref{lemm--orbit.flag}, there is a presentation of the underlying reduced ind-scheme
\begin{equation}\label{Schubert_praesi}
\Fl_\red\,=\,\on{colim}_w\Fl^{\leq w},
\end{equation}
where $w$ runs over the partial ordered set $W_{\bbf'}\backslash W/W_\bbf$. The following proposition equips $\Fl$ with a cellular stratification.

\prop\label{prop--cells.flag}
Assume $\bbf'=\bba_0$, and let $w\in W/W_\bbf$. \smallskip\\
i\textup{)} There is an isomorphism of schemes $\Fl^w\simeq \bbA_\bbZ^{l(w)}$ where $l(w)\in \bbZ_{\geq 0}$ is the length function on $W/W_\bbf$.\smallskip\\
ii\textup{)} The quotient map $LG\to \Fl$ has sections over $\Fl^w$.
 \xprop
\pf
For any affine root $\al=a+k\in \calR$, there is the root homomorphism $u_\al\co \bbG_a\to LG, x\mapsto u_a(x z^k)$, see e.g.~\cite[\S3.5]{dHL:Frobenius}. The reference is written over a field, but the same formulas work over $\bbZ$ as well, cf.~also \cite[\S5.1]{Conrad:Groups} for root subgroups in the relative set-up. Then $u_\al$ is a closed immersion, and we let $U_\al\subset LG$ be its image which is a closed $\bbZ$-subgroup scheme of $LG$ isomorphic to the additive group $\bbG_{a}=\bbA^1_\bbZ$. We consider the map
\begin{equation}\label{Root_Groups}
\pi\co \bigsqcap_{\al} U_\al \to \Fl^w,\;\;\;\; (u_\al)_\al\mapsto \left(\sqcap_\al u_\al\right)\cdot \dot{w}\cdot e,
\end{equation}
where the product (taken in any fixed order) ranges over all affine roots $\al\in \calR$ such that $(w\al)|_{\bba_0}$ takes positive values and $\al|_\bbf$ takes negative values. We claim that $\pi$ is an isomorphism of $\bbZ$-schemes. Indeed, as source and target are smooth $\bbZ$-schemes of finite type, the fibral isomorphism criterion from \cite[I.5, Prop.~5.7]{SGA1} reduces us to prove that $\pi\otimes k$ is an isomorphism for any (prime) field $k$. This is well-known, cf.~e.g.~\cite[Prop.~3.7.4, (3.32)]{dHL:Frobenius}. Also note that the number of such roots as above is $l(w)$ for the length taken on $W/W_\bbf$. This implies i). Part ii) also follows because  the sections are given by $\pi^{-1}$ composed with the closed immersion $\sqcap_\al U_\al\to LG$, $(u_\al)_\al\mapsto (\sqcap_\al u_\al)\cdot \dot{w}$. This concludes the proof of the proposition.
\xpf

\rema Proposition \ref{prop--cells.flag} ii) fails if $\bbf'$ is not an alcove, i.e., $\calP_{\bbf'}$ strictly contains the standard Iwahori $\calB$. Indeed, if the map $LG\to \Fl$ has a section over $\Fl^w$, then $\Fl^w$ must necessarily be affine (because $LG$ is ind-affine). However, whenever $\bbf'$ is not an alcove, there exists $w\in W$ such that $\Fl^w$ is not affine. As an example consider $G=\GL_2$, and take $\bbf'=\bbf=0$. Then, for $\mu=(1,0)\in X_*(T)_+$, we have $\Gr_G^{\leq \mu}=\P_\bbZ$.
\xrema

\coro
\label{coro--Fl.stratification}
Let $\Fl^+:=\sqcup_{w\in W_{\bbf'}\setminus W/W_\bbf}\Fl^w$. Then the inclusion $\iota\co \Fl^+\to \Fl$ is a cellular stratified ind-scheme in the sense of Definition \ref{defi--cellular}. For $\bbf' = \bba_0$, we refer to this stratification on $\Fl$ as the {Iwahori stratification}.
\xcoro
\pf

Each $\calP'$-orbit $\Fl^w$ is a smooth $\bbZ$-scheme with geometrically connected fibers by Lemma \ref{lemm--orbit.flag} ii), and it decomposes in Iwahori orbits as
\[
\Fl^w\,=\, \bigsqcup_{v\in W_{\bbf'} wW_\bbf/W_\bbf}\Fl^v,
\]
where each $\Fl^v$ is isomorphic to an affine space, cf.~Proposition \ref{prop--cells.flag} i). Thus each $\Fl^w$ is a cellular $\bbZ$-scheme. Further, by \eqref{Schubert_praesi} together with Lemma \ref{lemm--orbit.flag} iii), the map $\iota$ is bijective on the underlying topological spaces.
As each restriction $\iota|_{\Fl^w}$ is the composition $\Fl^w\subset \Fl^{\leq w}\subset \Fl$ of a quasi-compact open immersion followed by a closed immersion, it is a quasi-compact immersion.
Also we have for the closure $\overline{\Fl}^w=\Fl^{\leq w}$ by Lemma \ref{lemm--orbit.flag} ii). This implies the corollary.
\xpf

Now assume further that $\calB\subset \calP'\subset\calP$, i.e., $\bbf$ is contained in the closure of $\bbf'$. We end this subsection by investigating the behavior of the Iwahori stratification under the projection $\pi\co \Fl_{\bbf'}\to \Fl_{\bbf}$. The following proposition relates to  \refle{Tate.proper.descent} and \refle{smooth.detects.Tate}.

\prop \label{prop--change.facet}
i\textup{)} The projection $\pi\co \Fl_{\bbf'}\to \Fl_{\bbf}$ is representable by a smooth proper surjective map which is \'etale locally on the target isomorphic to the projection $\calP/\calP'\x \Fl_{\bbf}\to \Fl_\bbf$.\smallskip\\
ii\textup{)} The induced map on the Iwahori stratifications $\pi^+\co \Fl_{\bbf'}^+\to \Fl_{\bbf}^+$ is a Tate map, and admits a section $s^+\co \Fl_{\bbf}^+\to \Fl_{\bbf'}^+$ which is an open and closed immersion.
\xprop
\pf For i) we refer to the proof of \cite[Lem 4.9 i)]{HainesRicharz:TestFunctions} for details. For ii), first let $\bbf'=\bba_0$ be the base alcove, and abbreviate $\Fl=\Fl_{\bbf'}$. For $w\in W/W_{\bbf}$, we have
\[
(\pi^+)^{-1}(\Fl_{\bbf}^w)\,=\, \bigsqcup_{v\in wW_{\bbf}} \Fl^v
\]
There exists a unique element $w_{\min}\in wW_{\bbf}$ of minimal length, cf.~e.g.~\cite[Lem 1.6 (i)]{Richarz:Schubert}. Thus, every $v\in wW_{\bbf}$ can be written uniquely in the form $v=v_0\cdot w_{\min}$. It follows from \eqref{Root_Groups} that the restriction $\pi^+|_{\Fl^v}\co \Fl^v\to \Fl_{\bbf}^w$ has the structure of a relative affine space of relative dimension $l(v_0)$. In particular, $\pi^+$ is a Tate map, and $\pi^+|_{\Fl^{w_{\min}}}$ is an isomorphism. The desired section is given by
\[
s^+\co \Fl_{\bbf}^+=\bigsqcup_{w\in W/W_{\bbf}}\Fl_{\bbf}^w\;\overset{\simeq}{\longleftarrow}\; \bigsqcup_{w\in W/W_{\bbf}}\Fl^{w_{\min}}\subset \bigsqcup_{w\in W}\Fl^w=\Fl^+.
\]
The case of more general facets $\bbf'$ is reduced to this case by considering the projections $\Fl_{\bba_0}\to \Fl_{\bbf'}\to \Fl_{\bbf}$. This implies ii), and proves the proposition.
\xpf

\exam \label{exam--simple.reflection}
Let $s\in \bbS$ be a simple affine reflection. Then there is a unique facet $\bbf_s$ of maximal dimension in the closure of $\bba_0$ such that $s(\bbf_s)=\bbf_s$, i.e., $W_{\bbf_s}$ is the subgroup generated by $s$. We specialize Proposition \ref{prop--change.facet} to the case $\pi\co \Fl\to \Fl_{\bbf_s}$ so that $\calB=\calP\subset \calP'=\calP_{\bbf_s}$. In this case, the map $\pi$ has general fiber $(\calP_{\bbf_s}/\calB)^\et=\P_\bbZ$. If $w=v\cdot s$ is a reduced decomposition so that $v=w_{\min}$, then
\[
(\pi^+)^{-1}(\Fl_{\bbf_s}^v)\,=\,\Fl^{v}\sqcup\Fl^{vs}.
\]
Here $\pi^+|_{\Fl^v}$ is an isomorphism, and $\pi^+|_{\Fl^{vs}}$ is an affine space of relative dimension $1$.
\xexam

\subsection{Changing the base scheme}\label{sect--chevalley.base.change}

Let $S$ be any non-empty scheme. We change notation, and let $G$ be a split reductive $S$-group scheme, i.e., a smooth $S$-affine $S$-group whose fibers are connected reductive groups, and which admits a maximal split torus, cf.~\cite[Def 5.1.1]{Conrad:Groups} for a precise definition. Recall that by  the Isomorphism Theorem \cite[Thm 6.1.17]{Conrad:Groups} the group $G$ is already defined over $\bbZ$, i.e., there exists a Chevalley group $G_\bbZ$ such that $G=G_\bbZ\x_{\Spec{\bbZ}} S$ (we fix the isomorphism). We also fix $T_\bbZ\subset B_\bbZ\subset G_\bbZ$ as in \eqref{Chevalley_Triple}, and let $T\subset B\subset G$ be the base change to $S$.
We denote by $\scrA$ (resp. $W$) the apartment (resp. Iwahori-Weyl group) associated with $(G_\bbZ,T_\bbZ)$.
The definitions and constructions from \S\ref{sect--loop.definitions} and \S\ref{sect--Stratifications.flag} generalize to general base schemes as follows.

The loop group $LG$ is the functor given by $LG(R)=G(R\rpot{\varpi})$ for $\Spec(R)\in\AffSch_S$. Clearly, we have $LG=LG_\bbZ\x_{\Spec(\bbZ)}S$.
Likewise, for any facet $\bbf\subset \scrA$ the parahoric subgroup $\calP_\bbf\subset LG$ is defined by base change from $\bbZ$. In particular, $\calP_\bbf=\lim_{i\geq 0}\calP_{\bbf,i}$ is a strictly pro-algebraic $S$-group with geometrically connected fibers.
The partial affine flag variety $\Fl_\bbf$ is the \'etale sheaf associated with the functor $R\mapsto LG(R)/\calP_\bbf(R)$ for $\Spec(R)\in \AffSch_S$.
Since sheafification commutes with base change, we see that $\Fl_\bbf=\Fl_{\bbf,\bbZ}\x_{\Spec(\bbZ)}S\to S$ is the base change from $\bbZ$ as well, and in particular ind-projective.

Let $\bbf,\bbf'\subset \scrA$ be two facets, and abbreviate $\calP':=\calP_{\bbf'}$ and $\Fl:=\Fl_\bbf$.
For any $S$-scheme $T$, we write $LG(T):=\Hom_S(T,LG)$, and likewise for $\Fl$.

\defi\label{defi--Schubert.scheme}
For $w\in W_{\bbf'}\backslash W/W_\bbf$, the {\it Schubert scheme $\Fl^{\leq w}={^{\bbf'}\!\!\Fl}_\bbf^{\leq w}$ over $S$} is the scheme-theoretic image of the map
\begin{equation}\label{Schubert_Map_Rel}
\calP'\to \Fl,\;\; p'\mapsto p'\cdot \dot{w}\cdot e,
\end{equation}
where $e\in \Fl(S)$ is the base point, and $\dot{w}\in LG(S)$ is the image of a representative of $w$ under the map $LG(\bbZ)\to LG(S)$.
\xdefi

As in Definition \ref{defi--Schubert} one sees that $\Fl^{\leq w}\subset \Fl$ defines a closed subscheme whose underlying topological space coincides with the closure of the topological image of \eqref{Schubert_Map_Rel}.

\prop
\label{prop--Schubert.base.change}
Let $\Fl_\bbZ^{\leq w}\subset \Fl_\bbZ$ be the Schubert scheme over $\bbZ$.
Then the natural closed immersion $\Fl^{\leq w}\subset \Fl_\bbZ^{\leq w}\x_{\Spec(\bbZ)}S$ is a Nil thickening.
\xprop
\pf
We first note that the proposition is obvious whenever $S\to \Spec(\bbZ)$ is flat because scheme-theoretic images along quasi-compact maps commute with flat base change, cf.~\StP{01R8}.
If $S\to \Spec(\bbZ)$ is not necessarily flat, we have to show that $\Fl^{\leq w}\subset \Fl_0^{\leq w}\x_{\Spec(\bbZ)}S$ is an equality on topological spaces.
By functoriality of the scheme theoretic image \StP{01R9}, for every field $\Spec(k)\to S$ we have closed immersions $\Fl_k^{\leq w}\subset \Fl^{\leq w}\x_{\Spec(S)}\Spec(k)\subset \Fl_\bbZ^{\leq w}\x_{\Spec(\bbZ)}\Spec(k)$ where $\Fl_k^{\leq w}$ denotes the Schubert variety over $k$.
Hence, \refle{base.change.schubert.field} implies the claim.
\xpf

The canonical closed immersion
\begin{equation}\label{Schubert_praesi_rel}
\on{colim}_w\Fl^{\leq w}\,\hookrightarrow\, \Fl,
\end{equation}
is a Nil thickening, and hence an isomorphism on the underlying reduced loci. In particular, \eqref{Schubert_praesi_rel} induces an equivalence on the categories of motives. For completeness, we remark that the Schubert scheme $\Fl^{\leq w}$ is non-reduced if $S$ is non-reduced so that we need to pass to the underlying reduced loci on both sides in \eqref{Schubert_praesi_rel} to get an isomorphism.

For $w\in W_{\bbf'}\backslash W/W_\bbf$, we define $\Fl^w$ as the \'etale sheaf image of \eqref{Schubert_Map_Rel}. Since sheaf theoretic images commute with base change, we see that $\Fl^w$ is the base change from $\bbZ$. In particular, $\Fl^w\subset \Fl^{\leq w}$ is an open subscheme which is smooth, fibrewise dense and geometrically connected over $S$. Now all results of Lemma \ref{lemm--orbit.flag}, Proposition \ref{prop--cells.flag}, Corollary \ref{coro--Fl.stratification}, Proposition \ref{prop--change.facet} and Example \ref{exam--simple.reflection} translate literally to the general context by base change $S\to \Spec(\bbZ)$.

\section{Mixed Tate motives on affine flag varieties}\label{sect--DTM.Fl}


\nota
\label{nota--BS.vanishing.ladic}
Throughout \refsect{DTM.Fl}, we assume $S$ is as in \refno{S} and satisfies furthermore the Beilinson-Soulé vanishing conjecture as in \textup{\refeq{BS.vanishing}}.
We also assume $S$ admits an $\ell$-adic realization functor in the sense of \refre{realization.functor}.
Examples include finite fields $\Fq$, function fields $\Fq(t)$, number fields $F$, and their algebraic \textup{(}separable/perfect\textup{)} closures.
Further examples are the ring of algebraic integers $\calO_F$, and smooth curves over finite fields.
\xnota

W fix $T\subset B\subset G$ over $S$ as in \S\ref{sect--chevalley.base.change}. We let $\bbf,\bbf'\subset \scrA$ be two facets which are contained in the closure of the base alcove $\bba_0$.
Their associated parahoric groups are denoted $\calP_{\bbf}, \calP_{\bbf'} \subset LG$.
In the following, we use \S\ref{sect--chevalley.base.change} without explicit reference in order to apply the results from \S\S\ref{sect--loop.definitions}-\ref{sect--Stratifications.flag}.
Throughout, $w$ denotes an element in $W_{\bbf'} \backslash W / W_\bbf$; thus, $w$ parametrizes the orbits of the left action of the pro-algebraic $S$-group $\calP_{\bbf'}$ on the ind-scheme $\Fl_\bbf$.
The inclusion of such an orbit is denoted $\iota_w : \Fl_\bbf^w \r \Fl_\bbf$. By \refco{Fl.stratification}, these orbits yield a cellular stratification which is denoted
$$\iota: \bigsqcup_{w \in W_{\bbf'} \setminus W / W_\bbf} \Fl_\bbf^w \r \Fl_\bbf.\eqlabel{iota.Fl}$$

In this section, we apply the results of \refsect{DTM} to partial affine flag varieties and obtain a category of Tate motives on the double quotient $\calP_{\bbf'} \backslash LG / \calP_\bbf$.
Also, if $\bbf'=\bbf$, then we obtain an abelian subcategory of mixed motives on $\calP_{\bbf} \backslash LG / \calP_\bbf$.
The latter category contains the intersection motives $\IC_w \in \DM(\calP_{\bbf} \backslash LG / \calP_\bbf)$ that will be used in \refsect{intersection} to construct the intersection motive of the moduli stack of $G$-shtukas.

\subsection{Whitney-Tate stratifications on partial affine flag varieties}\label{sect--WT.Fl}

\theo
\label{theo--Fl.WT}
The stratification \textup{\refeq{iota.Fl}} is a Whitney-Tate stratification.
\xtheo

In \cite[Prop.~4.10]{SoergelWendt:Perverse}, Soergel and Wendt prove the analogous statement for the Borel orbit stratification in partial flag varieties over fields. Thus, \refth{Fl.WT} generalizes their result in three ways. We work with infinite-dimensional partial affine flag varieties in which the finite-dimensional partial flag varieties can be embedded compatibly with the Iwahori respectively Borel stratifications. We allow the stratification into $\calP_{\bbf'}$-orbits instead of merely Iwahori orbits. We work over more general base schemes; for example $S=\Spec(\bbZ)$ is allowed.
The last feature will be used in \cite{RicharzScholbach:Satake} to transfer the purity of the intersection motives $\IC_w$ from the case that $S$ has positive characteristic to the case that $S$ has characteristic zero.
The proof of \refth{Fl.WT} proceeds in three steps; the first two, i.e., the case of the Iwahori stratification, are an extension of the arguments \cite[Prop.~4.10]{SoergelWendt:Perverse} to the affine flag variety.
We also point out that Habibi \cite[Cor.~5.4.12]{Habibi:Motive} has shown that the motive of affine Schubert varieties $\Fl^{\le w}_\bbf$ is a Tate motive provided that the Demazure resolution is semi-small.

\begin{proof}[Proof of \refth{Fl.WT}]
{\em First step: $\bbf'=\bbf=\bba_0$ is the base alcove.}
Write $\Fl := \Fl_{\bba_0}$, and $\calB:=\calP_{\bba_0}$ for the Iwahori group.
We have to show $\iota^! \iota_! 1 \in \DTM(\Fl^+)$, i.e., the restriction to each stratum $\Fl^w$ is a Tate motive. By induction on the length $l(w)$, we show that
$$\iota^! (\iota_w)_! 1\,\in\, \DTM(\Fl^+).$$
If $l(w)=0$, then necessarily $w\in \on{Stab}_{\bba_0}$ by \eqref{Quasi_Coxeter} so that $\iota_w$ is a closed immersion.

If $l(w)>0$, there is a decomposition $w=vs$ for some simple reflection $s\in\bbS$ such that $l(v)=l(w)-1$. As in Example \ref{exam--simple.reflection} we denote by $\calB\subset \calP_{\bbf_s} \subset LG$ the parahoric subgroup associated with $s\in \bbS$, and by $\Fl_{\bbf_s}$ the corresponding partial affine flag variety. The projection $\pi\co \Fl\to \Fl_{\bbf_s}$ is smooth and proper with fibre \'etale locally isomorphic to $(\calP_{\bbf_s}/\calB)^\et=\P_S$, and the induced map on the Iwahori stratifications $\pi^+\co \Fl^+\to \Fl_{\bbf_s}^+$ is a Tate map which admits a section, cf.~\refpr{change.facet} and \refex{simple.reflection}. We consider the commutative diagram
$$\xymatrix{
\Fl^v \sqcup \Fl^{w} \ar[r] \ar[dr]_{\pi^+} & \pi^{-1}(\Fl_{{\bbf_s}}^v) \ar[d] \ar[r] & \Fl \ar[d]^{\pi} \\
& \Fl_{{\bbf_s}}^v \ar[r] & \Fl_{{\bbf_s}}.
}$$
Then $\pi^+|_{\Fl^v}$ is an isomorphism, and $\pi^+|_{\Fl^{w}}$ is a relative $1$-dimensional affine space. The localization sequence for $\Fl^v \r \pi^{-1}(\Fl_{{\bbf_s}}^v) \gets \Fl^{w}$ therefore gives a fiber sequence
$$(\iota_w)_! 1 \r \pi^* \pi_{*} (\iota_v)_! 1 \r (\iota_v)_! 1.\eqlabel{iota.v.w}$$
If we apply $\iota^!$, the right hand term lies in $\DTM(\Fl^+)$ by induction, the middle term therefore also, by \refle{Tate.up.down} and the smoothness of $\pi$ (which allows to replace $\pi^*$ by $\pi^!$).
Hence the left hand term is also a Tate motive.\smallskip\\
{\em Second step: $\bbf'=\bba_0$ is the base alcove, and $\bbf$ arbitrary \textup{(}but contained in the closure of $\bba_0$\textup{)}.}
We consider the canonical projection $\pi\co \Fl_{\bba_0} \r \Fl_{\bbf}$ where both affine flag varieties are equipped with the Iwahori stratification.
By \refpr{change.facet} and the first step, we may apply \refle{Tate.proper.descent} to $\pi$ and conclude.\smallskip\\
{\em Third step: $\bbf'$, $\bbf$ are arbitrary.}
We have to show that
\begin{equation}\label{Tate.big.orbit}
M := (\iota_w)^! (\iota_v)_! 1\,\in\, \DTM(\Fl_\bbf^w),
\end{equation}
for each  $v, w \in W_{\bbf'} \backslash W / W_\bbf$, where $\iota_w=\iota|_{\Fl_\bbf^w}$ (resp. $\iota_v=\iota|_{\Fl_\bbf^v}$). Note that $M$ is $\calP_{\bbf'}$-equivariant because $\iota_w$ and $\iota_v$ are $\calP_{\bbf'}$-equivariant. By Step 2) we know that $!$-restricting further to the Iwahori orbits gives Tate motives, and we use the equivariance to prove \eqref{Tate.big.orbit} as follows:
By Lemma \ref{lemm--orbit.flag} i) and ii), the map $\calP':=\calP_{\bbf'}\to \Fl_\bbf^w$, $p'\mapsto p'\cdot \dot{w}\cdot e$ induces an isomorphism $\Fl_\bbf^w=\calP'/\calP'_{w}=\calP'_i/\bar{\calP}'_{w,i}$ for $i>\!\!>0$ where both $\calP'_i$ and $\bar{\calP}'_{w,i}$ are fibrewise connected and cellular.
Let $e_w\co S \r \Fl_\bbf^w$ be the inclusion of the base point.
By \refpr{DTM.G/H}, $M$ is Tate iff $(e_w)^! M$ is a Tate motive on $S$. But this holds true since $M$ is a Tate motive with respect to the Iwahori stratification.
\xpf

\subsection{Tate motives on partial affine flag varieties}\label{sect--Tate.Fl}
Given the cellular Whitney-Tate stratification of $\Fl_\bbf$ by $\calP_{\bbf'}$-orbits we can apply \refdele{Whitney.Tate.condition} to get a well-defined subcategory of stratified Tate motives
\begin{equation}\label{DTM.flag}
\DTM(\Fl_\bbf)\,\subset\, \DM(\Fl_\bbf).
\end{equation}
It is the subcategory generated (by arbitrary shifts and colimits) by the objects $(\iota_w)_! 1_{\Fl^w_\bbf}(n)$ for $n \in \Z$, $w \in W_{\bbf'} \setminus W / W_\bbf$.
This category admits the following characterization, similarly to \cite[Lem.~3.2.1]{Soergel:ICandRT}:

\prop\label{prop--DTM.Fl.characterization}
We equip all categories with the Iwahori stratification.\smallskip\\
i\textup{)} The category $\DTM(\Fl_{\bba_0})$ is the smallest cocomplete full subcategory of $\DM(\Fl_{\bba_0})$ which contains the twists of the unit motives supported at the base points $\{\tau\}$ for each $\tau\in \on{Stab}_{\bba_0}$ \textup{(}cf.~\eqref{Quasi_Coxeter}\textup{)}, and which is stable under the operation $\pi_{s}^*\pi_{s,*}$ \textup{(}equivalently $\pi_{s}^!\pi_{s,!}$\textup{)} along the smooth proper projection maps $\pi_s\co \Fl_{\bba_0} \r \Fl_{\bbf_s}$ for all $s\in\bbS$ \textup{(}in the notation of \refex{simple.reflection}\textup{)}.\smallskip\\
ii\textup{)} Consider $\pi\co \Fl_{\bba_0}\to\Fl_\bbf$. The functor $\pi_!=\pi_*\co \DTM(\Fl_{\bba_0})\to \DTM(\Fl_{\bbf})$ is well-defined and the images of the generators as in i\textup{)} generate the target category.
\xprop

\pf
For i), it is immediate from \refeq{iota.v.w} that the generators $(\iota_w)_! 1$ (for $w \in W$) are obtained inductively by writing $w$ as a product of simple reflections as in \eqref{decomp_element}.
Part ii) is immediate from \refpr{change.facet} ii) and \refex{basic.WT}.
\xpf

\theo
\label{theo--generators.DTM.flag}
i\textup{)} The category $\DTM(\Fl_\bbf)$ admits a non-degenerate ``motivic'' $t$-structure. Its heart is the abelian category of mixed stratified Tate-motives
\[
\MTM(\Fl_\bbf)\,\subset\, \DTM(\Fl_\bbf).
\]
If in addition $S$ is irreducible, the simple objects in $\MTM(\Fl_\bbf)$ are precisely the intersection motives $\IC_w(n)$ for $n\in \bbZ$, $w\in W_{\bbf'}\backslash W/W_\bbf$ \textup{(}see \eqref{intersection.complex}\textup{)}.\smallskip\\
ii\textup{)} The restriction of the $\ell$-adic realization functor
\[
\rho_\ell\co \DM(\Fl_\bbf)\to \D_{\textup{\'et}}(\Fl_\bbf,\bbQ_\ell),
\]
\textup{(}cf.~Synopsis \ref{syno--motives} xvii\textup{)}, \refth{motives.Ind-schemes}\textup{)} to the subcategory $\DTM(\Fl_\bbf)$ is conservative.
Moreover, for $M \in \DTM(\Fl_\bbf)$ the following are equivalent: a\textup{)} $M$ lies in $\MTM(\Fl_\bbf)$, and b\textup{)} $\rho_\ell(M)$ is a perverse sheaf.
Finally, $\rho_\ell(\IC_w(n))$ is the $\ell$-adic intersection complex normalized relative to $S$ for all $w$, $n\in \bbZ$.
\xtheo

\pf
For i), we combine Corollary \ref{coro--t-structure} and Theorem \ref{theo--simple.objects}. Part ii) is also immediate from i) and \refle{Tate.conservative}.
\xpf

\subsection{Tate motives on double quotients of the loop group}
\label{sect--DTM.double}

We now turn to (Tate) motives on the prestack
\begin{equation}\label{double_quotient}
\calP_{\bbf'}\backslash LG/\calP_\bbf.
\end{equation}
Here we view \eqref{double_quotient} as a prestack in the sense of \refsect{DM.prestacks} where we may choose $\kappa=\omega$ to be the countable cardinal.
By \refpr{sheafification.iso}, its \'etale stackification
$(\calP_{\bbf'}\backslash LG/\calP_\bbf)^{\et }$ is given by the prestack which sends $T\in \AffSch_S$ to the groupoid $(\calP_{\bbf'}\backslash LG/\calP_\bbf)^{\et }(T)$ of diagrams of ind-schemes $T \stackrel a \gets P \stackrel b \r LG$ where $a$ is an \'etale-locally trivial torsor under the pro-algebraic group $(\calP_{\bbf'}\x\calP_\bbf)\x_ST$ (in particular an affine scheme), and $b$ is equivariant for the action of this group.
(Here ``groupoid'' is understood in the sense of an ordinary category whose morphisms are invertible. We regard it as an $\infty$-groupoid in the natural way.)
The following lemma shows that the category of motives does not change when viewing \eqref{double_quotient} either as a prestack or an honest stack.

\lemm
\label{lemm--DM.double.tau}
i\textup{)} The stack $(\calP_{\bbf'}\backslash LG/\calP_\bbf)^{\et }$ is a sheaf of groupoids for the fpqc topology.\smallskip\\
ii\textup{)} Étale sheafification of prestacks \textup{(}or alternatively, in the above description of the sheafification, the map induced from the trivial $\calP_{\bbf'} \x \calP_\bbf$-torsor on $LG$\textup{)}, yields an equivalence of categories of motives
\[
 \DM((\calP_{\bbf'}\backslash LG/\calP_\bbf)^{\et }) \,\overset{\simeq}{\longrightarrow}\, \DM(\calP_{\bbf'}\backslash LG/\calP_\bbf).
\]
In particular, $\DM(\calP_{\bbf'}\backslash LG/\calP_\bbf)$ is also equivalent to either of the categories $\DM(\calP_{\bbf'}\backslash \Fl_\bbf)$, $\DM(\Fl^{\opp}_{\bbf'}/\calP_\bbf)$.
\xlemm
\begin{proof} For i), let $T'\to T$ be a faithfully flat map in $\AffSch_S$. Let $T' \gets P' \r LG$ be an object in $(\calP_{\bbf'}\backslash LG/\calP_\bbf)^{\et }(T')$ together with a descent datum along $T'\r T$. By effectivity of descent for affine schemes \StP{0244}, the torsor $T' \gets P'$ descends to a fpqc-locally trivial torsor $T\gets P$ represented by affine schemes. The map $P'\r LG$ descends as well because every ind-scheme is an fpqc-sheaf, cf.~\refsect{ind.schemes}. By \refpr{vector.extension} (cf.~also \refle{affine.proj} and Example \ref{exam--groups} iii.a)) every fpqc-locally trivial torsor under $\calP_{\bbf'}\x\calP_{\bbf}$ is \'etale-locally trivial. Thus, $T \gets P \r LG$ is an object of $(\calP_{\bbf'}\backslash LG/\calP_\bbf)^{\et }(T)$. Part ii) follows from \refpr{sheafification.iso} applied with $\tau = \et$.
\end{proof}

All the above also applies, by symmetry, to $\Fl_{\bbf'}^{\opp}:=(\calP_{\bbf'}\backslash LG)^\et$ equipped with its stratification by orbits for the right $\calP_\bbf$-action.
Combining the Whitney-Tate stratification with the group action, \refde{DTM.G} yields two full subcategories
$$\DTM_{\calP_{\bbf'}}(\Fl_\bbf) \;\text { and }\; \DTM_{\calP_\bbf}(\Fl_{\bbf'}^\opp) \subset \DM (\calP_{\bbf'} \backslash LG / \calP_\bbf).\eqlabel{DTM.asymmetric}$$
We now address this seeming asymmetry and also upgrade \refth{generators.DTM.flag} to  equivariant motives.
Let $LG^w := \calP_{\bbf'} w \calP_\bbf$ as a closed subscheme of $LG$.
Formally, $LG^w$ is the scheme-theoretic image of the map $\calP_{\bbf'}\x\calP_\bbf\to LG$, $(p',p)\mapsto p'\dot{w}p$ where $\dot{w}\in LG(S)$ is any representative.
We write $\iota_w$ for all maps of prestacks stemming from the inclusion of $LG^w \subset LG$, in particular $\iota_w : \calP_{\bbf'} \backslash LG^w / \calP_{\bbf} \r \calP_{\bbf'} \backslash LG / \calP_{\bbf}$ and $\iota_w : \calP_{\bbf'} \backslash \Fl^w_\bbf \r \calP_{\bbf'} \backslash \Fl_\bbf$.

\theo
\label{theo--equivariant.DTM.flag} Let $S$ be irreducible.\smallskip\\
i\textup{)} The functor $\iota_w^! : \DM(\calP_{\bbf'} \backslash LG / \calP_{\bbf}) \r \DM(\calP_{\bbf'} \backslash LG^w / \calP_{\bbf})$ has a left adjoint denoted $(\iota_w)_!$.\smallskip\\
ii\textup{)} The two full categories in \textup{\refeq{DTM.asymmetric}} both agree with the full subcategory of $\DM(\calP_{\bbf'}\backslash LG/\calP_\bbf)$ generated \textup{(}under shifts and colimits\textup{)} by the objects $(\iota_w)_!1(n)$ for $n\in \bbZ$, $w\in W_{\bbf'}\backslash W/W_\bbf$.
This category is denoted $\DTM(\calP_{\bbf'}\backslash LG/\calP_\bbf)$.\smallskip\\
iii\textup{)} The motivic $t$-structures on the categories $\DTM_{\calP_{\bbf'}}(\Fl_\bbf)$ and $\DTM_{\calP_{\bbf}}(\Fl_ {\bbf'}^\opp)$ yield two $t$-structures on $\DTM(\calP_{\bbf'} \backslash LG / \calP_\bbf)$.
For $\bbf' = \bbf$, these two t-structures agree in which case its heart is denoted
\[
\MTM(\calP_{\bbf}\backslash LG/\calP_\bbf)\,\subset\, \DTM(\calP_{\bbf}\backslash LG/\calP_\bbf).
\]
It is generated by the intersection motives $\IC_w(n)$ for $n\in\bbZ$, $w\in W_{\bbf}\backslash W/W_\bbf$ from \refth{generators.DTM.flag}.
Moreover, the forgetful functor $\MTM(\calP_{\bbf}\backslash LG/\calP_\bbf)\to \MTM(\Fl_\bbf)$ is fully faithful, and induces a bijection on the isomorphism classes of simple objects.
\xtheo

\pf
For i), we may replace $LG/\calP_\bbf$ and $LG^w / \calP_\bbf$ by their étale sheafification (cf.~\refth{descent.prestacks}), which are $\Fl_\bbf$ and its stratum $\Fl^w_\bbf$, respectively.
These are ind-schemes, so we are done by \refth{motives.Ind-schemes} and \refle{functoriality.equivariant}.

For ii), the objects $(\iota_w)_! 1(n)$ are independent of the role of $\bbf'$ vs.~$\bbf$. So it is enough to show that $\DTM_{\calP_{\bbf'}}(\Fl_\bbf)$ is generated by these objects as a subcategory of $\DM(\calP_{\bbf'} \backslash LG / \calP_\bbf)$.
By \refpr{DTM.G}, which is applicable to $\calP_{\bbf'}$ acting on $\Fl_\bbf$ by \refle{affine.proj} and \refle{orbit.flag}, we can reduce this claim to the case of the action of some algebraic quotient $\calP_{\bbf', i}$ of $\calP_{\bbf'}$ on some subscheme (in $\Sch_S^\ft$) $\Fl_{\bbf, i} \subset \Fl_\bbf$.
The stratification of $\Fl_{\bbf, i}$ by $\calP_{\bbf', i}$-orbits is cellular by \refco{Fl.stratification} and the stabilizers are connected by \refle{orbit.flag}i), so we are done by \refco{DTM.G.X.generators}.

For iii), the $t$-structure on $\DTM_{\calP_{\bbf'}}(\Fl^w_\bbf)$ is characterized by the property that its $\le 0$-part is generated by means of arbitrary colimits by the objects $1_{\Fl_\bbf^w}(n)[\dim \Fl_\bbf^w]$.
Indeed, by \refpr{DM.G.homotopy.invariant} and \refex{affine.proj}.ii), we may replace $\calP_{\bbf'}$ by some algebraic quotient acting on $\Fl_\bbf^w$ and then apply \refpr{generators.DTM.G}.
By construction of the glued t-structure in \refco{t-structure}, $\DTM_{\calP_{\bbf'}}(\Fl_\bbf)^{\le 0}$ is generated by means of arbitrary colimits by the objects $(\iota_w)_! 1_{\Fl_\bbf^w}(n)[\dim \Fl_\bbf^w]$.
Thus, for $\bbf' = \bbf$, the two t-structures have the same $\le 0$-part and therefore agree.
\xpf

\rema
For $\bbf' \ne \bbf$ it may happen that $\dim \Fl_\bbf^w \ne \dim \Fl^{\opp, w}_{\bbf'}$, so the $t$-structures are different. As an example consider $G=\GL_2$, $\bbf'=\bba_0$, $\bbf= 0$, and $w:=s_1$ the simple finite reflection. Then $\Fl_\bbf$ is the affine Grassmannian, and $\Fl_\bbf^w=\{e\}$ is the base point. On the other hand $\Fl^{\opp}_{\bbf'}$ is the full affine flag variety, and $\Fl^{\opp,w}_{\bbf'}=\mathbf{P}^1_S$.
\xrema

We end this section by pointing out the following corollary which is needed in \refsect{intersection}. Specialize to $\bbf'=\bbf=\{0\}$ being the base point, so that $W_{0}\backslash W/W_{0}=X_*(T)_+$. The action of $L^+\bbG_{m,S}$ on $LG$ by changing the variable $\varpi$ preserves the subgroup $L^+G$, and thus gives an action on the double quotient $L^+G\backslash LG/L^+G$, and the affine Grassmannian $\Gr:=(LG/L^+G)^\et$.

\coro\label{coro--equivariant.IC}
For each $\mu\in X_*(T)_+$, $n\in\bbZ$, the object $\IC_\mu(n)$ is $L^+\bbG_m$-equivariant, and defines an object
\[
\IC_\mu(n)\;\in\;\DM\left(L^+\bbG_{m,S}\backslash (L^+G\backslash LG/L^+G)\right),
\]
supported on the Schubert variety $\Gr^{\leq \mu}$.
\xcoro
\pf The statement about the support follows from \refth{equivariant.DTM.flag} iii), and it is enough to prove that each $\IC_\mu(n)$ is $L^+\bbG_{m,S}$-equivariant.
But this is immediate from the $L^+\bbG_{m,S}$-invariance of the $L^+G$-orbits $\Gr^\mu\subset \Gr$:
by \refpr{equivariant.MTM}, we have $\IC_\mu(n)\in\MTM_{L^+G\rtimes L^+\bbG_{m,S}}(\Gr)$. The latter category is a full subcategory of $\DM(L^+\bbG_{m,S}\backslash (L^+G\backslash LG/L^+G))$ using \refpr{sheafification.iso} for the \'etale sheafifications.
\xpf

\section{Intersection motives on moduli stacks of shtukas}
\label{sect--intersection}
In this final section, we show that the intersection (cohomology) motive of the moduli stack of $G$-shtukas with bounded modification is defined independently of the standard conjectures on $t$-structures on triangulated categories of motives, cf.~\refco{prelim.intersection.motives} below.
Our presentation is expository in parts, and follows \cite[\S 2]{Lafforgue:Chtoucas}.
We put a stronger emphasis on the stack of relative positions, and the invariant which is the global function field analogue of the Grothendieck-Messing period map, cf.~\cite{ScholzeWeinstein:Moduli}.

Let $X$ be a geometrically connected smooth projective curve over the finite field $k=\Fq$, and let $G$ be a split reductive $k$-group scheme.

For an effective divisor $N\subset X$, we let $\Bun_{N,G}$ denote the {\em moduli stack of $G$-torsors on $X$ with level-$N$-structure} viewed as an \'etale sheaf of groupoids $(\AffSch_k)^\opp\to \Gpd$. Then $\Bun_G:=\Bun_{\varnothing, G}$ is representable by a quasi-separated Artin stack locally of finite type over $k$ (cf.~e.g.~\cite[Prop.~1]{Heinloth:Uniformization}), and the forgetful map $\Bun_{N,G}\to \Bun_G$ is representable by a torsor under the restriction of scalars $\on{Res}_{N/k}(G\x N)$ (a schematic smooth affine surjective map).

\subsection{The stack of relative positions}
We need the ``fusion version'' of the loop group $L_IG\to X^I$ introduced in \cite{BeilinsonDrinfeld:Quantization}. The relation to the loop group $LG\to S$ from \refsect{loop.grps} is explained in \refex{fusion.loop}.

For a test scheme $T\in \AffSch_k$, and a relative effective Cartier divisor $D\subset X_T$ which is finite and locally free over $T$, we denote by $\hat{D}_T$ the spectrum of the ring of global functions on the formal affine scheme $(X_T/D_T)^\wedge$ obtained as the completion of $X_T$ along $D_T$. Then $D_T\subset \hat{D}_T$ defines a Cartier divisor, and thus $\hat{D}_T^o:=\hat{D}_T\backslash D_T$ is an affine $k$-scheme as well, cf.~\cite[\S3.1.1]{HainesRicharz:TestFunctionsWeil} for details.
For example, if $D$ is the graph of a point $x\in X(T)$, then $\hat{D}_T= \Spec(R\pot{\varpi_x})$ and $\hat{D}_T^o= \Spec(R\rpot{\varpi_x})$ where $T=\Spec(R)$ and $\varpi_x$ is a local coordinate at $x\in X(T)$.

For any finite index set $I$, we consider the loop group functor $L_IG\co (\AffSch_k)^\opp\to \Sets$  defined by
\begin{equation}\label{glob_loop_group}
L_IG(T)\defined \left\{\left(x,g\right)\;|\; x=\{x_i\}\in X^I(T),\; g\in G\big( \hat{\Gamma}_x^o\big)\right\},
\end{equation}
where $\Gamma_x\subset X_T$ denotes the relative effective Cartier divisor given by the union of the graphs of the points $x_i\in X(T)$, $i\in I$.
Likewise, the positive loop group functor $L_I^+G\co (\AffSch_k)^\opp\to \Sets$ is defined as in \eqref{glob_loop_group} by replacing $ \hat{\Gamma}_x^o$ with $ \hat{\Gamma}_x$. The projection $L_IG\to X^I$ (resp.~$L_I^+G\to X^I$) makes $L_IG$ (resp.~$L_I^+G$) into an ind-affine $X^I$-group ind-scheme (resp.~pro-algebraic $X^I$-group scheme), cf.~\cite[Lem.~3.2]{HainesRicharz:TestFunctionsWeil}.
Note that $L^+_IG$ is a special case of the general set-up introduced in \refpr{vector.extension} below by viewing $X^I$ as base scheme, and considering the relative curve $X^I\x X\to X^I$ together with the universal degree $\#I$ divisor in $X^I\x X$, cf.~also \refex{groups} ii). Clearly, $L^+_IG\subset L_IG$ defines a subgroup functor over $X^I$.

\defi\label{defi--rel.pos}
For any finite index set $I$, the \'etale sheaf of groupoids $\AffSch_k^\opp\to \Gpd$ given by
\[
\Rel_I\defined(L^+_IG\backslash L_IG/L_IG^+)^\et
\]
is called the {\em stack of relative positions}.
\xdefi

This is an affine analogue of the relative position defined in Deligne-Lusztig theory, cf.~\cite[\S1.2]{DeligneLusztig:RepsGrps}.
The importance of this stack lies in its relation to the Hecke stack (resp.~moduli stack of shtukas) via the relative position \eqref{invariant.map} (resp.~\eqref{shtuka.inv}).

Note that $\Rel_I$ is an fpqc sheaf of groupoids: this follows as in \refle{DM.double.tau} i) using that every $L^+_IG$-torsor is \'etale locally trivial by \refpr{vector.extension}.

\exam\label{exam--fusion.loop}
Let $x\co S\to X$ be a map where $S\in \AffSch_k$ is the spectrum of a local ring. Then, for $I=\{*\}$ a singleton, the fiber of $L_{I}G\to X$ (resp.~$L_{I}G\to X$) over $x$ is the loop group $LG_x$ (resp.~$L^+G_x$) considered in \refsect{loop.grps} formed by using as base scheme $S$, the group scheme $G_x:=G\x S$ and the local coordinate $\varpi=\varpi_x$ defined by $x$. Thus,
\begin{equation}\label{fiber_over_Grass}
\Rel_{I}\x_{X,x}S\;=\;(L^+G_x\backslash LG_x/L^+G_x)^\et.
\end{equation}
If $S$ is the spectrum of a field, the underlying topological space of $(L^+G_x\backslash LG_x/L^+G_x)^\et$ is the topological space associated with the partial ordered set $(X_*(T)_+,\leq)$, cf.~\refle{orbit.flag} and \refex{simple.reflection}.
For $\la, \mu\in  X_*(T)_+$, this means that $\mu$ specializes to $\la$ if and only if $\la \leq \mu$.
For general $I$, the underlying topological space of $\Rel_I$ is a fusion version of the topological space $(X_*(T)_+,\leq)$ with fusion structure induced by the monoid structure of $X_*(T)_+$, cf.~\cite[\S5.3.10]{BeilinsonDrinfeld:Quantization}.
\xexam

The following lemma is a slight reformulation of \cite[Rmk.~5.1]{MirkovicVilonen:Satake}.

\lemm\label{lemm--var.action}
For $I=\{*\}$ a singleton, there is canonical map of \'etale sheaves of groupoids
\begin{equation}\label{unif.unab}
\Rel_I \;\to\; \left(L^+\bbG_m\backslash (L^+G\backslash LG/L^+G)\right)^\et,
\end{equation}
where $L^+\bbG_m$ acts \textup{(}as in \refco{equivariant.IC}\textup{)} by changing the formal variable $\varpi$ used to form $LG$, $L^+G$.
\xlemm
\qed

\subsection{The invariant}\label{sect--invariant}

For any finite index set $I$, the {\em Hecke stack $\Hecke_I$} is the \'etale sheaf of groupoids $(\AffSch_k)^\opp\to \Gpd$ given by $T\mapsto (\calE, \calE',\{x_i\}_{i\in I},\al)$ where $\calE,\calE'\in \Bun_G(T)$ are torsors, $\{x_i\}_{i\in I}\in X^I(T)$ are points, and $\al\co \calE|_{X_T\backslash \cup x_i}\to \calE'|_{X_T\backslash \cup x_i}$ is a map (i.e., an isomorphism) of torsors. A convenient notation (cf.~\cite{Heinloth:SurveyLafforgue}) for the Hecke stack is
\[
\Hecke_I\;=\; \lan \calE\underset{I}{\overset{\al}{\dashrightarrow}} {\calE'}\ran,
\]
where $\al\co\calE\dashrightarrow \calE'$ is a birational map defined outside $\cup x_i$, i.e., the torsor $\calE'$ differs from $\calE$ by an ``algebraic modification'' at a neighborhood of $\cup x_i$. The points $\{x_i\}_{i\in I}$ are called the {\em paws} (or {\em legs}) of the modification $\calE\dashrightarrow \calE'$. Further, $\Hecke_I$ is representable by a quasi-separated ind-Artin stack, ind-(locally of finite type) over $k$ (cf.~\cite[Lem.~3.1]{Varshavsky:Moduli}), and equipped with a forgetful map $\Hecke_I\to X^I$.

Following the notation in \cite{KottwitzRapoport:Crystals} (cf.~also \cite[\S 1.2.1]{Zhu:Affine}), the {\em relative position} (or {\em invariant})
\begin{equation}\label{invariant.map}
\inv\co \Hecke_I\to \Rel_I, \;\;(\calE\underset{I}{\overset{\al}{\dashrightarrow}}\calE')\;\mapsto\; \inv(\al),
\end{equation}
is the map of \'etale sheaves of groupoids over $X^I$ defined in terms of \refpr{sheafification.iso} as follows. For $T\in \AffSch_{X^I}$, and $(\calE\dashrightarrow \calE')\in \Hecke_I(T)$, we consider the \'etale sheaf $P\co (\AffSch_{T})^\opp\to \Sets$ given by
\[
P(T')= \Isom(\calE'|_{ \hat{\Gamma}_x},\calE^0|_{ \hat{\Gamma}_x})\x \Isom(\calE|_{ \hat{\Gamma}_x}, \calE^0|_{ \hat{\Gamma}_x}),
\]
where $x=\{x_i\}_{i\in I}\in X^I(T)$ are the legs of the modification, and $\calE^0$ denotes the trivial torsor. The map $a\co P\to T$ has the structure of a left $(L_I^+G\x L_I^+G)\x_{X^I}T$-torsor via the rule $(g_1,g_2)*(\be_1,\be_2)=(g_1\be_1,g_2\be_2)$. It is \'etale locally trivial by the approximation argument given in \cite[Lem.~3.4 ii)]{HainesRicharz:TestFunctionsWeil}, and thus an \'etale torsor. We now define a map $b\co P\to L_IG$ by sending $(\be_1,\be_2)\in P(T')$
to the element $\be_1\al\be_2^{-1}\in\Aut(\calE_0)( \hat{\Gamma}_x^o)=G( \hat{\Gamma}_x^o)$. The map $b$ is equivariant for the left $L^+_IG\x L^+_IG$-action on $L_IG$ given by $(g_1,g_2)*g:=g_1gg_2^{-1}$. This defines the relative position
\[
\inv(\al):= ( P \stackrel {b\x a} \lr L_IG\x T)\in (L^+_IG\backslash L_IG/L_IG^+)^\et(T)=\Rel_I(T).
\]

\defi(\cite[Def.~1.2]{Lafforgue:Chtoucas})
For any effective divisor $N\subset X$, and any partition $I=I_1\sqcup \ldots\sqcup I_r$, $r\in \bbZ_{\geq 0}$, the iterated Hecke stack $\Hecke_{N,I}^{(I_1,\ldots,I_r)}$ with level-$N$-structure is the \'etale sheaf of groupoids $\AffSch_k\to \Gpd$ given by
\[
\Hecke_{N,I}^{(I_1,\ldots,I_r)}\defined \lan(\calE_r,\be_r)\underset{I_r}{\overset{\al_r}{\dashrightarrow}}(\calE_{r-1},\be_{r-1})\underset{I_{r-1}}{\overset{\al_{r-1}}{\dashrightarrow}}\ldots\underset{I_{2}}{\overset{\al_2}{\dashrightarrow}} (\calE_1,\be_1)\underset{I_{1}}{\overset{\al_1}{\dashrightarrow}} (\calE_0,\be_0)\ran,
\]
i.e., $\Hecke_{N,I}^{(I_1,\ldots,I_r)}(T)$ classifies data $((\calE_j,\be_j)_{j=1,\ldots, r}, \{x_i\}_{i\in I}, (\al_j)_{j=1,\ldots, r})$
where $(\calE_j,\be_j)\in \Bun_{N,G}(T)$ are torsors with level-$N$-structure, $\{x_i\}_{i\in I}\in (X\backslash N)^I(T)$ are points, and
\[
\al_j\co (\calE_j,\be_j)|_{X_T\backslash (\cup_{i\in I_j}x_i)}\to (\calE_{j-1},\be_{j-1})|_{X_T\backslash (\cup_{i\in I_j}x_i)}
\]
are maps of torsors with level-$N$-structure.
\xdefi

As above $\Hecke_{N,I}^{(I_1,\ldots,I_r)}$ is representable by a quasi-separated ind-Artin stack ind-locally of finite type over $k$, and equipped with the forgetful map $\Hecke_{N,I}^{(I_1,\ldots,I_r)}\to (X\backslash N)^I\subset X^I$. We need the following construction (cf.~also \cite[(1.5)]{Lafforgue:Chtoucas}): Fix a total order $I=\{1,\ldots,n\}$, $n=\#I$ compatible with the partition $I=I_1\sqcup \ldots\sqcup I_r$. This defines a refinement $I_1=\{1\}\sqcup\ldots\sqcup\{l_1\}$, $I_2=\{l_{1}+1\}\sqcup\ldots\sqcup\{l_2\}$,\ldots and also the new partition $I=\{1\}\sqcup\ldots\sqcup\{n\}$. There are maps of \'etale sheaves of groupoids over $X^I$ given by
\begin{equation}\label{funda.diag}
\Hecke_{N,I}^{(I_1,\ldots,I_r)} \;\overset{\pi_{I_\bullet}}{\longleftarrow}\; \Hecke_{N,I}^{(\{1\},\ldots,\{n\})} \;\overset{\inv_{I_\bullet}}{\longrightarrow}\; \bigsqcap_{i=1,\ldots,n}\Rel_{\{i\}}.
\end{equation}
Here $\pi_{I_\bullet}$ is given by forgetting certain $(\calE_j,\be_j)$'s and composing the $\al_j$'s in between as follows
\[
((\calE_n,\be_n)\underset{\{n\}}{\overset{\al_n}{\dashrightarrow}}\dots\underset{\{1\}}{\overset{\al_1}{\dashrightarrow}} (\calE_0,\be_0))\;\mapsto\;  ((\calE_n,\be_n)\underset{\{n,\ldots,l_{r-1}+1\}}{\overset{\al_{\tiny{l_{r-1}+1}}\circ\dots\circ\al_n}{\dashrightarrow}}(\calE_{l_r},\be_{l_r})\dashrightarrow\dots\dashrightarrow(\calE_{l_1},\be_{l_1})\underset{\{l_{1},\ldots, 1\}}{\overset{\al_{1}\circ\dots\circ\al_{l_{1}}}{\dashrightarrow}} (\calE_0,\be_0)).
\]
The relative position is given by $\inv_{I_\bullet}\co ((\calE_\bullet,\be_\bullet),\al_\bullet)\mapsto (\inv(\al_i))_{i=1,\ldots,n}$.

\exam For $N=\varnothing$, and $I=\{1,2\}$ two elements ($r=1$), the maps are given by
\[
(\calE_2\underset{\{2,1\}}{\overset{\al_1\circ\al_2}{\dashrightarrow}} \calE_0)\;\overset{\pi_{I_\bullet}}{\mapsfrom}\; (\calE_2\underset{\{2\}}{\overset{\al_2}{\dashrightarrow}}\calE_1\underset{\{1\}}{\overset{\al_1}{\dashrightarrow}} \calE_0)\;\overset{\inv_{I_\bullet}}{\mapsto}\; (\inv(\al_1),\inv(\al_2)).
\]
\xexam

\subsection{Intersection motives on moduli stacks of shtukas}\label{sect--shtukas.def}
The Hecke stack is used to construct the moduli stack of shtukas as follows, cf.~\cite[D\'ef.~2.1]{Lafforgue:Chtoucas}. For any effective divisor $N\subset X$, and any partition $I=I_1\sqcup \ldots\sqcup I_r$, $r\in \bbZ_{\geq 0}$, the {\em moduli stack of iterated $G$-shtukas $\Sht_{N,I}^{(I_1, \ldots, I_r)}$ with level-$N$-structure} (or simply {\em moduli stack of $G$-shtukas}) is the \'etale sheaf of groupoids $\AffSch_k^\opp\to \Gpd$ given by
\[
\Sht_{N,I}^{(I_1,\ldots,I_r)}\defined \lan (\calE_r,\be_r)\overset{\al_r}{\underset{I_r}{\dashrightarrow}}(\calE_{r-1},\be_{r-1})\overset{\al_r}{\underset{I_{r-1}}{\dashrightarrow}}\ldots\overset{\al_2}{\underset{I_2}{\dashrightarrow}} (\calE_1,\be_1)\overset{\al_1}{\underset{I_1}{\dashrightarrow}} (\calE_0,\be_0)=({^\tau\calE_r},{^\tau\be}_r) \ran,
\]
where ${^\tau\calE}:= (\id_X\x \Frob_{T/k})^*\calE$ denotes the pullback, and $\Frob_{T/k}$ is the relative Frobenius. Formally, $\Sht_{N,I}^{(I_1,\ldots,I_r)}$ is the fibre product of the forgetful map $\Hecke_{N,I}^{(I_1,\ldots,I_r)}\to \Bun_{G,N}\x \Bun_{G,N}$, $((\calE_r,\be_r)\dashrightarrow\ldots\dashrightarrow(\calE_0,\be_0))\mapsto ((\calE_r,\be_r),(\calE_0,\be_0))$ with the Frobenius correspondence $\id\x \Frob\co \Bun_{G,N}\to \Bun_{G,N}\x \Bun_{G,N}$. There is the forgetful map $\Sht_{N,I}^{(I_1,\ldots,I_r)}\to (X\backslash N)^I\subset X^I$.

We fix a Borel pair $T\subset B\subset G$. By \cite[Prop.~2.16]{Varshavsky:Moduli} (cf.~also \cite[Prop.~2.6]{Lafforgue:Chtoucas}), there is a presentation of the reduced locus
\begin{equation}\label{shtuka.pres}
(\Sht_{N,I}^{(I_1,\ldots,I_r)})_\red\;=\; \colim_{\underline{\mu}}\Sht_{N,I,{\underline{\mu}}}^{(I_1,\ldots,I_r)},
\end{equation}
with transition maps closed immersions. Here $\underline{\mu}=(\mu_i)_{i\in I}\in X_*(T)_+$ runs through the admissible tuples (i.e., $\sum_{i\in I}\mu_i=0$ in $\pi_1(G)$), and each $\Sht_{N,I,\underline{\mu}}^{(I_1,\ldots,I_r)}$ is representable by a non-empty Deligne-Mumford stack locally of finite type. Thus, \eqref{shtuka.pres} is an ind-Deligne-Mumford stack ind-locally of finite type over $k$.

Fixing a total order on $I=\{1,\ldots,n\}$ compatible with $I=I_1\sqcup\ldots\sqcup I_r$, the diagram \eqref{funda.diag} restricts to the diagram
\begin{equation}\label{shtuka.inv}
\Sht_{N,I}^{(I_1,\ldots,I_r)} \;\overset{\pi_{I_\bullet}}{\longleftarrow}\; \Sht_{N,I}^{(\{1\},\ldots,\{n\})} \;\overset{\inv_{I_\bullet}}{\longrightarrow}\; \bigsqcap_{i=1,\ldots,n}\Rel_{\{i\}}.
\end{equation}
A special case of \refpr{f_!.ind-Artin} is the following result.

\prop \label{prop--existence.functors}
There exists an adjunction of functors
\[
\pi_{I_\bullet,!} : \;\DM\left(\Sht_{N,I}^{(\{1\},\ldots,\{n\})}\right)\; \leftrightarrows \;\DM\left(\Sht_{N,I}^{(I_1,\ldots,I_r)}\right) \; : \pi_{I_\bullet}^!.
\]
\xprop
\qed

For any $\mu\in X_*(T)_+$, $m\in\bbZ$, we denote by
\[
\IC_{\{*\},\mu}(m)\;\in\; \DM(\Rel_{\{*\}})
\]
the $!$-pullback of $\IC_{\mu}(m)\in \DM(L^+\bbG_{m,k}\backslash (L^+G\backslash LG/L^+G))$ under \eqref{unif.unab}, cf.~\refco{equivariant.IC}.

\defi \label{defi--intersection.motive}
Fix a total order $I=\{1,\ldots,n\}$, and a compatible partition $I=I_1\sqcup \ldots\sqcup I_r$. For each effective divisor $N\subset X$, each admissible tuple $\underline{\mu}=(\mu_i)_{i\in I}\in X_*(T)_+$ and each $\underline{m}=(m_i)_{i\in I}\in\bbZ^I$, one defines
\[
\calF_{\underline{\mu},\underline{m}}=\calF_{N,I,\underline{\mu},\underline{m}}^{(I_1,\ldots,I_r)}\defined \pi_{I_\bullet,!}\left(\inv_{I_\bullet}^! \left(\boxtimes_{i=1}^n\IC_{\{i\},\mu_i}(m_i)\right)\right) \;\in\; \DM\left(\Sht_{N,I}^{(I_1,\ldots,I_r)}\right).
\]
(see \refpr{boxtimes} for the box product, \refre{DM.prestacks} ii) for the pullback and \refpr{existence.functors} for the pushforward).
\xdefi

\coro\label{coro--prelim.intersection.motives} Let $\ell\in\bbZ$ be a prime number invertible on $k$. For each tuple of data as in \refde{intersection.motive}, the motive $\calF_{\underline{\mu},\underline{m}}$ is supported on $\Sht_{N,I,\underline{\mu}}^{(I_1,\ldots,I_r)}$, and its $\ell$-adic realization
\[
\rho_\ell\left(\calF_{\underline{\mu},\underline{m}}\right)\;\in\; \D_\et\left(\Sht_{N,I,\underline{\mu}}^{(I_1,\ldots,I_r)},\bbQ_\ell\right)
\]
is \textup{(}up to twist and the choice of a lattice in the adelic center\textup{)} the intersection complex defined in \textup{\cite[D\'ef.~2.14]{Lafforgue:Chtoucas}}. In particular, the motives $\calF_{\underline{\mu},\underline{m}}$ are normalized such that the $*$-restrictions of $\rho_\ell(\calF_{\underline{\mu},\underline{m}})$ along the fibers of the map $p\co \Sht_{N,I}^{(I_1,\ldots,I_r)}\to (X\backslash N)^I$ are perverse. Further, the $\ell$-adic realization of the motive
\[
 p_!\left(\calF_{\underline{\mu},\underline{m}}\right) \;\in\; \DM((X\backslash N)^I)
\]
is \textup{(}up to the normalizations above, and the bound of the Harder-Narasimhan slopes\textup{)} the intersection cohomology complex defined in \textup{\cite[D\'ef.~4.1]{Lafforgue:Chtoucas}}.
\xcoro

\rema In \cite[D\'ef.~2.14]{Lafforgue:Chtoucas}, the intersection complexes are normalized to be pure of weight zero along the fibers of the structure map $p$. For this reason, a square root of the cardinality of the residue field in a finite extension of $\bbQ_\ell$ is fixed in \textit{loc.~cit.} in order to define half Tate twists. Since this is not possible in the motivic setting, we have to add the Tate twists in \refde{intersection.motive} as an additional datum.
\xrema

\begin{proof}[Proof of \refco{prelim.intersection.motives}] We need to relate the $\ell$-adic realization of $\calF_{\underline{\mu},\underline{m}}$ to the intersection complex of $\Sht_{N,I,\underline{\mu}}^{(I_1,\ldots,I_r)}$, cf.~\cite[Def.~2.14]{Lafforgue:Chtoucas}. There is a Cartesian diagram
\[
\xymatrix{
\Sht_{N,I,\underline{\mu}}^{(I_1,\ldots,I_r)} \ar[r] \ar[d]^{i_\mu} &\bigsqcap_{i} \left(L^+_{\{i\}}G\backslash \Gr_{\{i\}}^{\leq \mu}\right)^\et \ar[d]^{i_\mu} \\
\Sht_{N,I}^{(I_1,\ldots,I_r)} \ar[r]^{\inv_{I_\bullet}} & \bigsqcap_{i}\Rel_{\{i\}},}
\]
where $\Gr_{\{*\}}^{\leq \mu}\subset (L_{\{*\}}G/L^+_{\{*\}}G)^\et$ is the preimage of $(L^+\bbG_m\backslash \Gr_G^{\leq \mu})^\et$ under \eqref{unif.unab}.
The $L^+_{\{*\}}G$-action on $\Gr_{\{*\}}^{\leq \mu}$ factors through a finite-dimensional quotient $L^+_{\{*\}}G\to G_{j}$, $j>\!\!>0$, with split pro-unipotent kernel, cf.~\refpr{vector.extension}. The top horizontal arrow induces the map onto the local model
\[
\epsilon\co \Sht_{N,I,\underline{\mu}}^{(I_1,\ldots,I_r)}\;\to\; \bigsqcap_{i=1,\ldots,n}\left(G_j\backslash \Gr_{\{i\}}^{\leq \mu}\right)^\et
\]
constructed in \cite[Prop.~2.8, 2.9]{Lafforgue:Chtoucas}. Denote $M:=\boxtimes_{i=1}^n\IC_{\{i\},\mu_i}(m_i)$ which we view as a motive on the target of $\epsilon$ by \refpr{DM.G.homotopy.invariant}. Together with \refle{Lurie.co.limit} and base change for closed immersions (\refpr{f_!.ind-Artin}) it follows that there is an equivalence $\inv_{I_\bullet}^! i_{\mu,!}M \simeq  i_{\mu,!}\epsilon^!M$. Thus, \cite[Cor.~2.16, 2.18]{Lafforgue:Chtoucas} shows (the lattice in the adelic center in {\em loc.~cit.} does not affect the isomorphisms) that the $\ell$-adic realization of $(\pi_{I_\bullet}\circ i_{\mu})_!(\epsilon^!M)$ is the intersection complex, cf.~\refpr{existence.functors}. Here we used \refth{f!.Artin} for the $\ell$-adic realization of the $!$-push forward. The rest of the corollary is immediate from this.
\xpf

\rema \label{rema--WT.properties.fusion}
i\textup{)} Similarly to \cite[D\'ef.~4.1]{Lafforgue:Chtoucas} one may also bound the Harder-Narasimhan slope of the bundles forming the shtuka in order to obtain locally constructible intersection cohomology motives. \smallskip\\
ii\textup{)} There is an analogous version of \refco{prelim.intersection.motives} for the fusion Grassmannians, cf.~\cite[Thm.~1.17]{Lafforgue:Chtoucas} for the $\ell$-adic version. We plan to improve on our result in two ways: independence of the intersection motives of the fixed total order on $I$ (and hence also in \refco{prelim.intersection.motives}), and compatibility with the fusion structure coming from the motivic Satake equivalence \cite{RicharzScholbach:Satake}. Both statements rely on Whitney-Tate properties of fusion Grassmannians, and ultimately on the Tateness of the convolution morphism. This is work in progress.\smallskip\\
iii\textup{)}
Given ii\textup{)}, it seems possible to obtain a $(S=T)$-Theorem in this context, cf.~\cite[Prop.~6.2]{Lafforgue:Chtoucas}. To proceed further, a major hurdle seems to be a variant of Drinfeld's lemma for $\DM$.
\xrema

\appendix

\section{Ind-spaces and pro-groups}
In this appendix, we state our conventions about ind-algebraic spaces/ind-schemes (\S\ref{sect--ind.schemes}), pro-algebraic groups (\S\ref{sect--algebraic.grps}) and their action on ind-algebraic spaces (\S\ref{sect--pro.action}). In \S \ref{sect--torsors}, we prove that pro-algebraic groups which are constructed as a ``positive loop group'' (or ``jet group'') satisfy a remarkable property: every torsor under such a pro-algebraic group admits sections \'etale-locally.

\subsection{Strict ind-spaces}\label{sect--ind.schemes}
Let $S$ be any scheme, and let $\AffSch_S$ be the category of affine schemes equipped with a map to $S$. A {\em \textup{(}strict\textup{)} ind-algebraic space $X$ over $S$} is a presheaf $X\co (\AffSch_S)^\opp\to \Sets$ which admits a presentation $X=\text{colim}_iX_i$ where $\{X_i\}_{i\in I}$ is a direct system of algebraic spaces $X_i$ with transition maps $t_{i,j}\co X_i\r X_j$ ($j\geq i$) being closed immersions. Here $I$ is a countable directed (a.k.a.~filtered) index set. Every ind-algebraic space is an fpqc sheaf on $\AffSch_S$ (because every algebraic space defines a sheaf by \StP{03W8}, and filtered colimits of sheaves on $\AffSch_S$ are computed termwise).

By definition, the category $\IndAlgSp_S$ of ind-algebraic spaces over $S$ is a full subcategory of presheaves.
Note that every map every map $T\to X$ from a quasi-compact algebraic space factors over some $X_i$ (by quasi-compactness of $T$ it is covered by finitely many affine schemes).
Further, every map $f\co \colim_{i\in I}X_i\r \colim_{j\in J}Y_j$ can be written as a colimit of maps $f_{(i,j)}\co X_{(i,j)}:=X_i\x_YY_j\r Y_j=:Y_{(i,j)}$, $(i,j)\in I\x J$. In particular, after possibly changing the presentation every map $f\co X\r Y$ is a colimit of maps $f_i\co X_i\r Y_i$ for the same directed index set.
Thus, the category $\IndAlgSp_S$ is closed under fibre products, i.e., $X\x_YZ=\colim_i X_i\x_{Y_i}Z_i$ is an ind-algebraic space for any maps $X\to Y, Z\to Y$ in $\IndAlgSp_S$.

 Let $\mathcal{P}$ be a property of algebraic spaces (or morphism of algebraic spaces). An ind-algebraic space $X$ (or a map $X\to Y$) is said to have $\on{ind-}\calP$ if there exists a presentation $X=\on{colim}_iX_i$ where each $X_i$ has property $\mathcal{P}$. A map $f\colon X\r Y$ of ind-algebraic spaces is said to have property $\mathcal{P}$ (resp. to be schematic and to have $\mathcal{P}$) if for all $T\in \AffSch_S$, the pullback $X\x_YT$ is an algebraic space (resp.~a scheme) and the map $f\x_YT $ has property $\mathcal{P}$.

Likewise, the category $\IndSch_S$ of (strict) ind-schemes over $S$ is the full subcategory of $\IndAlgSp_S$ of those objects $X=\colim_iX_i$ where each $X_i$ is a scheme.

\subsection{Strictly pro-algebraic groups}\label{sect--algebraic.grps}
Let $S$ be any scheme. A {\it \textup{(}strictly\textup{)} pro-algebraic group scheme $G$ over $S$} is a presheaf $G\co (\AffSch_S)^\opp\r \Grps$ which admits a presentation $G= \lim_i G_i$ where $\{G_i\}_{i\in \bbN}$ is an inverse system of smooth $S$-affine (hence finitely presented) $S$-group schemes $G_i$ with smooth surjective transition maps of $S$-groups $\pi_{i,j}\co G_j\r G_i$ for $j\geq i$.

\lemm \label{lemm--pro.group}
Let $G=\lim_{i\in \bbN}G_i$ be a pro-algebraic $S$-group.\smallskip\\
i\textup{)} The presheaf $G$ is representable by a faithfully flat $S$-affine $S$-group scheme.\smallskip\\
ii\textup{)} For each $i\in \bbN$, the map $G\r G_i$ is faithfully flat, and hence there is a short exact sequence of flat $S$-affine $S$-group schemes $1\r U_i\r G\r G_i\r 1.$
\smallskip\\
iii\textup{)}
If all the $G_i$ have connected fibers over $S$, then so does $G$ in which case they are automatically geometrically connected.
\xlemm
\pf
Let $\calA_i$ be the Hopf $\cO_S$-algebras defining $G_i$. Parts i) and ii) follow by noting that the colimit $\calA := \colim_i \calA_i$ has a natural Hopf algebra structure. It is faithfully flat since all the $G_j \r G_i$ are smooth surjective. Part i) and ii) are immediate.
For iii), note that any (pro-)algebraic $S$-group $G$ with connected fibers automatically has geometrically connected fibers by \StP{04KV} (because the unit section always defines a rational point). Hence, we may assume that $S$ is the spectrum of an algebraically closed field, so that the topological space $|G_i|$ is connected for every $i\in \bbN$. Using that the map $|G|\r |G_i|$, $i\in \bbN$ is surjective and open (being quasi-compact, surjective and flat \StP{02JY}), one checks that $|G|$ is connected.
\xpf

\subsection{Pro-algebraic groups acting on ind-algebraic spaces}\label{sect--pro.action}

Let $G$ be a pro-algebraic group and $X$ an ind-algebraic space over $S$. Then a map of presheaves $a\co G\x_SX\r X$ which satisfies the axioms of an action map is called an action of $G$ on $X$ (over $S$).

\defi \label{defi--adm.act}
The action $a\co G\x_SX\r X$ is called {\it admissible} if there exist presentations $G=\lim_{i\in \bbN}G_i$ and $X=\colim_{j\in J}X_j$ with the following properties:\smallskip\\
i) The presentation $X=\colim_{j}X_j$ is $G$-stable, i.e., for each $j\in J$, the restriction $a|_{G\x_S X_j}$ factors as $G\x_S X_j\overset{a_j}{\r} X_j\subset X$.\smallskip\\
ii) For each $j\in J$, the $G$-action on $X_j$ factors through the algebraic $S$-group $G_{i}$ for some $i>\!\!>0$, i.e., the subgroup $U_i=\ker(G\to G_i)$ operates trivially on $X_j$.
\xdefi

\lemm \label{lemm--adm.strata}
Let $a\co G\x_SX\r X$ be an admissible action for the presentations $G=\lim_{i\in \bbN}G_i$, $X=\colim_{j\in J}X_j$. By taking suitable finite unions of the $X_j$, there exists a $G$-stable presentation $X=\colim_{i\in \bbN}X'_i$ such that the $G$-action on $X_i'$ factors through $G_i$ for every $i\in \bbN$.
\xlemm
\pf For each $i\in \bbN$, let $J_i=\{j\in J\,|\,\text{the $G$-action on $X_j$ factors exactly through $G_i$}\}$. Then the sets $J_i$ are countable, pairwise disjoint, and one has $\bigsqcup_{i\in \bbN}J_i=J$ by Definition \ref{defi--adm.act} ii). Cantor's diagonal argument produces a family of finite subsets $\{J_i'\}_{i\in \bbN}$ of $J$ with the following properties: for each $j\in J_i'$ the $G$-action on $X_j$ factors through $G_i$, one has $J_i'\subset J'_{i'}$ for $i\leq i'$, and $\cup_{i\in \bbN}J_i'=J$. For each $i\in \bbN$, we define the closed subspace $X_i':=\cup_{j\in J_i'}X_j\subset X$, i.e., the scheme-theoretic image (cf.~\StP{082W}) of the quasi-compact map $\bigsqcup_{j\in J_i'}X_j\r X$. Note that we have a presentation $X=\colim_iX_i'$.
To prove that $X_i'$ is $G$-stable, we note that the diagram of ind-algebraic spaces
$$\xymatrix{
G\x_S \left(\bigsqcup_{j\in J_i'}X_j\right) \ar[r] \ar[d]^{\bigsqcup_ja_j} & G\x_SX \ar[d]^a \\
\left(\bigsqcup_{j\in J_i'}X_j\right) \ar[r] & X
}$$
is Cartesian. Since $G$ is $S$-flat and taking the scheme-theoretic image along quasi-compact maps commutes with flat base change (follows from \StP{082Z}), the scheme-theoretic image of the top arrow is $G\x_SX_i'$. This implies that $X_i'$ is $G$-stable.
By construction, the $G$-action on $X_i'$ factors through $G_i$.
\xpf

\lemm\label{lemm--adm.reduced}
Let $a\co G\x_SX\r X$ be an admissible action for the presentations $G=\lim_{i\in \bbN}G_i$, $X=\colim_{j\in J}X_j$. Then the action restricts to an action $a_\red\co G\x_S X_\red\r X_\red$ on the underlying reduced sub-ind-algebraic space
which is admissible for the presentations $G=\lim_{i}G_i$ and $X_\red=\colim_jX_{j,\red}$.
\xlemm
\pf Once we know that $X_{j,\red}\subset X_j$, $j\in J$ is $G$-invariant, the admissibility of the induced action is immediate. We reduce to the case where $X=X_j$ is an algebraic space. We need to show that the reduced subspace $X_\red\subset X$ is $G$-invariant. Our claim follows, by applying the functor $(\str)_\red$, from the following equality of algebraic spaces
\begin{equation}\label{reduced.eq}
G\x_SX_\red\,=\,\left(G\x_SX\right)_\red.
\end{equation}
If $G$ is a smooth $S$-group, then \eqref{reduced.eq} holds true because being reduced is local in the smooth topology, cf.~\StP{034E}. The general case follows from $\calA=\colim_i\calA_i$, in the notation of the proof of Lemma \ref{lemm--pro.group}, and the compatibility of tensor products and colimits using that $\calA_i\to\calA_j$ is universally injective (because faithfully flat).

\xpf

\lemm \label{lemm--adm.finite.type}
Let $X$ be of ind-finite type over a Noetherian scheme $S$, and let $G$ be a pro-algebraic $S$-group. Then every action $a\co G\x_SX\r X$ is admissible.
\xlemm
\pf Let $X=\colim_{j\in J}X_j$ be a presentation by finite type $S$-algebraic spaces, and let $a\co G\x_SX\r X$ be an action. As $G$ is an $S$-affine scheme, the algebraic space $G\x_SX_j$, $j\in J$ is quasi-compact. Hence, the map $a|_{G\x_SX_j}$ factors through $X_{j'}$ for some $j'>\!\!>0$, and we define $X_j'$ as the scheme-theoretic image of $a|_{G\x_SX_j}$. Since $S$ is Noetherian, the closed subspace $X_j'\subset X_{j'}$ is of finite type over $S$, and clearly $G$-stable by construction.
Also $X=\colim_jX_j'$ because $X_j\subset X_j'$.
Now one verifies that every $G$-action on any finite type $S$-algebraic space $X$ factors through $G_i$ for some $i>\!\!>0$.
The lemma follows.
\xpf

\subsection{Torsors under pro-algebraic groups}\label{sect--torsors}
Let $G$ be a pro-algebraic $S$-group. By Lemma \ref{lemm--pro.group}, the map $G\to S$ is faithfully flat and affine (hence quasi-compact). Let $P\to S$ be a right $G$-torsor in the fpqc topology on $S$\nts{as defined in \cite[Tag 03AH]{StacksProject}}. By fpqc descent for affine morphisms \StP{0245}, the map $P\to S$ is also faithfully flat and affine (hence $P$ is a scheme). We denote the set\footnote{Lemma \ref{lemm--torsor.sequence} below implies that $\H^1_\fpqc(S,G)$ is indeed a set (use the twisting trick).} of isomorphism classes of right $G$-torsors in the fpqc (resp.~\'etale) topology on $S$ by $\H^1_\fpqc(S,G)$ (resp.~$\H^1_\et(S,G)$).

Our aim is to show that $\H^1_{\on{fpqc}}(S,G)=\H^1_{\text{\'et}}(S,G)$ under suitable conditions on $G$ (cf.~Corollary \ref{coro--etale.torsor}), and to show that all examples we have in mind satisfy this condition, cf.~Proposition \ref{prop--vector.extension} and Example \ref{exam--groups}. This generalizes the \'etale-local triviality of the torsors considered in \cite[Thm 1.4]{PappasRapoport:LoopGroups} and \cite[Lem 3.4 ii)]{HainesRicharz:TestFunctionsWeil} for example.

Given a presentation $G=\lim_{i\in \bbN} G_i$, $P_i:=P\x^GG_i$ is a $G_i$-torsor on $S$. For each $i\in \bbN$, the transition map $G_{i+1}\to G_i$ induces an identification $P_{i+1}\x^{G_{i+1}}G_i=P_i$. This gives a natural map
\begin{equation}\label{torsor.lim}
\H^1_{\fpqc}(S,G)\to \lim \H^1_\et(S,G_i), \;\;\; P\mapsto \{P_i\}_{i\in \bbN}.
\end{equation}
Here we use that $\H^1_{\text{\'et}}(S,G_i)\subset \H^1_{\on{fpqc}}(S,G_i)$ is a bijection: as $G_i\to S$ is smooth, any fpqc-$G_i$-torsor is also smooth, and hence admits sections \'etale-locally. As in \cite[Ch IX, \S2]{BousfieldKan} consider the group $\bigsqcap_{i\in\bbN} \H^0(S,G_i)$ acting on the set $\bigsqcap_{i\in\bbN} \H^0(S,G_i)$ via the formula
\[
(g_0,g_1,g_2,\ldots)*(x_0,x_1,x_2,\ldots):= (g_0x_0 g_1^{-1},\; g_1x_1 g_2^{-1}, \; g_2x_2 g_3^{-1},\ldots),
\]
where $G_{i+1}$ acts on $G_i$ via the transition map $G_{i+1}\to G_i$. We denote by $\on{lim}^1\H^0(S,G_i):=\bigsqcap \H^0(S,G_i)/\sim $ its set of equivalence classes. There is a natural map
\begin{equation}\label{torsor.lim1}
\bigsqcap \H^0(S,G_i)\to \H^1_{\on{fpqc}}(S,G),\;\;\; (x_i)_{i\in \bbN}\mapsto \on{lim}_{x_i}G_i,
\end{equation}
where $G_i$ is considered as the trivial $G_i$-torsor via right multiplication with transition maps $G_{i+1}\to G_i$, $a\mapsto x_{i}\cdot a$. Then $\{G_i, x_i\}_{i\in \bbN}$ forms an inverse system, and its limit $\on{lim}_{x_i}G_i\to S$ defines a right $G$-torsor (prove that $\on{lim}_{x_i}G_i\to S$ is faithfully flat and affine similarly to the proof of Lemma \ref{lemm--pro.group}, and further that the right $G$-action is simply transitive).


\lemm \label{lemm--torsor.sequence}
The maps \eqref{torsor.lim} and \eqref{torsor.lim1} induce an exact sequence of pointed sets
\[
1\to \on{lim}^1\H^0(S,G_i)\to \H^1_{\on{fpqc}}(S,G)\to \lim \H^1_{\textup{\'et}}(S,G_i)\to 1.
\]
\xlemm
\pf
The map $\on{lim}^1\H^0(S,G_i)\to \H^1(S,G)$ is well-defined. Indeed, if $x,y\in \bigsqcap \H^0(S,G_i)$ with $y=g*x$ for some $g\in \bigsqcap \H^0(S,G_i)$, then there is an isomorphism of inverse systems $\{G_i,x_i\}_{i\in\bbN}\to \{G_i,y_i\}_{i\in\bbN}$ induced from the maps $G_i\to G_i$, $a\mapsto g_i\cdot a$. The exactness properties are elementary to check, and left to the reader. Note that this can also be regarded as an example of a Milnor type exact sequence \cite[Prop.~VI.2.15]{GoerssJardine:Simplicial}.

\xpf

Let $\calE$ be a $S$-vector bundle, i.e., a locally free $\calO_S$-module of finite rank.
Using that $\calE$ is quasi-coherent and reflexive (i.e., $\calE=(\calE^\vee)^\vee$), one shows that the group-valued functor on the category of $S$-schemes $T$ given by $T\mapsto \calE(T)$ is representable by the $S$-group scheme
$$\bbV(\calE):=\underline{\Spec}_{\calO_S}(\on{Sym}^\otimes(\calE^\vee))\to S,\eqlabel{VE}$$
cf.~\cite[\S1.7]{EGA2}.
Note that our notion is dual to the reference.
We also write $\bbV^\x(\calE) :=\bbV(\calE) \setminus S$ for the complement of the zero section.
The $S$-group $\bbV(\calE)$ is algebraic, unipotent, commutative, and Zariski locally on $S$ isomorphic to $\bbG_{a,S}^r$, $r:=\on{rank}(\calE)$. A \textit{vector group} is an $S$-group isomorphic to $\bbV(\calE)$ for some $S$-vector bundle $\calE$.

\defi\label{defi--unipotent}
A pro-algebraic $S$-group $G$ is called \textit{split pro-unipotent} if it admits a presentation $G= \lim_{i \in \bbN} G_i$ such that the group $G_0$ and all groups $\ker(G_{i+1}\to G_i)$, $i\in \bbN$ are vector groups.
In other words, a split pro-unipotent group is an (possibly infinite) successive extension of a vector group by vector groups.
\xdefi

\prop \label{prop--unipotent}
If $S$ is affine and $G$ is split pro-unipotent, then $\H^1_{\on{fpqc}}(S,G)$ is trivial.
\xprop
\pf
We have $\H^1_\et(S, G_i) = 1$ for all $i$.
Indeed, by induction on $i$ and the standard $6$-term exact sequence for non-abelian group cohomology \cite[Ch III, Prop.~3.3.1]{Giraud:Cohomologie}, this reduces to $\H^1_\et(S,\bbV(\calE))=\H^1_\et(S,\calE)=\H^1_\Zar(S,\calE)$ \StP{03P2} being trivial which holds because $S$ is affine.
This argument also shows that the maps $\H^0(S, G_{i+1}) \r \H^0(S, G_i)$ are surjective, so $\lim^1 \H^0(S, G_i)=1$ by \cite[Ch IX, \S2, Prop.~2.4]{BousfieldKan}.
Thus, the proposition follows from Lemma \ref{lemm--torsor.sequence}.
\xpf

\exam Proposition \ref{prop--unipotent} is false for general pro-algebraic groups: Let $p\in \bbZ$ be a prime, and let $G=\on{Gal}(\bbQ_p(\zeta_{p^\infty})/\bbQ_p)\simeq \bbZ_p$ considered as a pro-algebraic group. Then the $G$-torsor $\Spec(\bbQ_p(\zeta_{p^\infty}))\to \Spec(\bbQ_p)$ has no sections \'etale-locally.
\xexam

\coro\label{coro--etale.torsor}
If there exists an $i_0\in \bbN$ such that $\ker(G\to G_{i_0})$ is split pro-unipotent, then the natural map $\H^1_{\textup{\'et}}(S,G)\to \H^1_{\on{fpqc}}(S,G)$ is bijective.
\xcoro
\pf
The map is clearly injective, and we have to show that every $G$-torsor $P\to S$ for the fpqc topology admits sections \'etale-locally.
For this we may assume $S$ is affine.
Let $U:=\ker(G\to G_{i_0})$, and consider the factorization $P \stackrel a \to P/U=P\x^GG_{i_0} \stackrel b \to S$. The map $b$ 
is a $G_{i_0}$-torsor, and hence admits sections \'etale-locally.
The map $a$ is a trivial $U$-torsor by \refpr{unipotent}, since $S$ and therefore $P/U$ is affine as well.
\xpf

We end this section by proving that pro-algebraic groups which are defined as positive loop groups (or sometimes called jet groups) satisfy the assumption of Corollary \ref{coro--etale.torsor} with $i_0=0$. All examples of these pro-algebraic groups which we encounter in the main body of this manuscript fall under the following general set-up, cf.~Example \ref{exam--groups} below.

Let $X\to S$ be smooth and pure of relative dimension $1$. Let $D\subset X$ be an effective Cartier divisor which is finite and locally free over $S$. Let $\calI_D\subset \calO_X$ be the ideal sheaf defined by $D$. For $i\geq 0$, the subscheme $D_i\subset X$ defined by $\calI_D^{i+1}$ is again finite and locally free over $S$. Let $\hat{D}=\colim_{i\in \bbN}D_i$ considered as an ind-scheme. Let $\hat{\calG}\to \hat{D}$ be a group functor which is relatively representable\footnote{We do not want to require that $\hat{\calG}$ spreads to a group scheme over $X$. The weaker assumption suffices for our purposes.} by a smooth affine group scheme of finite presentation. For each $i\geq 0$, let $\calG_i:=\hat{\calG}\x_{\hat{D}}D_i$. We consider the strictly pro-algebraic $S$-group (cf.~proof of Proposition \ref{prop--vector.extension})
\[
G\defined \lim_{i\geq 0}G_i,
\]
where $G_i:=\on{Res}_{D_i/S}(\calG_i)$ denotes the Weil restriction of scalars, and $G_i\to G_{i-1}$ are the obvious transition maps. Further, we denote $U:=\ker(G\to G_0)$.

\prop \label{prop--vector.extension}
The pro-algebraic group $U$ is split pro-unipotent. More precisely, for each $i\geq 1$, there is a short exact sequence of algebraic $S$-group schemes
\[
0\to \bbV(\calE_i)\to G_i\to G_{i-1}\to 0,
\]
for an explicit $S$-vector bundle $\calE_i$ of rank $\on{rk}(D)\cdot \dim(\hat{\calG}/\hat{D})$, viewed as a locally constant function on $S$. In particular, by \refco{etale.torsor}, we have $\H^1_{\on{fpqc}}(S,G)=\H^1_{\textup{\'et}}(S,G)$.
\xprop
\pf
Since $D_i\to S$ is finite locally free, the Weil restriction of scalars $G_i$ is representable by an algebraic $S$-group scheme, cf.~\cite[\S7.6, Thm.~4, Prop.~5]{BLR}. The canonical map $G_i\to G_{i-1}$ is locally of finite presentation (because limit preserving \StP{01ZC}), formally smooth, and thus is a smooth map, cf.~\StP{02H6}. As $G_i\to G_{i-1}$ is also surjective, it follows that $G_i\to G_{i-1}$ is a surjection of \'etale sheaves. It remains to identify the kernel $\ker(G_i\to G_{i-1})$ as a vector group.

We consider the following general set-up. Let $S'$ be a base scheme, and let $Y, Z$ be $S'$-schemes. Let $Y_0\subset Y$ be a closed subscheme defined by a sheaf of ideals $\calJ$ with $\calJ^2=0$. Let $g_0\co Y_0\to Z$ be a map of $S'$-schemes. If $Z\to S'$ is smooth, then \cite[III.5, Cor.~5.2]{SGA1} implies that for all $T\to S'$ we have a functorial identification
\begin{equation}\label{deformation}
\{g\in \Hom_{T}(Y_T,Z_T)\;|\;g|_{Y_{0,T}}=g_{0,T} \}\,=\, \left(g_0^*\mathfrak g_{Z/S}\otimes_{\calO_{Y_0}}\calJ \right)(Y_{0,T})
\end{equation}
where $\mathfrak g_{Z/S}:=(\Omega_{Z/S}^1)^*$. We apply this as follows.

Let $S'=D_i$, and set $Z:=\calG_i$. Let $Y_0:=D_{i-1}\subset D_i=:Y$, i.e., $\calJ=\calI_{D}^i/\calI_{D}^{i+1}$. Let $g_0\co D_{i-1}\to \calG_i$ be given by the inclusion $D_{i-1}\subset D_i$ composed with the identity section $1\co D_i\to \calG_i$. Also let $\pi_i\co D_i\to S$ be structure map. Taking $\pi_{i,*}$ of the left hand side in \eqref{deformation}, i.e., restricting the functor to the category of $S$-schemes, is by definition equal to the functor $\ker(G_i\to G_{i-1})$. We define
\begin{equation}\label{vector_bundle}
\calE_i\defined \pi_{i,*} \left(\on{Lie}^{(i)}(\hat{\calG})\otimes_{\calO_{D_{i-1}}} ( \calI_{D}^i/\calI_{D}^{i+1})\right),
\end{equation}
where $\on{Lie}^{(i)}(\hat{\calG}):= 1^*(\Omega_{\calG_{i-1}/D_{i-1}}^1)^*$. Note that $\on{Lie}^{(i)}(\hat{\calG})$ is a locally free $\calO_{D_{i-1}}$-module of rank $\dim(\hat{\calG}/\hat{D})$ because $\hat{\calG}\to\hat{D}$ is smooth. Since $\calI_{D}^i/\calI_{D}^{i+1}$ is locally free of rank $\on{rk}(D)$ when considered as an $\calO_S$-module, we see that $\calE$ is a locally free $\calO_S$-module of rank $\on{rk}(D)\cdot \dim(\hat{\calG}/\hat{D})$. Further, the argument above shows that $\bbV(\calE_i)=\ker(G_i\to G_{i-1})$ as $S$-schemes, and a calculation using $\calJ^2=0$ shows that the identification is compatible with the group structure.
\xpf

Here is a list of examples which are of interest to us.

\exam\label{exam--groups}
i) {\em Loop groups.} Let $S=\Spec(k)$ for some field $k$, $X=\bbA^1_k$, and $D=\{0\}$. Then $\hat{D}=\on{Spf}(k\pot{\varpi})$ where $\varpi$ is a local parameter at $0$, and we let $\hat{\calG}=H\x_S\on{Spf}(k\pot{\varpi})$ for a smooth affine finite type $k$-group $H$. In this case, $G=L^+H$ which is the functor on the category of $k$-algebras $R$ given by $L^+H\co R\mapsto H(R\pot{\varpi})$.\smallskip\\
ii) {\em Fusion loop groups.} Let $Y$ be a smooth curve over the field $k$. For any finite index set $I$, let $S:=Y^I$, and $X:=S\times_kY\to S$ the projection. Let $D\subset X$ be the universal degree $\#I$ divisor. Let $H$ be a smooth affine finite type $k$-group, and let $\hat{\calG}:=H\x_k\hat{D}$. Then $G=L^+_{I}H\to Y^I$ is the fusion loop group from \eqref{glob_loop_group} introduced in \cite{BeilinsonDrinfeld:Quantization}. We make the case $\#I=2$ more explicit. Let $\Delta\subset Y^2$ be the diagonal with complement $U=Y^2\backslash \Delta$. Then $L^+_{I}H|_U=(L^+H\x_k L^+H)\x_kU$ whereas $L^+_{I}H|_\Delta=L^+H\x_k\Delta$. We observe that the group $G_0=\on{Res}_{D/Y^2}(H\x_kD)\to Y^2$ is not reductive. Indeed, $G_0|_U=H\x_kH$, but $G_0|_{\Delta}=H\ltimes \on{Lie}^{(1)}(H)$ because $D\to Y^2$ is ramified along $\Delta$ (e.g.,~take $Y=\bbA^1_k$, then $D=\{f=0\}$ for $f=(\varpi-x_1)(\varpi-x_2)$ where $x_1,x_2$ are the coordinates on $S=\bbA^2_k$, and $\varpi$ is the coordinate on $\bbA^1_k$ in $X=\bbA^2_k\x_k\bbA^1_k$). This is in accordance with $\ker(G_1\to G_0)$ being a vector group. \smallskip\\
iii) {\em Non-constant group schemes.} Let $S=\Spec(R)$ for some ring $R$. Let $X=\bbA^1_R$, and let $D=\{f=0\}$ for some polynomial $f\in R[\varpi]$. \smallskip\\
iii.a) Specialize to $f=\varpi$, and $\calG$ a smooth affine $R\pot{\varpi}$-group scheme. Define $\hat{\calG}:=\calG\x_{\Spec(R\pot{\varpi})}\on{Spf}(R\pot{\varpi})$. In the case $R=k$ is a field, the group $G$ is the twisted positive loop group in the sense of \cite{PappasRapoport:LoopGroups}. \smallskip\\
iii.b) Specialize to $R=\bbZ_p$ for $p$ a prime number, $f=\varpi-p$ (resp.~any Eisenstein polynomial). Let $\calG$ be a smooth affine $\bbA^1_{\bbZ_p}$-group scheme, e.g., one of the group schemes constructed in Pappas-Zhu \cite[\S 4]{PappasZhu:Kottwitz} (resp.~Levin \cite[\S 3]{Levin:Weil}). Then $G$ over $\bbZ_p$ is the positive loop group constructed in \cite[(6.4), 6.2.6]{PappasZhu:Kottwitz} (resp.~\cite[Prop 4.1.4 ff]{Levin:Weil}), cf.~also \cite[Ex 3.1 ii)]{HainesRicharz:TestFunctionsWeil}.
\xexam

\bibliographystyle{alpha}
\bibliography{bib}

\end{document}